\documentclass[a4paper,10pt,reqno]{amsart}
\usepackage[english]{babel}
\usepackage{amsmath,amsthm,amssymb,amsfonts}
\usepackage{bm}
\usepackage{hyperref}
\usepackage{mathrsfs}
\usepackage{mathtools}

\usepackage[usenames,dvipsnames]{xcolor}
\usepackage{thmtools}
\usepackage{enumitem}
\setlist[enumerate]{leftmargin=*}

\allowdisplaybreaks

\theoremstyle{plain}
\newtheorem{proposition}{Proposition}[section]
\newtheorem{theorem}[proposition]{Theorem}
\newtheorem{lemma}[proposition]{Lemma}
\newtheorem{corollary}[proposition]{Corollary}

\theoremstyle{definition}
\newtheorem{definition}[proposition]{Definition}
\declaretheorem[name=Remark,sibling=proposition,qed={\lower-0.3ex\hbox{$\blacklozenge$}}]{remark}

\numberwithin{equation}{section}

\newcommand{\N}{\mathbb{N}}	
\newcommand{\Z}{\mathbb{Z}}
\newcommand{\C}{\mathbb{C}}	
\newcommand{\R}{\mathbb{R}}	
\newcommand{\Sphere}{\mathbb{S}}       % sphere

\newcommand{\nR}{\dot\R}    % without zero
\newcommand{\nT}{\dot T}

\newcommand{\Id}{\mathrm{Id}}	% identity map

\DeclareMathOperator{\rank}{rank}

\DeclareMathOperator{\diag}{diag}
\DeclareMathOperator{\Span}{span}	% Span

\renewcommand{\Im}{\operatorname{Im}}

\newcommand{\dd}{\mathrm{d}}	% 'd' of derivative
\newcommand{\de}{\partial}		% partial derivative
\renewcommand{\div}{\operatorname{div}}	% divergence

\newcommand{\DD}{\mathrm{D}}     % differential of a map

\newcommand{\mm}{\underline{m}}
\newcommand{\cc}{\underline{c}}

\DeclareMathOperator{\spt}{supp} % support
\DeclareMathOperator{\singspt}{sing\,supp}
\newcommand{\WF}{\mathtt{WF}}
\DeclareMathOperator{\essspt}{ess\,supp}

\newcommand{\Exp}{\mathtt{Exp}}  % geodesic exponential map
\newcommand{\X}{\mathcal{X}}     % hamiltonian field

\newcommand{\dist}{\mathrm{dist}}

\newcommand{\tc}{\,:\,}  %  such that
\newcommand{\defeq}{\mathrel{:=}}

\newcommand{\interior}{\mathop\mathtt{int}}

\newcommand{\CO}{\mathcal{C}}    % cutoff functions

\newcommand{\sLap}{\mathscr{L}}        % sub-Laplacian

\newcommand{\Dist}{\mathscr{D}'}       % distributions
\newcommand{\EDist}{\mathscr{E}'}      % distributions with compact support
\newcommand{\Sch}{\mathscr{S}}         % Schwartz class

\newcommand{\cH}{\mathscr{H}}          % horizontal vector fields

\newcommand{\taut}{\mathfrak{t}}       % tautological form

\newcommand{\opA}{\mathfrak{a}}

\newcommand{\cU}{\mathscr{U}}
\newcommand{\cV}{\mathscr{V}}
\newcommand{\cW}{\mathscr{W}}

\newcommand{\cD}{\mathcal{D}}    % domain (of Hamiltonian)

\newcommand{\sloc}{\mathrm{sloc}}

\newcommand{\cl}{\mathrm{cl}}     % classical (symbols, PseuDO)

\newcommand{\thr}{\varsigma}   % MH threshold

\newcommand{\Lie}[1]{\mathfrak{#1}}    % Lie algebra

\newcommand{\Sect}{\mathbf{\Gamma}}  % smooth sections of a bundle
\newcommand{\Forms}{\mathbf{\Omega}} % smooth forms on a manifolds
\newcommand{\CLF}{\mathscr{C}\!\ell}   % closed forms

\newcommand{\Ev}{\mathtt{Ev}}   % evaluation map

\newcommand{\wphase}{{\bm{w}}}
\newcommand{\phase}{{\bm{\phi}}} 

\newcommand{\Rop}{\mathscr{R}}     % regular operators

\newcommand{\la}{\lambda}

\title[Multipliers and wave equation for sub-Laplacians]{Spectral multipliers and wave equation for sub-Laplacians: lower regularity bounds of Euclidean type}

\author[Martini]{Alessio Martini}
\address[A. Martini]{School of Mathematics \\ University of Birmingham \\ Edgbaston \\ Birmingham \\ B15 2TT \\ United Kingdom}
\email{a.martini@bham.ac.uk}

\author[M\"uller]{Detlef M\"uller}
\address[D. M\"uller]{Mathematisches Seminar \\ Christian-Albrechts-Universit\"at zu Kiel \\ Ludewig-Meyn-Str.\ 4 \\ D-24118 Kiel \\ Germany}
\email{mueller@math.uni-kiel.de}

\author[Nicolussi Golo]{Sebastiano Nicolussi Golo}
\address[S. Nicolussi Golo]{Dipartimento di Matematica \\ Universit\`a di Padova \\ Via Trieste, 63 \\ I-35121 Padova \\ Italy}
\email{sebastiano2.72@gmail.com}


\thanks{This research was partially supported by the EPSRC Grant ``Sub-Elliptic Harmonic Analysis'' (EP/P002447/1). Part of the work was carried out during a two-month visit of the first-named author to the Christian-Albrechts-Universit\"at zu Kiel (Germany), made possible by the generous financial support of the Alexander von Humboldt Foundation.}

\dedicatory{In memory of Eli Stein.}

\begin{document}

\begin{abstract}
Let $\sLap$ be a smooth second-order real differential operator in divergence form on a manifold of dimension $n$. Under a bracket-generating condition, we show that the ranges of validity of spectral multiplier estimates of Mihlin--H\"ormander type and wave propagator estimates of Miyachi--Peral type for $\sLap$ cannot be wider than the corresponding ranges for the Laplace operator on $\R^n$. The result applies to all sub-Laplacians on Carnot groups and more general sub-Riemannian manifolds, without restrictions on the step. The proof hinges on a Fourier integral representation for the wave propagator associated with $\sLap$ and nondegeneracy properties of the sub-Riemannian geodesic flow.
\end{abstract}

\maketitle

\section{Introduction}

Let $M$ be a smooth manifold, $H : T^*M\to[0,+\infty)$ a smooth function on the cotangent bundle that is a positive-semidefinite quadratic form on each fiber, and $\mu$ a smooth positive measure on $M$.
The \emph{sub-Laplacian} $\sLap$ defined by $(M,H,\mu)$ is the second-order differential operator given by
\[
\sLap f = - \div_\mu(B_H(\dd f))
\qquad \forall f\in C^\infty_c(M) ,
\]
where $B_H : T^*M\to TM$ is the linear map determined by the quadratic form $H$, and $\div_\mu$ is the divergence operator defined by $\mu$ (see Definition \ref{dfn:divergence} below).
The sub-Laplacian $\sLap$ is a non-negative symmetric unbounded operator on $L^2(M) \defeq L^2(M,\mu)$, and it has principal symbol $H$.

The above definition encompasses a number of second-order differential operators considered in the literature. In particular, if $H$ is a positive-definite quadratic form, then it is the cometric of a Riemannian tensor on $M$, and $\sLap$ is elliptic; moreover, if $\mu$ is the Riemannian volume, then $\sLap$ is the Laplace--Beltrami operator. More generally, if there is a bracket-generating family of vector fields $v_1,\dots,v_r\in\Sect(TM)$ such that $H=\sum_j v_j\otimes v_j$, then $H$ is the cometric of a sub-Riemannian structure and $\sLap$ is a sub-Laplacian as defined, e.g., in \cite{MR1867362}.

Assume that a self-adjoint extension of $\sLap$ has been chosen.
Then a functional calculus for $\sLap$ is defined via the spectral theorem and, for all bounded Borel functions $m : [0,+\infty) \to \C$, the operator
\[
m(\sLap)=\int_{[0,\infty)} m(s) \,\dd E_\sLap(s)
\]
is bounded on $L^2(M)$. An extensively studied problem in the literature is the determination of necessary conditions and sufficient conditions on the function $m$, also known as a spectral multiplier, for $m(\sLap)$ to extend to a bounded operator on $L^p(M)$ for some $p \neq 2$.

In the case where $\sLap$ is the Laplace operator on $\R^n$, the $L^p$ boundedness of $m(\sLap)$ can be ensured by suitable size and smoothness conditions on $m$. More specifically, for $m : [0,\infty)\to \C$, $q \in [1,\infty]$ and $\alpha \geq 0$, let us define the local scale-invariant $L^q$ Sobolev norm of order $\alpha$ of $m$ by
\[
\|m\|_{L^q_{\alpha,\sloc}} = \sup_{t \geq 0} \|\rho \, m(t\cdot)\|_{L^q_\alpha(\R)} ,
\]
where $L^q_\alpha(\R)$ is the $L^q$ Sobolev space of order $\alpha$, and $\rho\in C^\infty_c((0,\infty))$ is a nontrivial cutoff (different choices of $\rho$ give rise to equivalent norms). The classical Mihlin--H\"ormander multiplier theorem \cite{MR0080799,MR0121655} implies that
\begin{equation}\label{eq:MH}
\| m(\sLap) \|_{p \to p} \lesssim_{p,\alpha} \| m \|_{L^2_{\alpha,\sloc}}
\end{equation}
for all $p \in (1,\infty)$ and $\alpha > n/2$ (at the endpoint $p=1$, weak type $(1,1)$ and $H^1 \to L^1$ boundedness hold). Clearly one can replace the $L^2_{\alpha,\sloc}$ norm with the stronger $L^\infty_{\alpha,\sloc}$ norm in the right-hand side, and actually interpolation yields
\begin{equation}\label{eq:MHinterp}
\| m(\sLap) \|_{p \to p} \lesssim_{p,\alpha} \| m \|_{L^\infty_{\alpha,\sloc}}
\end{equation}
for all $p \in (1,\infty)$ and $\alpha > n|1/2-1/p|$.

Related to the above are $L^p$ estimates for oscillatory multipliers, and especially the Miyachi--Peral estimates for the wave propagator \cite{MR586454,MR568979}:
\begin{equation}\label{eq:miyachiperal}
\| (1+t^2\sLap)^{-\alpha/2} \cos(t \sqrt{\sLap}) \|_{p \to p} \lesssim_{p,\alpha} 1,
\end{equation}
uniformly in $t>0$, for $p \in [1,\infty]$ and $\alpha \geq (n-1)|1/p-1/2|$ (except for $p=1,\infty$ and $\alpha = (n-1)/2$, in which case a Hardy space boundedness result holds).
A spectrally localised version of the above estimate reads as follows:
\[
\| \chi(t\sqrt{\sLap}/\la) \, \cos(t \sqrt{\sLap}) \|_{p \to p} \lesssim_{p,\alpha} (1+\la)^\alpha
\]
uniformly in $t,\la > 0$, where $\chi \in C^\infty_c((0,\infty))$ is a nontrivial cutoff.

It is natural to investigate whether these results for the Euclidean Laplacian extend to more general manifolds $M$ and operators $\sLap$. As a matter of fact, in the case of elliptic operators $\sLap$ on compact manifolds $M$, both Mihlin--H\"ormander and Miyachi--Peral estimates are available \cite{MR1046745,MR1127475}, for the same range of indices, where $n$ is the dimension of the manifold $M$; a key ingredient in the proof of these results is the representation of the wave propagator $\cos(t\sqrt{\sLap})$ as sum of Fourier integral operators. The case of noncompact manifolds is much more delicate, in that the ranges of validity (if any) of the above estimates depend on the global geometry of $(M,H,\mu)$ and not only on the (local) dimension $n$  (see, e.g., \cite{MR0367561,MR1415763,MR2291727,MR3092733,kania_bessel_2018} and references therein); in addition, the available results are not as robust as in the compact case, especially if one is interested in sharp results. In any case, via transplantation \cite{MR654463} one immediately sees that, for an elliptic operator $\sLap$ on an $n$-dimensional manifold $M$ the ranges of validity of the above estimates cannot be larger than those for the Laplace operator on $\R^n$. We note that the aforementioned results for the Euclidean Laplace operator are sharp up to the endpoints; in particular, if we define the sharp Mihlin--H\"ormander threshold $\thr(\sLap)$ for a sub-Laplacian $\sLap$ as the infimum of the $\alpha \geq 0$ such that
\[
\forall p\in (1,\infty) \tc \exists C \in (0,\infty) \tc \forall m \in \mathcal{B} \tc \|m(\sLap)\|_{L^p\to L^p} \le C \|m\|_{L^2_{\alpha,\sloc}},
\]
where $\mathcal{B}$ is the set of bounded Borel functions $m : [0,\infty) \to \C$,
then $\thr(\sLap) = n/2$ for the Laplace operator $\sLap$ on $\R^n$ (see, e.g., \cite{MR1814106}).

Determining the optimal ranges of validity becomes even more difficult when one weakens the ellipticity assumption on $\sLap$. For instance, if $\sLap$ is a homogeneous sub-Laplacian on a Carnot (stratified) group, then a multiplier theorem of Mihlin--H\"ormander type for $\sLap$ is known \cite{MR1125759,MR1104196}, implying that $\thr(\sLap) \leq Q/2$, where $Q$ is the homogeneous dimension of the group; note that $Q$ is strictly larger than the topological dimension $n$ when the group has step $2$ or higher, i.e., when $\sLap$ is not elliptic. Similar results are actually known in greater generality (e.g., in the presence of suitable volume growth and heat kernel estimates, see \cite{MR1172944,hebisch_functional_1995,MR1943098}), involving a dimensional parameter $Q$ that is stricly larger than the topological dimension $n$ in the case $\sLap$ is not elliptic (cf.\ \cite{MR730094}). Despite the naturality of the dimensional parameter $Q$ in this context, these results turn out  not to be sharp in many cases.

This discovery was first made in the case of homogeneous sub-Laplacians on Heisenberg groups \cite{MR1240169,MR1290494}, for which it was proved that $\thr(\sLap) = n/2$. A number of results in this direction have been obtained since then, and we now know that $n/2 \leq \thr(\sLap) < Q/2$ for homogeneous sub-Laplacians on all $2$-step Carnot groups \cite{MR3513881}, and that actually the equality $\thr(\sLap) = n/2$ holds in a number of cases \cite{MR3106049,MR3357697,MR3283616}, also for more general manifolds and sub-Laplacians \cite{MR1860734,MR2728580,MR3039831,MR3293433,ahrens_quaternionic_2016,MR3708016,casarino_grushin_2017,dallara_grushin_2017}. Moreover, in the case of groups of Heisenberg type, sharp estimates of Miyachi--Peral type are also available \cite{MR1715410,MR3393673}, proving the validity of \eqref{eq:miyachiperal} for the same range of indices mentioned above for $\R^n$ (where $n$ is the topological dimension of the group); note that these results imply, by subordination (cf.\ \cite{MR1648116}), the sharp multiplier theorem of Mihlin--H\"ormander type in this context. Nevertheless the determination of the optimal ranges of validity of \eqref{eq:MHinterp} and \eqref{eq:miyachiperal} in general remains a widely open problem. In particular, the proofs of the lower bound $\thr(\sLap) \geq n/2$ given in \cite{MR1290494,MR3513881} crucially exploit the structure of $2$-step groups (more specifically, the existence of an explicit formula of Mehler type for the Schr\"odinger propagator) and do not seem to be easily extendable to the higher step case.

At this stage it is relevant to remark that, when $\sLap$ is not elliptic, the lower bound $\thr(\sLap) \geq n/2$ cannot be just obtained by comparison to the Euclidean situation via transplantation, as in the elliptic case. Indeed, the methods of \cite{MR654463} allow one to compare the operator $\sLap$ on $M$ with the ``local model operator'' $\sLap_o$ at any point $o \in M$, defined as the principal part of the constant-coefficient operator on the tangent space $T_o M$ obtained by ``freezing the coefficients'' of $\sLap$ at $o$.
If $H$ is not positive-definite at the point $o \in M$, then the local model $\sLap_o$ is a ``partial Laplacian'' corresponding to a proper subspace of $T_o M$, namely, the space
\[
\cH_o = (\{ H = 0 \} \cap T_o^* M)^\perp
\]
of ``horizontal vectors'' for $H$ at $o$, and therefore the lower bounds to $\thr(\sLap)$ obtained in this way would involve $\dim \cH_o$ in place of $n$.

It is clear from the above discussion that, in order to obtain lower bounds to $\thr(\sLap)$ in terms of the topological dimension $n$, additional assumptions on $H$ are necessary, ruling out the case where $\sLap$ actually ``lives'' on submanifolds of lower dimension that foliate $M$. In view of the Frobenius theorem, a natural condition in this context is the ``bracket-generating condition'' on $H$, that can be stated as follows. Let $\cH$ denote the set of (smooth) horizontal vector fields for $H$, and define recursively $\cH^{(k)}$ for $k \in \N \setminus \{0\}$ by 
\[
\cH^{(1)} = \cH, \qquad \cH^{(k+1)} = \cH^{(k)} + [\cH,\cH^{(k)}].
\]
Finally, for all $x \in M$, we define $\cH^{(k)}_x$ as the set of values $v|_x$ of vector fields $v \in \cH^{(k)}$. Then $H$ is said to be bracket-generating at the point $x \in M$ (of step $k$) if $\cH^{(k)}_x = T_x M$ for some $k \geq 1$. Note that, when $H = \sum_j v_j \otimes v_j$, the usual bracket-generating condition on the family of vector fields $\{v_j\}_j$ implies that $H$ is bracket-generating at each point of $M$; in particular, homogeneous sub-Laplacians on Carnot groups and more general sub-Laplacians on sub-Riemannian manifolds satisfy the condition. Recall that a celebrated result of H\"ormander \cite{MR0222474} relates the bracket-generating condition to the hypoellipticity of $\sLap$, while Chow's theorem \cite{MR0001880} relates it to connectivity via horizontal curves.

Our main result shows that, under the bracket-generating condition, the ranges of validity of \eqref{eq:MHinterp} and \eqref{eq:miyachiperal} for a sub-Laplacian $\sLap$ on an $n$-dimensional manifold are indeed not wider than those for the Euclidean Laplacian on $\R^n$.

\begin{theorem}\label{thm:main}
Let $M$ be a smooth manifold of dimension $n$, $H:T^*M\to[0,+\infty)$ a smooth function that is a positive semidefinite quadratic form on each fiber, and $\mu$ a smooth positive measure on $M$.
Let $\sLap$ be the sub-Laplacian defined by $(M,H,\mu)$ and let us fix a self-adjoint extension of $\sLap$.
If $H$ is bracket-generating at some point of $M$, then the following hold true.
\begin{enumerate}[label=(\roman*)]	
\item\label{en:main_MH} If $p \in [1,\infty]$ and $\alpha \geq 0$ are such that the estimate
\begin{equation}\label{eq:main_MH}
\| m(\sLap) \|_{L^p(M) \to L^p(M)} \lesssim \| m \|_{L^\infty_{\sloc,\alpha}}
\end{equation}
holds for all bounded Borel functions $m : [0,\infty) \to \C$, then
\[
\alpha \geq n |1/2-1/p|.
\]
In particular,
\[
\thr(\sLap) \geq n/2.
\]
\item\label{en:main_MP} If $p \in [1,\infty]$ and $\alpha \geq 0$ are such that, for some nontrivial $\chi \in C^\infty_c((0,\infty))$ and some $\epsilon,R>0$, the estimate
\begin{equation}\label{eq:main_MP}
\| \chi(t\sqrt{\sLap}/\la) \, \cos(t\sqrt{\sLap}) \|_{L^p(M) \to L^p(M)} \lesssim \la^\alpha
\end{equation}
holds for all $\la,t > 0$ such that $t \leq \epsilon$ and $\la \geq R$, then
\[
\alpha \geq (n-1) |1/2-1/p|.
\]
\end{enumerate}
\end{theorem}

Part \ref{en:main_MH} of Theorem \ref{thm:main} extends the results of \cite{MR3513881}, that apply only to $2$-step structures, to the case of arbitrary step, while part \ref{en:main_MP} appears to be new even in the $2$-step case. In addition, the method of proof is substantially different and more robust, in that it does not rely on special properties of $2$-step structures, and is based on a Fourier integral representation of the wave propagator $\cos(t \sqrt{\sLap})$.

In order to describe some ideas from the proof, let us first consider the case of the Laplace operator $\sLap$ on $\R^n$. Here via the Fourier transform one can write
\[
\cos(t\sqrt{\sLap}) u(x) = \frac{1}{2} \sum_{\varepsilon = \pm 1} \frac{1}{(2\pi)^n} \int \int e^{i (\xi \cdot (x-y) + \varepsilon t|\xi|) } u(y) \,\dd y \,\dd\xi,
\]
and properties of the wave propagator can be obtained by applying the method of stationary phase to the integrals in the right-hand side. A crucial property in this analysis is the fact that the Hessian $\partial_\xi^2 \phi$ of the phase function $\phi(t,x,y,\xi) = \xi \cdot (x-y) + t|\xi|$ has rank $n-1$, which is strictly related to the optimal range of validity of the Miyachi--Peral estimates.

In the case $\sLap$ is a more general elliptic operator on a manifold, one cannot directly apply the Fourier transform as before. However, a more sophisticated and by now classical analysis (see, e.g., \cite{MR3645429}) shows that one can write, locally and for small times,
\[
\cos(t\sqrt{\sLap}) u(x) = Q_t u(x) + Q_{-t} u(x)
\]
up to smoothing terms, where $Q_t$ is an oscillatory integral operator of the form
\begin{equation}\label{eq:FIOrepn}
Q_t u (x) = \int \int e^{i \phi(t,x,y,\xi)} \, q(t,x,y,\xi) \, u(y) \,\dd y \,\dd \xi,
\end{equation}
whose phase function $\phi$ satisfies the \emph{eikonal equation}
\begin{equation}\label{eq:eikonal}
\partial_t \phi(t,x,y,\xi) = A(x,\partial_x \phi(t,x,y,\xi))
\end{equation}
with $A = \sqrt{H}$. Hence properties of wave propagation can still be deduced by the method of stationary phase applied to \eqref{eq:FIOrepn}.
As observed in \cite{MR0609014}, one can actually find solutions $\phi$ to the eikonal equation of the form
\begin{equation}\label{eq:phase_hormander}
\phi(t,x,y,\xi) = \varphi(x,y,\xi) + t A(y,\xi),
\end{equation}
where $\varphi(x,y,\xi)=\xi\cdot (x-y)+O(|x-y|^2|\xi|)$, so
the Hessian $\partial_\xi^2 \phi$ is closely related to $\partial_\xi^2 A$ for $t \neq 0$ and $x$ sufficiently close to $y$, and one can use the ``full curvature'' of the nondegenerate quadratic form $H$ to deduce that $\partial_\xi^2 \phi$ has rank $n-1$ at critical points of $\phi$ (for $t\neq 0$ sufficiently small).

When $H$ is not positive-definite, there are a number of obstructions preventing one from straightforwardly applying the above argument. One of these is the vanishing (and consequent lack of smoothness) of $A$ for $\xi \neq 0$, which is an obstacle to the construction of a smooth solution $\phi$ to \eqref{eq:eikonal} defined for all $\xi \neq 0$. Nevertheless, by restricting to the region where $A$ does not vanish, one can obtain a solution $\phi$ to the eikonal equation that is only defined for $\xi$ in a specific cone $\Gamma \subset \R^n \setminus \{0\}$, where $H$ behaves as an elliptic symbol. This solution $\phi$ can be then used to obtain a Fourier integral representation of the form \eqref{eq:FIOrepn} for a ``microlocalised'' version of the wave propagator $\cos(t\sqrt{\sLap})$, which turns out to be enough for our purpose.

A second, perhaps more substantial difficulty is that it is not immediately clear why $\partial_\xi^2 \phi$ should have rank $n-1$ at critical points of $\phi$, when $H$ is not positive-definite: indeed in this case $H(y,\cdot)$ vanishes on a nontrivial subspace and therefore $\partial_\xi^2 A$ has smaller rank. Note that, in general, the rank of $\partial_\xi^2 \phi$ can actually be lower: for example, if $M = \R^n = \R^{n_1} \times \R^{n_2}$ with the Lebesgue measure and $H((x_1,x_2),(\xi_1,\xi_2)) = |\xi_1|^2$, then $\sLap$ is the partial Laplacian corresponding to the factor $\R^{n_1}$ and, via the Fourier transform, one obtains a representation of the form \eqref{eq:FIOrepn}
with phase function $\phi(t,x,y,\xi) = \xi \cdot (x-y) + t|\xi_1|$; so, in this case, the rank of $\partial_\xi^2 \phi$ is strictly less than $n-1$, but, on the other hand, here the bracket-generating condition fails. A crucial part of the proof of our result consists then in showing how the bracket-generating condition prevents such a degeneracy of the Hessian.

Namely, a careful analysis of the construction of solutions to the eikonal equation \eqref{eq:eikonal} allows us to relate the rank of $\partial_\xi^2 \phi$ to the rank of the differential of the geodesic exponential map $\Exp_H$, given by the projection to $M$ of the Hamiltonian flow on $T^* M$ associated with $H$. More precisely, instead of solutions of the form \eqref{eq:phase_hormander}, here we construct, following \cite{MR597145}, solutions $\phi$ of the form
\[
\phi(t,x,y,\xi) = w(t,x,\xi) - y \cdot \xi,
\]
whose relation with the Hamiltonian flow appears to be more transparent.
Indeed, for these solutions, we prove that, in suitable coordinates, at critical points of $\phi$ with respect to $\xi$,
\[
\rank \partial_\xi^2 \phi(t,x,y,\xi) = \rank (\DD \Exp_H^y|_{-t\hat\xi}|_{V_y}),
\]
where $\hat \xi = \xi/(2\sqrt{H(y,\xi)})$, $\DD \Exp_H^y|_{-t\hat\xi} : T^*_y M \to T_x M$ is the differential at $-t\hat\xi$ of the exponential map at $y$, and $V_y$ is a codimension $1$ subspace of $T^*_y M$ (the kernel of the differential at $-t \hat\xi$ of $H|_{T^*_y M}$); in particular, $\partial_\xi^2 \phi$ has rank $n-1$ whenever $\DD \Exp_H^y$ is nondegenerate.
Note that, differently from the elliptic case, the differential $\DD \Exp_H^y|_0$ at the origin is degenerate when $H(y,\cdot)$ is.
 Nevertheless, the bracket-generating condition ensures the existence of a generic set of points $(y,\xi)$ such that $\DD \Exp_H^y|_{r\xi}$ is nondegenerate for sufficiently small $r\neq 0$ \cite{MR2513150,abb_introduction_2018,MR3852258}. This geometric information is the essential ingredient that allows us to apply stationary phase to the integral in \eqref{eq:FIOrepn} and obtain the desired results.

For techical reasons, the proof described above is carried out under additional regularity assumptions on $(M,H,\mu)$, which are satisfied, e.g., on Carnot groups. However, under the bracket-generating condition, it is possible to locally approximate, at suitable points of the manifold, any sub-Laplacian $\sLap$ with a homogeneous sub-Laplacian on a Carnot group, so the result in full generality can be recovered by a suitable form of transplantation \cite{MR3671588}.

We stress once more that the method used here is substantially different from the ones used in \cite{MR1290494,MR3513881}, which are based in an essential way on a Mehler-type formula that is specific to $2$-step structures. In contrast, the present method is much more robust and applies to structures of arbitrary step; in addition, it clearly brings to light the strict relation between properties of the functional calculus for $\sLap$ and properties of the underlying geometry (specifically, the geodesic flow). 

A natural question is whether the necessary conditions given in Theorem \ref{thm:main} are also essentially sufficient for the validity of the Mihlin--H\"ormander and Miyachi--Peral estimates. It is striking that relatively limited ``positive'' results of this kind (featuring the topological dimension $n$) are available, and (with the exception of the recent result \cite{dallara_grushin_2017} for Grushin operators of arbitrary step) only apply to $2$-step structures and enjoy a low degree of robustness.

In this connection, let us remark that, by applying the $L^p$ estimates of \cite{MR1127475} to our Fourier integral representation \eqref{eq:FIOrepn}, one could obtain estimates of Miyachi--Peral type for the ``microlocalised'' version of the wave propagator corresponding to the aforementioned ``elliptic cone'' $\Gamma$. Hence, roughly speaking, in order to obtain estimates for the full wave propagator, what remains to be understood is what happens in the complement of such an elliptic cone. While this still appears to be a challenging problem in its generality, the argument presented here may be considered as a first step in the development of a robust approach for the analysis of spectral multipliers and wave equations for sub-Laplacians.

\subsection*{Acknowledgments}

We wish to thank Michael Christ for bringing to our attention the possible use of the elliptic region for our investigations, and Luca Rizzi for pointing out references on the regularity of the sub-Riemannian exponential map.

\subsection*{Structure of the paper}
In Section \ref{sec:preliminaries} we recall basic definitions and results about pseudodifferential and Fourier integral operators that will be used throughout, and we describe the construction of a parametrix for the ``half-wave equation'' associated to a first-order positive pseudodifferential operator, assuming that a solution to the corresponding eikonal equation is given. In Section \ref{sec:eikonal} we present the construction of a solution $\phi$ to the eikonal equation associated with a general Hamiltonian on the cotangent space $T^* M$ of a smooth manifold $M$, and deduce the relation between the Hessian $\partial_\xi^2 \phi$ and the differential of the exponential map associated with the Hamiltonian flow. In Section \ref{sec:subriemannian} we recall a number of definitions and results about sub-Riemannian manifolds and sub-Laplacians, and show how the results in the previous sections can be applied to construct a Fourier integral representation for a ``microlocalised'' version of the wave propagator associated to a sub-Laplacian. Finally, in Section \ref{sec:proof_main}, we exploit such representation to prove Theorem \ref{thm:main}.

\subsection*{Notation}
We write $\R^+$ for the positive half-line $(0,\infty)$.

For nonnegative quantities $A$ and $B$, we write $A \lesssim B$ to denote that there exists a constant $C \in \R^+$ such that $A \leq C B$; expressions such as $A \lesssim_k B$ indicate that the implicit constant $C$ depends on a parameter $k$.

For subsets $U,V$ of a topological space, we write $U \Subset V$ to denote that the closure $\overline{U}$ of $U$ is compact and contained in $V$. We also write $\interior(U)$ for the interior of $U$.

\section{Fourier integral and pseudodifferential operators}\label{sec:preliminaries}

The aim of this section is to fix a few definitions and notation regarding Fourier integral operators and pseudodifferential operators.

\subsection{Distributions and linear operators}\label{sec:pre_preliminaries}
We set $\nR^n \defeq \R^n \setminus \{0\}$.
A subset $\Gamma\subset X\times\R^N$, where $X \subset \R^n$, is said to be \emph{conic} if  $(x,\lambda v)\in \Gamma$ for all $(x,v)\in \Gamma$ and $\lambda>0$.
We shall denote by $\Sch(\R^n)$ the space of Schwartz functions on $\R^n$.
The Fourier transform $\hat f$ of $f\in\Sch(\R^n)$ is given by $\hat f(\xi) = \int_{\R^n} e^{-ix\cdot\xi} f(x) \, \dd x$.

If $X \subset \R^n$ is open, we then denote by $C^\infty(X)$ and $C^\infty_c(X)$ the spaces of all (complex valued) smooth functions on $X$ and of smooth functions with compact support, with the usual topologies.
Their duals $\EDist(X)$ and $\Dist(X)$ are the space of distributions with compact supports and the space of distributions on $X$.
The \emph{support} and the \emph{singular support} of a distribution $A \in \Dist(X)$ are denoted by $\spt(A)$ and $\singspt(A)$.
The \emph{wave front set} of $A\in\Dist(X)$ is denoted by $\WF(A)$.

Let $X \subset \R^{n_X}$ and $Y \subset \R^{n_Y}$ be open sets. By identifying continuous linear operators $P : C^\infty_c(Y) \to \Dist(X)$ with their integral kernels in $\Dist(X \times Y)$ via the Schwartz kernel theorem, we can also speak of the support, the singular support and the wave front of such operators $P$.

If $P:C^\infty_c(Y) \to \Dist(X)$ and $\WF(P)=\emptyset$, then $P$ has a smooth integral kernel and extends to an operator $P : \EDist(Y) \to C^\infty(X)$;
such operators $P$ are called \emph{smoothing operators} and their class is denoted by
$\Rop^{-\infty}(Y;X)$.

We say that a subset $C \subset X \times Y$ is \emph{proper} if both projections from $C$ to $X$ and $Y$ are proper mappings.
An operator $P:C^\infty_c(Y)\to \Dist(X)$ is \emph{properly supported} if $\spt(P)$ is proper.
For instance, if $P$ is \emph{compactly supported}, i.e., $\spt(P) \Subset X\times Y$, then it is properly supported;
moreover, if $Y=X$ and $\spt(P) = \diag(X\times X)$, then $P$ is properly supported.

We denote by $\Rop(Y;X)$ the linear space of \emph{regular operators}, that is, operators $P:C^\infty_c(Y)\to\Dist(X)$ such that, for all $(x,y;\xi,\eta)\in\WF(P)$, both $\xi$ and $\eta$ are nonzero.
We will be frequently using the following properties of regular operators.
\begin{enumerate}
\item
Any operator in $\Rop(Y;X)$
extends continuously to an operator $\EDist(Y) \to \Dist(X)$ that maps $C^\infty_c(Y)$ into $C^\infty(X)$ \cite[Corollary 1.3.8, p.\ 22]{MR1362544}.
\item
Any properly supported operator in $\Rop(Y;X)$
extends continuously to an operator $\Dist(Y) \to \Dist(X)$ that maps $C^\infty(Y)$ into $C^\infty(X)$ and preserves the compactness of supports. 
\item
If one of $Q\in\Rop(Z;Y)$ and $P\in\Rop(Y;X)$ is properly supported, then $P\circ Q\in\Rop(Z;X)$ is a well defined regular operator \cite[Theorem 8.2.14, p.\ 270]{MR717035}.
\item
If $P_j\in\Rop(Y_j;X_j)$ for $j\in\{1,2\}$, then $P_1\otimes P_2\in\Rop(Y_1\times Y_2;X_1\times X_2)$ \cite[Theorem 8.2.9, p.\ 267]{MR717035}.
\end{enumerate}

Most of the above notions can be extended to the case where $X,Y,\dots$ are smooth manifolds.
For a smooth manifold $M$, we also use the notation $\nT^*_xM=T^*_xM\setminus\{0\}$ and $\nT^*M = \bigsqcup_{x\in M} \nT^*_xM$.

\subsection{Symbol classes}

Let $X\subset\R^{n}$ be an open set, $N\ge1$ and $m\in\R$.
The \emph{symbol class} $S^m(X;\R^N)$ is the space of smooth functions $a: X \times \R^N \to \C$ such that, for all $K \Subset X$, all $\alpha\in\N^{n}$ and $\gamma\in\N^N$ there is a constant $C_{\alpha\gamma}^K$ such that, for all $(x,\xi)\in K\times\R^N$,
\[
|\de_x^\alpha \de_\xi^\gamma a(x,\xi) | \le C_{\alpha \gamma}^K \langle \xi \rangle^{m-|\gamma|} ,
\]
where $\langle \xi \rangle \defeq \sqrt{1+|\xi|^2}$.
We also define $S^{-\infty}(X;\R^N) \defeq \bigcap_{m \in \R} S^m(X;\R^N)$.

Let $m\in\R$.
The \emph{classical symbol class} $S^m_\cl(X;\R^N)$ is the set of all $a \in S^{m}(X;\R^N)$ such that there exist, for all $j \in \N$,
 functions $a_j \in C^\infty(X \times \nR^N)$ homogeneous of order $m-j$ in $\xi$ such that, for all $k\in\N$,
\[
a- (1 \otimes (1-\chi)) \sum_{j<k} a_j \in S^{m-k}(X;\R^N)
\]
for some $\chi\in C^\infty_c(\R^N)$.
In this case, we call the formal series $\sum_{j\geq 0} a_j$ the \emph{asymptotic expansion} of $a$ and we write $a \sim \sum_{j\geq 0} a_j$.

The \emph{essential support} of $a\in S^m(X;\R^N)$, denoted by $\essspt(a)$, is the smallest closed conic subset $\Gamma\subset X\times\nR^N$ such that $a$ is in $S^{-\infty}$ on $( X\times\R^N)\setminus\Gamma$, i.e., $(X\times\nR^N)\setminus\Gamma$ is the union of all the open conic subsets $U$ of $X\times \nR^N$ such that,
for all $\alpha\in\N^n$, and $\beta\in\N^N$, and for all $m\in\R$ there is $C$ such that, for all $(x,\xi)\in U$,
\[
|\de_x^\alpha\de_\xi^\beta a(x,\xi) | \le C \langle \xi \rangle^{m} .
\]
If $a$ is classical and $a \sim \sum_j a_j$, then
$\essspt(a) = \overline{\bigcup_{j} \spt(a_j)}$.

\subsection{Pseudodifferential and Fourier integral operators}

Let $X \subset \R^n$ be open. A (real) \emph{phase function} is a smooth function $\phi:X \times \nR^N \to\R$ such that, for all $(x,\xi) \in X \times \nR^N$ and $\lambda > 0$, 
\begin{enumerate}
\item
$\phi(x,\lambda\xi)=\lambda\phi(x,\xi)$;
\item
$\dd\phi(x,\xi)\neq0$.
\end{enumerate}
The \emph{stationary set} $\Sigma_\phi \subset X \times \nR^N$ and the \emph{wave front} $\Lambda_\phi \subset \nT^* X$ of a phase function $\phi$ are the conic sets defined by
\[
\Sigma_\phi \defeq \{(x,\xi) \in X \times \nR^N \tc \de_\xi\phi(x,\xi) = 0\}, \quad
\Lambda_\phi \defeq \{ (x,\de_x\phi(x,\xi)) : (x,\xi)\in\Sigma_\phi \}.
\]

Let $\phi$ be a phase function on $X \times \nR^N$ and $a\in S^m(X;\R^N)$.
The \emph{Fourier integral} (or \emph{oscillatory integral}) with \emph{phase} $\phi$ and \emph{amplitude} $a$ is the distribution 
\begin{equation}\label{eq11200942}
\int_{\R^N} e^{i\phi(x,\xi)}\, a(x,\xi) \,\dd\xi 
\end{equation}
in $\Dist(X)$, whose wave front set is contained in
\begin{equation}\label{eq11231818}
	\{ ( x,\de_x\phi(x,\xi)) : (x,\xi)\in\essspt(a)\cap \Sigma_\phi\} \subset \Lambda_\phi ,
\end{equation}
see \cite[Theorem 2.2.2, p.~29]{MR1362544}.

Let now $X \subset \R^{n_X}$ and $Y \subset \R^{n_Y}$ be open sets. Let $\phi : X \times Y \times \nR^N \to \R$ be a phase function, and let $a \in S^m(X \times Y;\R^N)$.
The operator  $\Theta : C^\infty_c(Y)\to\Dist(X)$, whose distributional integral kernel is the Fourier integral \eqref{eq11200942} with phase $\phi$ and amplitude $a$, is called a \emph{Fourier integral operator}.
We shall describe such operators with the formula
\begin{equation}\label{eq11200944}
\Theta u(x) = \int_Y \int_{\R^N} e^{i\phi(x,y,\xi)}\, a(x,y,\xi) \, u(y) \,\dd\xi \,\dd y .
\end{equation}
The phase function $\phi$ is an \emph{operator phase function} if it satisfies the following condition: for all $(x,y,\xi) \in X \times Y \times \nR^n$,
if $\de_\xi\phi(x,y,\xi)=0$, then $\de_x\phi(x,y,\xi)\neq0$ and $\de_y\phi(x,y,\xi)\neq0$.
If $\phi$ is an operator phase function, then from \eqref{eq11231818} one can deduce that the Fourier integral operator $\Theta$ defined in \eqref{eq11200944} is a regular operator, that is, $\Theta\in\Rop(Y;X)$.

If $X = Y$ and $n_X=n_Y=n$, the simplest example of operator phase function is the \emph{standard phase} $(x,y,\xi) \mapsto (x-y) \cdot \xi$. The Fourier integral operators corresponding to the standard phase are called \emph{pseudodifferential operators}. More precisely, the pseudodifferential operator $\Theta$ on $X$ with amplitude $a \in S^m(X \times X;\R^{n})$ is the operator given by
\[
\Theta u(x) = (2\pi)^{-n}\int_X \int_{\R^N} e^{i(x-y) \cdot \xi}\, a(x,y,\xi) \, u(y) \,\dd\xi \,\dd y. 
\]
We denote by $\Psi^m(X)$ the collection of all pseudodifferential operators with amplitude in $S^m(X \times X;\R^n)$, which are called pseudodifferential operators of order $m$ on $X$. 
Moreover, for $m\in\R$, we denote by $\Psi^m_\cl(X)$ the collection of all \emph{classical pseudodifferential operators of order $m$}, i.e., the pseudodifferential operators with amplitude in $S^m_\cl(X \times X;\nR^n)$. One can check that the set $\Psi^{-\infty}(X)$ of pseudodifferential operators on $X$ with amplitude in $S^{-\infty}(X \times X \times \R^n)$ coincides with $\bigcap_{m \in \R} \Psi^{m}(\R)$ and with the set $\Rop^{-\infty}(X;X)$ of smoothing operators on $X$.

If the amplitude of a pseudodifferential operator on $X$ does not depend on the variable $y$, then it is called \emph{(Kohn--Nirenberg) symbol}. While different amplitudes may define the same pseudodifferential operator $P$, the symbol (if it exists) is uniquely determined by the operator, and moreover $P$ is classical if and only if its symbol is classical. Every properly supported pseudodifferential operator has a symbol, and every pseudodifferential operator differs from a properly supported one by a smoothing operator.
For $m \in \R$, we define the \emph{principal symbol} of $P \in \Psi^m_\cl(X)$ as the term of degree $m$ in the asymptotic expansion of the symbol of any pseudodifferential operator that differs from $P$ by a smoothing operator.

The basic example of pseudodifferential operator of order $m$ is a differential operator $P=\sum_{|\alpha|\le m} p_\alpha(x)(-i\de_x)^\alpha$ with smooth coefficients $p_\alpha$.
This is a classical, properly supported pseudodifferential operator. Its symbol is $\sum_{0\le|\alpha|\le m} p_\alpha(x)\xi^\alpha$ and its principal symbol is $\sum_{|\alpha|= m} p_\alpha(x)\xi^\alpha$.

Pseudodifferential operators and classical pseudodifferential operators can be defined on manifolds $M$, because of the invariance of the main objects under change of coordinates, see  \cite[Definition 18.1.20, p.~85]{MR781536}.
Although the symbol of a pseudodifferential operator is not well defined on a manifold, the principal symbol of a classical pseudodifferential operator $P \in \Psi^m_\cl(M)$ is a well-defined smooth function on $\nT^*M$ which is homogeneous of degree $m$ along the fibres.

A classical pseudodifferential operator $P \in \Psi_\cl^m(M)$ is said to be \emph{elliptic of order $m$} if its principal symbol never vanishes on $\nT^* M$.
For elliptic pseudodifferential operators, one can easily construct approximate square roots via an iterative argument (see, e.g., the first part of the proof of \cite[Theorem~3.3.1]{MR3645429} or \cite{MR0237943}):

\begin{lemma}\label{lem09030037}
	If $P\in\Psi_{\cl}^m(M)$ is elliptic of order $m$ with nonnegative principal symbol $\tilde p$, then there is a properly supported $Q\in\Psi_{\cl}^{m/2}(M)$ elliptic of order $m/2$ with principal symbol $\sqrt{\tilde p}$ such that $Q^2 - P \in \Psi^{-\infty}(M)$.
\end{lemma}

\subsection{Fourier integral representation of the half-wave propagator}
The result below is a variation of results available in the literature 
(see, in particular, \cite[Section 3]{MR0609014}, \cite[Section 20.2]{MR1852334} and \cite[Section 4.1]{MR3645429}), keeping track of supports and ensuring that the construction produces classical symbols.

\begin{theorem}\label{thm08141709}
	Let $\opA$ be a properly supported pseudodifferential operator of order~$1$ on an open set $X\subset\R^n$ with classical real symbol.
	Let $ A : X\times \nR^n \to \R$ be the principal symbol of $\opA$.
	Let $\phi: (-T,T)\times X \times X\times \nR^n \to \R$ be a phase function such that
\[	
X \times X \times \nR^n \ni (x,y,\xi) \mapsto \phi(0,x,y,\xi) \in \R
\]
is an operator phase function, and assume that $\phi$ satisfies the following eikonal equation: for all $(t,x,y,\xi)\in(-T,T)\times X\times X\times\nR^n$,
	\begin{equation}\label{eq08091101}
	\de_t\phi(t,x,y,\xi) =  A(x,\de_x\phi(t,x,y,\xi))  .
	\end{equation}
	Then, for every open subsets $X',X''$ of $X$ with $X'' \Subset X' \Subset X$, there is $T' \in (0,T]$ such that the following hold true:
	if $\Gamma\subset\nR^n$ is a closed cone and $P \in \Rop(X;X)$ is a Fourier integral operator with distributional integral kernel
	\[
	P(x,y) = \int_{\R^n} e^{i\phi(0,x,y,\xi)} p(x,y,\xi) \,\dd\xi
	\]
	and amplitude $p\in S^0_{\cl}(X\times X;\R^n)$
	satisfying $\essspt(p)\subset X''\times X''\times\Gamma$,
	then there is a Fourier integral operator $Q\in \Rop(X;(-T',T')\times X)$ with distributional integral kernel
	\begin{equation}\label{eq:fioapprox}
	Q_t(x,y) \defeq Q(t,x,y) = \int_{\R^n} e^{i\phi(t,x,y,\xi)} q(t,x,y,\xi) \,\dd\xi 
	\end{equation}
	and amplitude $q\in S^0_{\cl}((-T',T')\times X\times X; \R^n)$,
	such that:
	\begin{enumerate}[label=(\roman*)]
	\item\label{en:trnsupp} $\spt(q)\subset (-T',T')\times X'\times X'\times\Gamma$;
	\item\label{en:trninit} $Q_t \in \Rop(X;X)$ for all $t \in (-T',T')$, and $Q_0 - P\in \Rop^{-\infty}(X; X)$;
	\item\label{en:trnsoln} $(i\de_t + \opA) Q\in\Rop^{-\infty}(X;(-T',T')\times X)$.
	\end{enumerate}
\end{theorem}

\begin{proof}
By our assumption on $\phi$, both $\de_{(x,\xi)} \phi$ and $\de_{(y,\xi)} \phi$ never vanish on $\{0\} \times X \times X \times \nR^n$. Hence, if we take $X_0 \Subset X$ such that $X' \Subset X_0$, we can find $T_0 \in (0,T]$ such that both $\de_{(x,\xi)} \phi$ and $\de_{(y,\xi)} \phi$ never vanish on $(-T_0,T_0) \times X_0 \times X_0 \times \nR^n$. In other words, up to shrinking $(-T,T)$ and $X$, we may assume that
\[
X \times X \times \nR^n \ni (x,y,\xi) \mapsto \phi(t,x,y,\xi) \in \R
\]
is an operator phase function for all $t \in (-T,T)$. In particular, the Fourier integral operators $Q$ and $Q_t$ defined by \eqref{eq:fioapprox} for any given amplitude $q$ are regular operators.

Notice that \eqref{eq08091101} forces $(x,\de_x\phi(t,x,y,\xi))$ to be in the domain $X \times \nR^n$ of $A$, and in particular $\de_x \phi \neq 0$
on the domain of $\phi$. Since $\de_x\phi$ is $1$-homogeneous in $\xi$, up to taking a smaller $T$, condition \cite[(2.13)]{MR0609014} is satisfied by $\phi$ on $(-T,T)\times X\times X\times \nR^n $.
	So, if $q \in S^0_\cl((-T',T')\times X\times X;\R^n)$ for some $T' \in (0,T]$ to be chosen later, and $Q$ is defined by \eqref{eq:fioapprox}, then
	\[
	(i\de_t + \opA) \, Q = \int_{\R^n} e^{i\phi(t,x,y,\xi)} r(t,x,y,\xi) \,\dd\xi
	\]
	where $r \in S^1((-T',T') \times X \times X;\R^n)$ has the asymptotic expansion described in \cite[Theorem 2.12]{MR0609014}.
	Namely, 
	if $a$ is the symbol of $\opA$ 
	and if we write $q\sim\sum_{j \geq 0} q_{-j}$ and $a\sim\sum_{j \geq 0} a_{1-j}$ for the asymptotic expansions of $a$ and $q$ (here $a_1 = A$),
	then 
\begin{equation}\label{eq:transport_comp_symbol}
\begin{split}
r(t,x,y,\xi) 
&= e^{-i\phi(t,x,y,\xi)}\left.(i\de_t + \opA_z)\left[ e^{i\phi(t,z,y,\xi)} q(t,z,y,\xi)\right]\right|_{z=x}\\
&\sim \left( A(x,\de_x\phi) -\de_t\phi \right) q + \sum_{j \geq 0} r_{-j} ,
\end{split}
\end{equation}	
where, for $k \leq 0$, the $r_k \in C^\infty( (-T',T') \times X \times X \times \nR^n )$ are homogeneous in $\xi$ of degree $k$ and are given by
	\[\begin{split}
	r_k &=i\de_t q_k 
		+ a_0(x,\de_x\phi) q_k \\
		&- i \sum_{|\alpha|=1} (\de^\alpha_\xi A)(x,\de_x\phi) \de_x^\alpha q_k 
		- i \sum_{|\alpha|=2} \frac{1}{\alpha!} (\de^\alpha_\xi A)(x,\de_x\phi) \, (\de_x^\alpha \phi) \, q_k
		- R_k ;
	\end{split}\]
	here the remainder $R_k = R_k(a,\phi,q_0,q_{-1},\dots,q_{k+1})$ 
is homogeneous in $\xi$ of degree $k$ and 
has the form
	\begin{equation}\label{eq:transport_remainder}
	R_k= \sum_{\substack{|\alpha| \le 1-k \\ 0 \geq \ell>k}} c_{k\alpha}^\ell(\phi,a) \, \de_x^\alpha q_\ell,
	\end{equation}
	where the $c_{k\alpha}^\ell(\phi,a)$ are certain polynomials in the derivatives of $\phi$ and the $a_j$ (independent of the $q_j$).
	In particular, $R_0 = 0$.
		
			Note that $ A(x,\de_x\phi) -\de_t\phi = 0$, because of \eqref{eq08091101}. Thus, in view of \eqref{eq:transport_comp_symbol}, in order for \ref{en:trnsoln} to be satisfied, it is sufficient to choose $q$ so that $r_k=0$ for all $k\le 0$. Similarly, \ref{en:trninit} corresponds to the condition  $q_k|_{t=0}=p_k$ for all $k \leq 0$, where $p \sim \sum_{j\geq 0} p_{-j}$.
			
	Notice that $-ir_k=0$ is a linear differential equation in $q_k$ where all derivatives of $q_k$ have real coefficients.
	More precisely, consider the time-dependent real vector field $W$ on $X$ 
	with parameters $(y,\xi)\in X\times\nR^n$, given by
	\[
	W(t,x,y,\xi) = - \sum_{|\alpha|=1} (\de_\xi^\alpha A)(x,\de_x\phi)\de_x^\alpha ,
	\]
	and the function $F(t,x,y,\xi) =  -ia_0(x,\de_x\phi) - \sum_{|\alpha|=2} \frac{1}{\alpha!}  \de_\xi^\alpha A(x,\de_x\phi)  \de_x^\alpha \phi$.
	Then we want $q_k$ to solve the equation
	\begin{equation}\label{eq08171655}
	\begin{cases}
	 	\de_t q_k + W q_k + F q_k + iR_k=0 ,\\
		q_k|_{t=0}=p_k .
	\end{cases}
	\end{equation}
	This equation is called \emph{transport equation} and it is solved with the method of characteristics.
	Namely, for $(t,y,\xi) \in (-T,T) \times X \times \nR^n$, let 
	\[
	C(t,y,\xi) = \left\{ \gamma : I \to X \tc 
	\begin{array}{c}
	 I\subset (-T,T) \ \text{ interval with }0,t\in I,\\
	\forall s \in I \tc \gamma'(s)=W(s,\gamma(s),y,\xi)
	\end{array}
	\right\} 
	\]
	be the set of the integral curves of $W(\cdot,\cdot,y,\xi)$ defined at times $0$ and $t$, and let $\Omega\subset(-T,T)\times X\times X\times \nR^n$ be the open set
	\[
	\Omega = \left\{ (t,\gamma(t),y,\xi) \tc
  (t,y,\xi) \in (-T,T) \times X \times \nR^n, \, \gamma \in C(t,y,\xi) 
	\right\} .
	\]
	Notice that, since $W$ is $0$-homogeneous in $\xi$, the set $\Omega$ is  conic.
	
	For every $(t,x,y,\xi)\in\Omega$, the initial-value problem \eqref{eq08171655} induces a Cauchy problem for a linear ODE along a curve $\gamma \in C(t,y,\xi)$ with $\gamma(t) = x$.
	Since this Cauchy problem is globally solvable, we obtain that, if $R_k$ is defined and smooth on the whole $\Omega$, then there is a well-defined $q_k:\Omega\to\C$ solution to~\eqref{eq08171655}.
	Smoothness and uniqueness of $q_k$ on $\Omega$ are also guaranteed by the theory of ODEs, and $q_k$ is $k$-homogeneous in $\xi$ whenever $R_k$ is.

	Since $R_0=0$ is defined, smooth, and $0$-homogeneous on $\Omega$, the solution $q_0$ to \eqref{eq08171655} exists on $\Omega$.
	Inductively, by \eqref{eq:transport_remainder}, it follows that the $q_k:\Omega\to\C$ solving \eqref{eq08171655} are defined, smooth and $k$-homogeneous on $\Omega$ for all $k\le0$.

Let now	$\Omega_0$ be the open subset of $\Omega$ defined by
	\[
	\Omega_0 = \left\{ (t,\gamma(t),y,\xi) \tc \begin{array}{c} (t,y,\xi) \in (-T,T) \times X \times \nR^n, \,
	\gamma \in C(t,y,\xi), \\ (\gamma(0),y,\xi) \notin \essspt(p) \end{array}
	\right\} .
	\]
Arguing as above, the solution to~\eqref{eq08171655} is unique on $\Omega_0$, but here the initial value for the Cauchy problem along each integral curve is zero, whence $q_k = 0$ on $\Omega_0$.

	Let $X',X''$ be open subsets of $X$ with  $X'' \Subset X' \Subset X$, and let $K$ be a compact neighbourhood of $\overline{X''}$ in $X'$. 
	We claim that
	there is $T'\in(0,T]$ such that, if $\Gamma$ is a closed cone in $\nR^n$ and $\essspt(p)\subset X''\times X''\times\Gamma$, then, for all $k$,
	\begin{equation}\label{eq:ampl_supp_cond}
	\spt(q_k|_{\Omega'})
	\subset (-T',T') \times K\times K\times \Gamma
	\subset \Omega,
	\end{equation}
	where $\Omega' = \Omega \cap (-T',T') \times X \times X \times \nR^n$.
	Indeed, since $\Omega$ is a conic open neighbourhood of $\{0\}\times X\times X\times\nR^n$ in $(-T,T)\times X\times X\times\nR^n$,  there is $\epsilon \in (0,T]$ such that $(-\epsilon,\epsilon)\times X'\times X'\times \nR^n \subset \Omega$. A further compactness argument yields a $T' \in (0,\epsilon]$ so that $\gamma(t) \in K$ for all $(t,y,\xi) \in (-T',T') \times X''\times \nR^n$ and $\gamma \in C(t,y,\xi)$ with $\gamma(0) \in X''$. By the previous discussion, if $\essspt(p)\subset X''\times X''\times\Gamma$, then
\[\begin{split}
&\{ (t,x,y,\xi) \in \Omega \tc t \in (-T',T'), q_k(t,x,y,\xi) \neq 0\} \\
&\subset \left\{ (t,\gamma(t),y,\xi) \tc t \in (-T',T'),	\gamma \in C(t,y,\xi), \, (\gamma(0),y,\xi) \in X'' \times X'' \times \Gamma 	\right\} \\
&\subset (-T',T') \times K \times X'' \times \Gamma 
\end{split}\]	
and \eqref{eq:ampl_supp_cond} follows.

We can now extend by zero the functions $q_k|_{\Omega'}$ to smooth homogeneous functions $q_k$ on the whole $(-T',T') \times X \times X \times \nR^n$, and these extensions still satisfy \eqref{eq08171655}; this is because $(-T',T') \times K\times K\times \Gamma$ is closed in $(-T',T') \times X\times X\times \nR^n$, and $\Omega'$ contains $\{0\} \times X \times X \times \nR^n$. Hence any $q \in S^0_\cl((-T',T') \times X \times X;\R^n)$ with the asymptotic expansion $\sum_{j \geq 0} q_{-j}$ satisfies \ref{en:trninit} and \ref{en:trnsoln}.
	One of such symbols is given by $q(t,x,y,\xi) = \sum_{j \geq 0} \chi_j(\xi) \, q_{-j}(t,x,y,\xi)$ for suitable smooth cutoffs $\chi_j$ vanishing at $\xi = 0$, and this $q$ also satisfies
	$\spt(q)\subset (-T',T')\times K\times K\times \Gamma$, as desired.
\end{proof}

\section{Eikonal equation}\label{sec:eikonal}

Consider the initial value problem for the eikonal equation~\eqref{eq08091101}, namely
\begin{equation}\label{eq09121416}
\begin{cases}
\de_t\phi(t,x,y,\xi) = A(x,\de_x\phi(t,x,y,\xi)), \\
\phi(0,x,y,\xi) = (x-y)\cdot\xi,
\end{cases}
\end{equation}
on a simply connected coordinate domain $M_o$ in a manifold $M$.
Following Tr\`eves, see for instance \cite[Example 2.1, p.\ 320]{MR597145},
we seek a solution $\phi$ in the form $\phi(t,x,y,\xi) = w(t,x,\xi)-y\cdot\xi$.
The eikonal equation~\eqref{eq09121416} is then equivalent to 
\begin{equation}\label{eq09121416bis}
\begin{cases}
\de_t w(t,x,\xi) = A(x,\de_x w(t,x,\xi)), \\
w(0,x,\xi) = x\cdot\xi .
\end{cases}
\end{equation}
If we define the 1-forms $\alpha_\xi = \xi\cdot\dd x$ and
\begin{equation}\label{eq09121105}
	\mu^\xi 
	= \dd_{(t,x)}w = \de_tw \,\dd t + \dd_xw 
	= \mu_\R^\xi \,\dd t + \mu_M^\xi ,
\end{equation}
on an open subset of $\R\times M_o$, then we have that $\mu^\xi$ is closed, $\mu_\R^\xi=A(\mu_M^\xi)$ and $\mu_M^\xi|_{t=0} = \alpha_\xi$.
Moreover, by Poincar\'e's Lemma, since $M_o$ is simply connected, $w(t,x,\xi)$ in return is determined by $\mu^\xi$ and~\eqref{eq09121105} (up to an additive constant).

We may therefore study an equation in $\mu$ in place of~\eqref{eq09121416bis}: 
Given a closed $1$-form $\alpha$ on $M_o$, we will show that there is a unique 1-form $\mu^\alpha= \mu_\R^\alpha \,\dd t + \mu_M^\alpha$ on a neighbourhood of $\{0\}\times M_o$ satisfying 
\begin{equation}\label{eq09121101}
	\begin{cases}
	\dd\mu^\alpha=0, \\ 
	\mu_\R^\alpha=A(\mu_M^\alpha) , \\
	\mu_M^\alpha|_{t=0} = \alpha .
	\end{cases}
\end{equation}
Once we have $\mu^\alpha$, we define $\phase(t,x,y,\alpha)$ as $\wphase^\alpha(t,x)-\wphase^\alpha(0,y)$
where $\wphase^\alpha$ is determined by $\mu^\alpha = \dd_{(t,x)}\wphase^\alpha$.
Finally, we will characterize the points where $\de_\alpha\phase = 0$ and the rank of the Hessian $\de^2_\alpha\phase$ at these points
in terms of the Hamiltonian flow of $A$ on $T^*M$,
where derivatives in $\alpha$ are in the sense of G\^ateaux.

Up to this point, the construction is coordinate-free.
A choice of coordinates determines the restriction to the subspace of forms $\alpha = \xi\cdot\dd x$, with the corresponding phase
\[
\phi(t,x,y,\xi) = \phase(t,x,y,\xi\cdot\dd x) .
\]
We will show that both the characterization of critical points  and the rank of the Hessian at those points do not depend on such restriction.

\subsection{Preliminaries on symplectic geometry}\label{sec09190901}

Here we recall some fundamental definitions and results from symplectic geometry. We refer to \cite{MR2954043,MR3674984} for additional details.

A \emph{symplectic form} on a smooth manifold $N$ is a $2$-form $\omega \in \Forms^2(N)$ such that $\dd\omega=0$ and $\omega|_p$ is non-degenerate for every $p \in N$.
The pair $(N,\omega)$ is called \emph{symplectic manifold}.

Every smooth function $F : N \to \R$ on a symplectic manifold has an associated Hamiltonian vector field $\X_F \in \Sect(TN)$ defined by
\[
\dd F|_p(v) = \omega|_p(\X_F|_p,v)
\qquad\forall v\in T_p N,\ \forall p \in N .
\]
We denote by $\Phi_F : (t,p) \mapsto \Phi^t_F(p)$ the flow of $\X_F$ on $N$.
As usual, the domain of $\Phi_F$ is an open neighbourhood of $\{0\}\times N$ in $\R\times N$.
We recall that
\begin{equation}\label{eq09181520}
\X_FF=0 ,
\end{equation}
that is, $F$ is constant along the integral curves of $\X_F$, 
and that
\begin{equation}\label{eq09181521}
(\Phi^t_F)^*\omega = \omega ,
\end{equation}
that is, the flow $\Phi_F$ preserves the symplectic form.

\begin{proposition}\label{prop04121107}
	Let $(N,\omega)$ be a symplectic manifold, $\psi_0:U\to N$ a smooth map from a manifold $U$ and $F:N\to\R$ smooth.
	Let $\cU \subset \R \times U$ be the preimage of the domain of $\Phi_F$ via the map $\R\times U\to \R\times N$, $(t,p)\mapsto(t,\psi_0(p))$;
	clearly, $\cU$ is a neighbourhood of $\{0\}\times U$.
	Define $\psi: \cU \to N$ by
	\[
	\psi(t,p) = \Phi_F^t(\psi_0(p)) .
	\]
If $F\circ\psi_0$ is constant, then also $F\circ\psi$ is constant; and, if moreover $\psi_0^*\omega=0$,
	then also $\psi^*\omega=0$.
\end{proposition}
\begin{proof}
	Recall that we have a canonical identification $T_{(t,p)} \cU \simeq T_t\R\times T_p U$.
	If $(t,p) \in \cU$, $r \in \R$ and $v\in T_p U$, then
	\begin{equation}\label{eq04111551}
	D\psi|_{(t,p)}[r\de_t + v] 
	= r \X_F|_{\psi(t,p)} + \DD\Phi^t_F\circ\DD\psi_0|_p[v] .
	\end{equation}
	Since $\X_F|_{\psi(t,p)} = \left.\frac{\dd}{\dd h}\right|_{h=0} \Phi^{h}_F\circ\Phi^{t}_F(\psi_0(p))
	= \left.\frac{\dd}{\dd h}\right|_{h=0} \Phi^{t}_F\circ\Phi^{h}_F(\psi_0(p))$,
	we have
	\begin{equation}\label{eq09181523}
	\X_F|_{\psi(t,p)} = \DD\Phi^t_F|_{\psi_0(p)} [\X_F|_{\psi_0(p)}] .
	\end{equation}
	Suppose that $F\circ\psi_0$ is constant.
	Then $F\circ\psi$ is constant by~\eqref{eq09181520}.
	Suppose in addition that $\psi_0^*\omega=0$.
	Notice that, by the definition of pull-back and~\eqref{eq04111551},
	\begin{multline*}
	\psi^*\omega|_{(t,p)} \left(r\de_t+v,r'\de_t+v'\right)
		= r\omega \left(\X_F(\psi(t,p)),\DD\Phi^t_F \circ \DD\psi_0|_p[v']\right)  \\
		- r'\omega \left(\X_F(\psi(t,p)),\DD\Phi^t_F \circ \DD\psi_0|_p[v]\right)
		+ \omega \left(\DD\Phi^t_F\circ\DD\psi_0|_p[v],\DD\Phi^t_F \circ \DD\psi_0|_p[v']\right) ,
	\end{multline*}
	for all $(t,p)\in \cU$, $r,r'\in\R$ and $v,v'\in T_pU$.
	By~\eqref{eq09181521} and~\eqref{eq09181523},
	we have
	\begin{equation*}
	\omega \left(\X_F(\psi(t,p)),\DD\Phi^t_F \circ \DD\psi_0|_p[v]\right)
	= \omega\left(\X_F(\psi_0(p)),\DD\psi_0|_p[v] \right) 
	= \dd F(\DD\psi_0|_p[v]) 
	= 0 ,
	\end{equation*}
	where we used the hypothesis $F\circ\psi_0$ is constant.
	Again, by~\eqref{eq09181521} we also obtain
	\[
	\omega\left(\DD\Phi^t_F \circ \DD\psi_0|_p[v], \DD\Phi^t_F \circ \DD\psi_0|_p[v'] \right)
	= \omega\left(\DD\psi_0|_p[v], \DD\psi_0|_p[v'] \right)
	= 0,
	\]
	where we used the hypothesis $\psi_0^*\omega=0$.
	We conclude that $\psi^*\omega=0$.
\end{proof}

\begin{corollary}\label{cor:lagrangian_uniqueness}
Let $(N,\omega)$ be a symplectic manifold, $\psi_0 : U \to N$ a smooth map from a manifold $U$ and $F : N \to \R$ smooth. Assume that $F \circ \psi_0$ is constant and $\psi_0^* \omega = 0$. For all $p \in U$, if $2\dim \Im(\DD \psi_0|_p) \geq \dim N$, then $\X_F|_p \in \Im(\DD \psi_0|_p)$.
\end{corollary}
\begin{proof}
Let $\psi : \cU \to N$ be constructed as in Proposition \ref{prop04121107}. Then clearly
\[
\X_F|_p \in \Im(\DD \psi|_{(0,p)}) \supseteq \Im(\DD \psi_0|_p).
\]
On the other hand $\psi^* \omega = 0$, so the symplectic bilinear form $\omega|_{\psi_0(p)}$ vanishes on $\Im(\DD \psi|_{(0,p)}) \times \Im(\DD \psi|_{(0,p)})$ and therefore $2\dim \Im(\DD \psi|_{(0,p)}) \leq \dim N$. Dimensional considerations then imply that $\Im(\DD \psi|_{(0,p)}) = \Im(\DD \psi_0|_p)$ and $\X_F|_p \in \Im(\DD \psi_0|_p)$.
\end{proof}

The cotangent space $T^*M$ of a smooth manifold $M$ has a canonical symplectic structure, described as follows.
Let $\pi_M : T^*M\to M$ be the bundle projection and $\taut \in \Forms^1(T^*M)$ the \emph{tautological form} defined by
\begin{equation*}
\taut|_\alpha(v) = \alpha(\DD\pi_M[v])
\end{equation*}
for $\alpha\in T^*M$ and $v\in T_\alpha(T^*M)$.
The symplectic form on $T^*M$ is
\[
\omega=-\dd\taut .
\]
The tautological $1$-form is characterised by the fact that, if $\mu \in \Forms^1(M)$ is a $1$-form (which in particular is a smooth embedding $\mu : M \to T^*M$), then $\mu^* \taut=\mu$.
We shall need the following lemma.
Recall that a submanifold $S \subset T^*M$ is called \emph{Lagrangian} if $\dim(S) = \dim M$ and $\omega|_{TS} = 0$.

For a proof of the following lemma, see for instance \cite[Proposition 9.20]{MR2954043}.

\begin{lemma}\label{lem03071514}
$\mu \in \Forms^1(M)$ is closed if and only if $\mu(M)$ is a Lagrangian submanifold of $T^*M$, i.e., $\mu^*\omega = 0$.
\end{lemma}

If $F:\cD_F\to\R$ is a smooth function on some open set $\cD_F\subset T^*M$,
we define the \emph{(Hamiltonian) exponential map} 
\begin{equation}\label{eq11201112}
\Exp_{F}^{x,t}(\xi) = \pi_M(\Phi_F^t(\xi)),
\end{equation}
for all $x \in M$, $\xi \in T^*_x M$, $t \in \R$  such that $(t,\xi)$ is in the domain of $\Phi_F$.

A system of coordinates $(U,x)$ on an open set $U \subset M$ induces so-called \emph{canonical coordinates} $(T^* U,(x,\eta))$ on $T^*U=\pi_M^{-1}(U)\subset T^*M$, whereby $\alpha \in T^* U$ corresponds to the pair $(x(\alpha), \sum_j \eta_j(\alpha) \dd x_j)$ in the trivialisation of $T^* U$ induced by $(U,x)$.
In canonical coordinates we have
\begin{equation*}
\taut = \sum_j \eta_j \,\dd x_j
\qquad\text{ and }\qquad
\omega = \sum_j \dd x_j\wedge\dd\eta_j .
\end{equation*}
Moreover, if $\cD_F \subset T^*M$ is open and $F:\cD_F\to\R$ is smooth, then
\begin{equation}\label{eq09190852}
\X_F = \sum_j \frac{\de F}{\de \eta_j} \de_{x_j} - \sum_j \frac{\de F}{\de x_j} \de_{\eta_j}  .
\end{equation}
Hence, a curve $\gamma(t) =(x(t),\eta(t))$ in $\cD_F \subset T^*M$ is an integral curve of $\X_F$ if and only if it satisfies the \emph{Hamilton--Jacobi equations}:
\begin{equation}\label{eq09190853}
\begin{dcases}
\dot x_j = \frac{\de F}{\de \eta_j} ,\\
\dot \eta_j = -\frac{\de F}{\de x_j} ,
\end{dcases}
\end{equation}
for all $j$.

\subsection{Solution to the eikonal equation}\label{sec09270951}
For further reference on the content of this section, see \cite[p.\ 167]{MR0447753} and \cite[Chapter VI]{MR597145}.

For a manifold $M$, we denote the space of closed $1$-forms by $\CLF(M) = \{\alpha \in \Forms^1(M):\dd \alpha=0\}$,
and the bundle projection $T^*M\to M$ by $\pi_M$.

Let $M$ be a manifold and
$A : \cD_A \to \R$ a smooth function defined on an open set $\cD_A \subset T^* M$. Note that $\cD_A$ inherits the symplectic structure of $T^* M$.
As in the previous Section~\ref{sec09190901}, we denote by $\X_A$ and $\Phi_A$ the Hamiltonian vector field of $A$ and its flow, respectively.

If $\alpha \in \Forms^1(U)$ for some $U\subset M$ open, let $\rho^\alpha$ be the smooth map
\begin{equation}\label{eq:def_rhoalpha}
\cU^\alpha \ni (t,x) \mapsto \rho^\alpha_t(x) \defeq \pi_M\Phi_A^{-t}(\alpha|_x) \in M,
\end{equation}
where 
\begin{equation}\label{eq:Ualpha}
\cU^\alpha = \{ (t,x) \in \R \times U \tc (-t,\alpha|_x) \text{ is in the domain of $\Phi_A$}\}
\end{equation}
is an open neighbourhood of $\{0\} \times U$ in $\R \times U$. Note that $\rho^\alpha_0 = \Id_U$.

Set $\tilde M \defeq \R\times M$.
We shall write an element $\alpha\in T^*\tilde M$ as
\[
\alpha = \alpha_\R \, \dd t|_t + \alpha_M
\]
with $t,\alpha_\R\in\R$ and $\alpha_M\in T^*M$.
 Note that $T^* \tilde M$ is naturally isomorphic to the product $T^* \R \times T^* M$.
The composition of the projection $T^*\tilde M\to T^*M$ with the bundle projection $\pi_M : T^*M \to M$ gives a submersion $\tilde\pi_M : T^*\tilde M \to M$. Then the canonical symplectic form $\omega_{\tilde M}$ on $T^* \tilde M$ is the ``sum'' of the symplectic forms on the factors; more precisely,
	\begin{equation}\label{eq09181756}
	\omega_{\tilde M} = \dd t\wedge\dd\tau + \tilde \omega_M ,
	\end{equation}
	where $(t,\tau)$ are the canonical coordinates on $T^*\R$ and $\tilde \omega_M$ is the pull-back via $\tilde\pi_M$ of the canonical symplectic form $\omega_M$ on $T^* M$.

Let $\cD_F = T^* \R \times \cD_A \subset T^* \tilde M$, and define $F : \cD_F \to \R$ by
\begin{equation}\label{eq09181755}
F(\tau\dd t|_t +\eta) = A(\eta) - \tau ,
\end{equation}
for every $(t,p)\in\tilde M$, $\tau\dd t|_t \in T_t^*\R$ and $\eta\in T^*_pM$. A moment's thought shows that the vector field $\X_F$ associated with $F$ splits as follows:
\begin{equation*}
\X_F = -\de_t + \tilde \X_A,
\end{equation*}
where $\tilde\X_A$ is the lifting of $\X_A$ to $\cD_F$. Consequently the flow of $\X_F$ is given by
\begin{equation}\label{eq09181758}
\Phi^s_F(\tau \,\dd t|_t+\eta) = \tau\dd t|_{t-s} + \Phi^s_A(\eta) 
\end{equation}
for all $s,t,\tau \in \R$ and all $\eta$ in the domain of $\Phi_A^s$.

The main result of this section is the following proposition, where the map $\rho^\alpha$ and the set $\cU^\alpha$ are defined as in \eqref{eq:def_rhoalpha} and \eqref{eq:Ualpha}.

\begin{proposition}\label{prop09181821}
	Let $U\subset M$ be open, $\alpha\in\CLF(U)$ and $\epsilon > 0$ such that
	\begin{enumerate}[label=(\roman*)]
	\item\label{prop09181821item1}
	$(-\epsilon,\epsilon) \times U \subset \cU^\alpha$;
	\item\label{prop09181821item2}
	$\rho^\alpha_s|_U :U\to M$ is an embedding for all $s\in(-\epsilon,\epsilon)$.
	\end{enumerate}
	Then the set
	\[
	\tilde U=\{(s,\rho^\alpha_s(x)) : |s|<\epsilon,\ x\in U\}
	\]
	is open in $\tilde M$ and 
	there exists a unique $\mu\in\CLF(\tilde U)$ such that
	\begin{equation}\label{eq09181818}
	\begin{dcases}
		F(\mu(\tilde x)) = 0 & \forall \tilde x\in \tilde U , \\
		\mu(0,x) = A(\alpha|_x) \,\dd t|_0 + \alpha|_x & \forall x\in U .
	\end{dcases}
	\end{equation}
	Moreover, for all $s\in(-\epsilon,\epsilon)$ and $x\in U$,
	\begin{equation}\label{eq09181817}
	\mu(s,\rho^\alpha_s(x)) = A(\alpha|_x) \,\dd t|_s + \Phi^{-s}_A(\alpha|_x) .
	\end{equation}
\end{proposition}
\begin{proof}
Let us first discuss the existence of a solution to \eqref{eq09181818}.
Let $\cU = (-\epsilon,\epsilon) \times U$. Under our assumptions, the map $\cU  \ni (s,x) \mapsto (s,\rho^\alpha_s(x)) \in \tilde U$ is a diffeomorphism, so \eqref{eq09181817} actually defines a $1$-form $\mu \in \Forms^1(\tilde U)$.

	Define $\psi_0 : U \to T^*\tilde M$ as
	\[
	\psi_0(x) = A(\alpha|_x) \,\dd t|_0 + \alpha|_x \in T^*_{(0,x)} \tilde M ,
	\]
	so that $F\circ\psi_0\equiv0$ by \eqref{eq09181755}.
	Using \eqref{eq09181756}, Lemma~\ref{lem03071514} and the fact $t \circ \psi_0 \equiv 0$, we obtain
	\[
	\psi_0^*\omega_{\tilde M} = \psi_0^*(\dd t\wedge\dd\tau) + \psi_0^*(\tilde\omega_M)
	= \dd(t\circ\psi_0)\wedge\dd(\tau\circ\psi_0) + \alpha^*\omega_M 
	= 0 .
	\]
	
	By Proposition~\ref{prop04121107}, the map $\psi:\cU \to T^*\tilde M$, $\psi(s,x) = \Phi_F^s(\psi_0(x))$ satisfies $F\circ\psi\equiv0$ and $\psi^*\omega_{\tilde M} = 0$.
	Moreover, by \eqref{eq09181758}, 
	\[
	\psi(s,x) = A(\alpha|_x)\dd t|_{-s} + \Phi^s_A(\alpha|_x) = \mu(-s,\rho^\alpha_{-s}(x)) .
	\]
	In other words, $\psi = \mu \circ \Xi$ for some diffeomorphism $\Xi : \cU \to \tilde U$. From $\psi^* \omega_{\tilde M} = 0$ and $F \circ \psi \equiv 0$ we then deduce $\mu^* \omega_{\tilde M} = 0$ and $F \circ \mu \equiv 0$, i.e., $\mu \in \CLF(\tilde U)$ by Lemma \ref{lem03071514}, and $\mu$ solves \eqref{eq09181818}.
	
	As for the uniqueness, assume conversely 	that $\mu \in \CLF(\tilde U)$ solves  \eqref{eq09181818}. Then $F \circ \mu = 0$ and $\mu^*\omega = 0$, i.e., $\mu(\tilde U)$ is a Lagrangian submanifold of $T^* \tilde M$. By Corollary \ref{cor:lagrangian_uniqueness}, $\X_F$ is tangent to $\mu(\tilde U)$ at every point. Fix now $x \in U$ and let $I$ be the set of the $s \in (-\epsilon,\epsilon)$ such that \eqref{eq09181817} holds. Clearly $I$ is closed in $(-\epsilon,\epsilon)$, and $0\in I$ because of \eqref{eq09181818}. On the other hand, for all $s_0 \in I$, the flow curve of $\X_F$ starting from $\mu(s_0,\rho^\alpha_{s_0}(x))$ stays in $\mu(\tilde U)$ for some time and, by \eqref{eq09181758},
	\[
	\Phi_F^t(\mu(s_0,\rho^\alpha_{s_0}(x))) = A(\alpha|_x) \,\dd t|_{s_0-t} + \Phi^{t-s_0}_A(\alpha|_x),
	\]
	which shows that \eqref{eq09181817} also holds for $s$ in a neighbourhood of $s_0$. This proves that $I$ is open, so by connectedness $I = (-\epsilon,\epsilon)$, and \eqref{eq09181817} holds for all $s \in (-\epsilon,\epsilon)$ and all $x \in U$.
\end{proof}

\subsection{Existence domains and smooth dependence on the initial datum}

Proposition \ref{prop09181821} yields, under certain assumptions, the existence of a (local) solution $\mu = \mu^\alpha$ to the eikonal equation for a given initial datum $\alpha \in \CLF(M)$. We will now show how those assumptions can be satisfied and, at the same time, we will obtain suitable smoothness properties of the map $\alpha \mapsto \mu^\alpha$. In what follows, we consider $\CLF(M)$ as a Fr\'echet space with the $C^\infty$ topology (i.e., the topology of uniform convergence on compact sets of derivatives of all orders).

We define an \emph{existence domain (ED)} to be a 
triple $(\Omega,U,\epsilon)$
such that
\begin{enumerate}[label=(\alph*)]
\item $U\subset M$
is open and simply connected, and $\epsilon>0$,
\item $\Omega\subset\CLF(M)$ is 
open in the $C^\infty$ topology,
\item the conditions \ref{prop09181821item1} and \ref{prop09181821item2} of Proposition~\ref{prop09181821}  are satisfied for all $\alpha\in\Omega$.
\end{enumerate}
If $(\Omega,U,\epsilon)$ is an ED, then for all $\alpha \in \Omega$ and $t \in (-\epsilon,\epsilon)$ the inverse $\sigma^\alpha_t$ of $\rho^\alpha_t|_U$ is defined.
In addition, the set $\tilde U^\alpha = \{ (t,\rho^\alpha_t(x)) \tc t \in (-\epsilon,\epsilon), \, x \in U \}$ is open in $\R \times M$ and there is a unique solution $\mu = \mu^\alpha \in \CLF(\tilde U^\alpha)$ to the eikonal equation \eqref{eq09181818}, given by \eqref{eq09181817}. 
We can split $\mu^\alpha=\mu_\R^\alpha \dd t + \mu_M^\alpha$, with $\mu_\R^\alpha: \tilde U^\alpha \to \R$ and $\mu^\alpha_M: \tilde U^\alpha \to T^* M$ smooth.
In view of \eqref{eq09181755}, the eikonal equation~\eqref{eq09181818} becomes
\begin{equation}\label{eq09190902}
\begin{cases}
\mu^\alpha_\R = A \circ \mu^\alpha_M,\\
\mu^\alpha_M|_{(0,x)} = \alpha|_x &\text{for all $x \in U$.}
\end{cases}
\end{equation}
Moreover, by~\eqref{eq09181817}, for all $(t,x) \in \tilde U^\alpha$,
\[
	\sigma^\alpha_t(x) = \pi_M \Phi^t_A(\mu_M^\alpha(t,x)) .
\]

The existence of ED and the smoothness properties of $\alpha \mapsto \mu^\alpha$ are given by the following result. Recall here the definition of the set $\cU^\alpha$ from \eqref{eq:Ualpha}.

\begin{proposition}\label{prp:existenceED}
The following hold true.
\begin{enumerate}[label=(\roman*)]
\item The set $\cU = \{ (\alpha,t,x) \tc \alpha \in \CLF(M), \, (t,x) \in \cU^\alpha\}$ is open in $\CLF(M) \times \R \times M$, and the map
\[
\cU \ni (\alpha,t,x) \mapsto \rho_t^\alpha(x) \in M
\]
is of class $C^\infty$ in the sense of G\^ateaux.
\item For all $\hat\alpha \in \CLF(M)$ and $\hat x \in M$ such that $\hat \alpha|_{\hat x} \in \cD_A$, there exists an ED $(\Omega,U,\epsilon)$ with $\hat\alpha \in \Omega$ and $\hat x \in U$.
\item If $(\Omega,U,\epsilon)$ is an ED, then the set 
\[
\cW = \{ (\alpha,t,\rho^\alpha_t(x)) \tc \alpha \in \Omega, \, t \in (-\epsilon,\epsilon),\, x \in U\}
\]
is open in $\CLF(M) \times \R \times M$ and the maps
\begin{gather*}
\cW \ni (\alpha,t,x) \mapsto \sigma_t^\alpha(x) \in M,\\
\cW \ni (\alpha,t,x) \mapsto \mu^\alpha(t,x) \in T^* \tilde M
\end{gather*}
are of class $C^\infty$ in the sense of G\^ateaux.
\end{enumerate}
\end{proposition}

This result is obtained via an application of the inverse function theorem. Since $\CLF(M)$ is not a Banach space, however, it is convenient to introduce spaces of forms of finite order of differentiability, which are Banach spaces and to which we can apply the inverse function theorem directly.

For $U \subset M$ open and $k \in \N$, let $C^k \Forms^1(U)$ be the space of the $1$-forms of class $C^k$ on $U$, i.e., the sections of class $C^k$ of the bundle $T^* U$. In the case $U \Subset M$, we also denote by $C^k \Forms^1(\overline{U})$ the space of the $\alpha \in C^k \Forms^1(U)$ that extend continuously to $\overline{U}$ together with all their derivatives up to order $k$. Note that $C^k \Forms^1(\overline{U})$ is a Banach space with the uniform $C^k$ topology (the topology of uniform convergence of all derivatives up to order $k$). Note that \eqref{eq:def_rhoalpha} defines $\rho^\alpha$ also for $\alpha \in C^0 \Forms^1(U)$.

Proposition \ref{prp:existenceED} is then an immediate consequence of the following result.

\begin{lemma}\label{lem:edom}
Let $U \Subset M$ be open in $M$ and define
\[
\cU_k = \{ (\alpha,t,x) \tc \alpha \in C^k\Forms^1(\overline{U}), \, (t,x) \in \cU^\alpha\}
\]
for all $k \in \N$. Then:
\begin{enumerate}[label=(\roman*)]
\item\label{en:edom1} $\cU_k$ is an open neighbourhood of $C^k\Forms^1(\overline{U}) \times \{0\} \times U$ in $C^k\Forms^1(\overline{U}) \times \R \times U$;
\item\label{en:edom2} the map
\[
\cU_k \ni (\alpha,t,x) \mapsto \rho^\alpha_t(x) \in M
\]
is of class $C^k$.
\end{enumerate}
Moreover, for all $(\hat\alpha,\hat t,\hat x) \in \cU_1$ such that $\DD\rho^{\hat \alpha}_{\hat t}|_{\hat x}$ is invertible, there exist an open neighbourhood $\Omega$ of $\hat\alpha$ in $C^1\Forms^1(\overline{U})$, an open interval $I \subset \R$ containing $\hat t$, and an open neighbourhood $W$ of $\hat x$ in $U$ such that:
\begin{enumerate}[resume,label=(\roman*)]
\item\label{en:edom3} $I \times W \subseteq \cU^\alpha$ for all $\alpha \in \Omega$;
\item\label{en:edom4} $\rho^\alpha_t|_W : W \to M$ is a $C^1$ embedding 
for all $\alpha \in \Omega$ and $t \in I$;
\item\label{en:edom5} $\cW \defeq \{ (\alpha,t,\rho^\alpha_t(x)) \tc \alpha \in \Omega, \, t \in I, \, x \in W \}$ is open in $C^1\Forms^1(\overline{U})$ and moreover,
if $\sigma^\alpha_t$ denotes the inverse of $\rho^\alpha_t|_W$, then the map
\[
\cW \ni (\alpha,t,x) \mapsto \sigma^\alpha_t(x) \in M
\]
is of class $C^1$, and its restriction to $\cW \cap (C^k\Forms^1(\overline{U}) \times \R \times M)$ is of class $C^k$ (with respect to the uniform $C^k$ topology) for all $k > 1$;
\item\label{en:edom6} if $\cW^\alpha = \{ (t,x) \tc (\alpha,t,x) \in \cW\}$ and $\mu^\alpha \in C^1\Forms^1(\cW^\alpha)$ is defined by
\[
\mu^\alpha(s,\rho^\alpha_s(x)) = A(\alpha|_x) \,\dd t|_s + \Phi^{-s}_A(\alpha|_x) .
\]
for all $s \in I$ and $x \in W$, then the map
\[
\cW \ni (\alpha,t,x) \mapsto \mu^\alpha(t,x) \in T^* \tilde M
\]
is of class $C^1$, and its restriction to $\cW \cap (C^k\Forms^1(\overline{U}) \times \R \times M)$ is of class $C^k$ (with respect to the uniform $C^k$ topology) for all $k > 1$.
\end{enumerate}
\end{lemma}
\begin{proof}
Parts \ref{en:edom1} and \ref{en:edom2} are immediate consequences of the observation that
\[
\rho^\alpha_t(x) = \pi_M \Phi^{-t}_A(\Ev(\alpha,x)),
\]
where $\Ev(\alpha,x) = \alpha|_x$ is the evaluation map, and the fact that
\[
\Ev : C^k\Forms^1(\overline{U}) \times U \to T^* M
\]
is of class $C^k$ for all $k \in \N$. As a consequence, the map
\[
\Psi_k : \cU_k \ni (\alpha,t,x) \mapsto (\alpha,t,\rho^\alpha_t(x)) \in C^k\Forms^1(\overline{U}) \times \R \times M
\]
is also of class $C^k$, and,
if $k \geq 1$ it is easily checked that $D\Psi_k|_{(\alpha,t,x)}$ is continuously invertible for all $(\alpha,t,x) \in \cU_k$ such that $\DD \rho^\alpha_t|_x$ is invertible. Hence parts \ref{en:edom3}, \ref{en:edom4} and \ref{en:edom5} follow by applying the  inverse function theorem to $\Psi_1$, and observing that restrictions of a local inverse for $\Psi_1$ provide local inverses for all the $\Psi_k$ for $k > 1$. Finally, part \ref{en:edom6} follows by observing that $(\alpha,t,x) \mapsto \mu^\alpha(t,x)$ is the composition of the maps $(\alpha,t,x) \mapsto (\alpha,t,\sigma^t_\alpha(x))$ and
\[
(\alpha,s,x) \mapsto A(\Ev(\alpha,x)) \,\dd t|_s + \Phi_A^{-s}(\Ev(\alpha,x)),
\]
which have the required smoothness properties.
\end{proof}

We say that $A : \cD_A \to \R$ is $1$-homogeneous if $\lambda\xi \in \cD_A$ and $A(\lambda \xi) = \lambda A(\xi)$ for all $\xi \in \cD_A$ and $\lambda > 0$. When $A$ is $1$-homogeneous, we can find an ED $(\Omega,U,\epsilon)$ such that the set $\Omega \subset \CLF(M)$ is conic; such ED will be called \emph{conic existence domains} (CED).

\begin{proposition}\label{prp:ced}
Assume that $A$ is $1$-homogeneous.
\begin{enumerate}[label=(\roman*)]
\item\label{en:ced1} For all $t \in \R$, the domain of $\Phi_A^t$ is conic and
\[
\Phi_A^t(\lambda\xi) = \lambda \Phi_A^t(\xi)
\]
for all $\lambda > 0$ and $\xi$ in the domain of $\Phi_A^t$.
\item\label{en:ced2} For all $(\alpha,t,x) \in \cU$ and $\lambda >0$, we have $(\lambda\alpha,t,x) \in \cU$ and
\[
\rho^{\lambda\alpha}_t(x) = \rho^{\alpha}_t(x).
\]
\item If $(\Omega,U,\epsilon)$ is an ED and $\R^+ \Omega = \{ \lambda \alpha \tc \alpha \in \Omega, \, \lambda > 0\}$, then $(\R^+ \Omega,U,\epsilon)$ is a CED.
\item If $(\Omega,U,\epsilon)$ is a CED, then $\tilde U^{\lambda \alpha} = \tilde U^\alpha$ and
\[
\mu^{\lambda\alpha}(t,x) = \lambda \mu^\alpha(t,x)
\]
for all $\lambda>0$, $\alpha \in \Omega$ and $(t,x) \in \tilde U^\alpha$.
\end{enumerate}
\end{proposition}
\begin{proof}
Since $A$ is $1$-homogeneous, in local canonical coordinates $(x,\eta)$ we have
	\[
	\frac{\de A}{\de x_j}(x,r\eta) = r \frac{\de A}{\de x_j}(x,\eta) 
	\qquad\text{and}\qquad
	\frac{\de A}{\de \eta_j}(x,r\eta) =  \frac{\de A}{\de \eta_j}(x,\eta),
	\]
and part \ref{en:ced1} easily follows from \eqref{eq09190853}. The remaining statements are immediate consequences of part \ref{en:ced1}, the definition of CED and the expression \eqref{eq09181817} for $\mu^\alpha$.
\end{proof}

Once we have an ED $(\Omega,U,\epsilon)$, we can take the G\^ateaux derivative of $\alpha\mapsto \mu^\alpha$ at any $\alpha\in\Omega$.
If $\nu\in\CLF(M)$, then $\de_\alpha\mu^\alpha[\nu]\in\CLF(\tilde U^\alpha)$ is the $1$-form defined by 
\[
\de_\alpha\mu^\alpha[\nu](t,x) =  \left. \frac{\dd}{\dd h}\right|_{h=0} \mu^{\alpha+h\nu}(t,x) \in T^*_{(t,x)}\tilde M .
\]
Notice that 
\[
\de_\alpha\mu^\alpha[\nu] = \de_\alpha\mu^\alpha_\R[\nu] \,\dd t + \de_\alpha \mu^\alpha_M[\nu] ,
\]
where $\de_\alpha\mu^\alpha_\R$ and $\de_\alpha\mu^\alpha_M$ are defined as G\^ateaux derivatives of $\mu^\alpha_\R$ and $\mu^\alpha_M$.

We now obtain a useful identity for  $\de_\alpha\mu^\alpha[\nu]$ that follows from the eikonal equation.
Define, for all $x \in M$ and $\alpha,\beta \in T^*_xM$ with $\alpha \in \cD_A$,
\begin{equation}\label{eq:vert_diff}
\DD_2A|_\alpha[\beta] \defeq \left.\frac{\dd}{\dd h} \right|_{h=0} A(\alpha+h\beta) .
\end{equation}
Essentially, $\DD_2A|_\alpha$ is  the restriction of $\dd A$ to the ``vertical'' directions in the cotangent bundle $T^*M$. 
Using canonical coordinates and \eqref{eq09190852}, it is immediately seen that
\begin{equation}\label{eq09190927}
\DD_2A|_\alpha[\beta] = \beta[\DD\pi_M[\X_A|_\alpha]].
\end{equation}

\begin{lemma}\label{lem03131122}
	For every $\alpha\in\Omega$, $\nu\in\CLF(M)$ and $(t,x)\in\tilde U^\alpha$,
	\[
	\de_\alpha\mu^\alpha_\R[\nu](t,x) 
	= \de_\alpha\mu^\alpha_M[\nu](t,x) \left[\DD \pi_M\left[\X_A|_{\mu^\alpha_M(t,x)}\right] \right].
	\]
\end{lemma}
\begin{proof}
	Using the ``generalized  eikonal equation'' \eqref{eq09190902},
	\begin{align*}
	\de_\alpha\mu^\alpha_\R[\nu](t,x)
	&= \left.\frac{\dd}{\dd h}\right|_{h=0} A(\mu^{\alpha+h\nu}_M(t,x)) \\
	&= \DD_2A|_{\mu^\alpha_M(t,x)}\left[ \de_\alpha\mu^\alpha_M[\nu](t,x) \right] ,
	\end{align*}
	by the chain rule, and the conclusion follows by \eqref{eq09190927}.
\end{proof}

\subsection{Definition of the phase function}\label{sec10030007}

Let $(\Omega,U,\epsilon)$ be an ED.
For all $\alpha \in \Omega$, $\beta \in \CLF(\tilde U^\alpha)$ and $\tilde x,\tilde y \in \tilde U^\alpha$, we denote by $\int^{\tilde x}_{\tilde y} \beta$ the integral of $\beta$ along any path in $\tilde U^\alpha$ joining $\tilde y$ to $\tilde x$; since $U$ is simply connected, $\tilde U^\alpha$ is too and the value of the integral does not depend on the chosen path. Similarly we define $\int_x^y \nu$ for all $\nu \in \CLF(M)$ and $x,y\in U$.

We define the open set $\cD_\phase = \{ (t,x,y,\alpha) \tc \alpha \in \Omega,\, (t,x) \in \tilde U^\alpha, \, y \in U\}$ and  the ``phase function''
$\phase : \cD_\phase \to \R$ associated with $(\Omega,U,\epsilon)$ by
\begin{equation}\label{eq09191349}
 \phase(t,x,y,\alpha) = \int_{(0,y)}^{(t,x)} \mu^\alpha.
\end{equation}
Note that, since $d_{(t,x)}\phase=\mu^\alpha$, for all $(t,x,y,\alpha) \in \cD_\phase$ we have 
\begin{equation}\label{eq:phiderivatives}
\mu^\alpha_\R(t,x) = \partial_t \phase(t,x,y,\alpha) \, \qquad \mu^\alpha_M(t,x) = \partial_x \phase(t,x,y,\alpha),
\end{equation}
hence, by \eqref{eq09190902}, $\phase$ is a solution to 
\begin{equation}\label{eq:eikonal_phi}
\begin{dcases}
\de_t\phase(t,x,y,\alpha) = A(\de_x\phase(t,x,y,\alpha)) , & (t,x,y,\alpha) \in \cD_\phase, \\
\phase(0,x,y,\alpha) = \textstyle\int_{y}^{x} \alpha , & (x,y,\alpha) \in U \times U \times \Omega.
\end{dcases}
\end{equation}
In particular
\begin{gather}
\label{eq:domainx} \partial_x\phase(t,x,y,\alpha) \in \cD_A, \\
\label{eq:domainy}  \partial_y \phase(t,x,y,\alpha) = -\alpha|_y \in -\cD_A, \\
\label{eq:derivt} \partial_t\phase(t,x,y,\alpha) = A(\alpha|_{\sigma^\alpha_t(x)})
\end{gather}
for all $(t,x,y,\alpha) \in \cD_\phase$; the last identity follows from \eqref{eq09181817} (applied with $\sigma^\alpha_t(x)$ in place of $x$) and \eqref{eq:phiderivatives}.

\subsection{Critical points of the phase function}
We want to differentiate in $\alpha$ the phase function $\phase$ associated to an ED $(\Omega,U,\epsilon)$ and characterise the points $(t,x,y,\alpha)$ such that $\de_\alpha\phase(t,x,y,\alpha)=0$.

The G\^ateaux derivative of $\phase$ is, for 
$(t,x,y,\alpha)\in \cD_\phase$ 
and $\nu\in\CLF(U)$, 
\begin{equation}\label{eq09190933}
\de_\alpha\phase(t,x,y,\alpha) [\nu]
	= \int_{(0,y)}^{(t,x)} \de_\alpha\mu^\alpha[\nu].
\end{equation}
From the fact that $\mu^\alpha$ solves the eikonal equation, we can deduce a simpler expression for the G\^ateaux derivative of $\phase$.

\begin{proposition}[Tr\`eves]\label{prp:treves}
	For all $(t,x,y,\alpha) \in \cD_\phase$ and $\nu \in \CLF(M)$,
	\begin{equation}\label{eq09190934}
 \de_\alpha\phase(t,x,y,\alpha) [\nu] = \int_y^{\sigma_t^\alpha(x)} \nu.
	\end{equation}
\end{proposition}
\begin{proof}
	By \eqref{eq:def_rhoalpha} and  \eqref{eq09181817},
	\begin{align*}
	\frac{\dd}{\dd t}\rho^\alpha_t(x)
	&= \frac{\dd}{\dd t}\pi_M(\Phi_A^{-t}(\alpha|_x))
	= \DD\pi_M\left[\frac{\dd}{\dd t}\Phi_A^{-t}(\alpha|_x)\right] \\
	&= - \DD\pi_M\left[\X_A|_{\Phi_A^{-t}(\alpha|_x)}\right]
	= - \DD\pi_M \left[\X_A|_{\mu_M^\alpha(t,\rho_t^\alpha(x))}\right] .
	\end{align*}
	Hence,
	for all $\alpha \in \Omega$, $t \in (-\epsilon,\epsilon)$, $x,y \in U$, and $\nu\in \CLF(M)$,
	\begin{align*}
	\frac{\dd}{\dd t} &\left[ (\de_\alpha \phase) (t,\rho^\alpha_t(x),y,\alpha) [\nu] \right]\\
	&= (\de_\alpha\de_t\phase)  (t,\rho^\alpha_t(x),y,\alpha) [\nu]
		+ (\de_\alpha\de_x\phase)  (t,\rho^\alpha_t(x),y,\alpha)[\nu]  \left[ \frac{\dd}{\dd t}\rho^\alpha_t(x) \right]\\
	&= (\de_\alpha \mu_\R^\alpha)[\nu](t,\rho^\alpha_t(x)) 
		- (\de_\alpha \mu_M^\alpha)[\nu](t,\rho^\alpha_t(x))  \left[\DD\pi_M \left[\X_A|_{\mu_M^\alpha(t,\rho_t^\alpha(x))}\right] \right]\\
	&= 0 .
	\end{align*}
	by \eqref{eq:phiderivatives} and Lemma~\ref{lem03131122} (compare \cite[Eq.\ (2.31)]{MR597145}).
	Consequently
\[
(\de_\alpha\phase)(t,\rho_t^\alpha(x),y,\alpha) [\nu] = (\de_\alpha\phase)(0,x,y,\alpha) [\nu] = \int_y^x  \partial_\alpha\alpha [\nu] =  \int_y^x  \nu
\]
by \eqref{eq:eikonal_phi}, and \eqref{eq09190934} follows by replacing $x$ with $\sigma^\alpha_t(x)$.
\end{proof}

We say that $\cV \subset \CLF(M)$ is \emph{separating} for $U$ if, for all $x,y \in U$ with $x \neq y$, there exist $\nu \in \cV$ and $f \in C^\infty(U)$ such that $f(x) \neq f(y)$ and $\dd f = \nu|_U$. 

The following result, which relates the critical points of the phase $\phase$ to the geodesic flow, is an immediate consequence of Proposition \ref{prp:treves} and Stokes' theorem.

\begin{corollary}\label{cor:phase_critical}
	Let $(\Omega,U,\epsilon)$ be an ED and let $\cV \subset \CLF(M)$ be separating for $U$.
	Then, for all $(t,x,y,\alpha)\in\cD_\phase$,
	\[
	\de_\alpha\phase(t,x,y,\alpha)|_{\cV} = 0
	\quad\iff\quad
	x=\rho^\alpha_t(y) .
	\]
\end{corollary}

In the case $A$ is $1$-homogeneous the above expression for $\partial_\alpha \phase$ actually yields a corresponding expression for $\phase$.

\begin{proposition}\label{prp:hom_phase}
If $A$ is $1$-homogeneous and $(\Omega,U,\epsilon)$ is a CED, then the associated phase function $\phase$ is $1$-homogeneous in $\alpha$, i.e., $(t,x,y,\lambda \alpha) \in \cD_\phase$ and
\[
\phase(t,x,y,\lambda \alpha) = \lambda \phase(t,x,y,\alpha)
\]
for all $\lambda > 0$ and $(t,x,y,\alpha) \in \cD_\phase$.
In addition, for all $(t,x,y,\alpha) \in \cD_\phase$,
\[
\phase(t,x,y,\alpha) = \int_y^{\sigma_t^\alpha(x)} \alpha.
\]
\end{proposition}
\begin{proof}
Homogeneity of $\phase$ in $\alpha$ immediately follows from Proposition \ref{prp:ced} and \eqref{eq09191349}. From this we deduce that
\[
\phase(t,x,y,\alpha) = \left.\frac{\dd}{\dd \lambda}\right|_{\lambda=1} \phase(t,x,y,\lambda\alpha) = \partial_\alpha \phase(t,x,y,\alpha)[\alpha],
\]
which, together with Proposition \ref{prp:treves}, gives the desired expression for $\phase$.
\end{proof}

\subsection{The Hessian of the phase function}\label{sec11171048}
Let $(\Omega,U,\epsilon)$ be an ED and $\phase : \cD_\phase \to\R$ be the associated phase function.
The G\^ateaux-Hessian in $\alpha$ of $\phase$ at $(t,x,y,\alpha) \in \cD_\phase$ is the symmetric bilinear map
\[
\begin{array}{rccl}
\de_\alpha^2\phase(t,x,y,\alpha):&\CLF(M)\times\CLF(M) &\to &\R \\
 &(\nu_1,\nu_2) &\mapsto &\left.\frac{\dd}{\dd h}\right|_{h=0} \de_\alpha\phase(t,x,y,\alpha+h\nu_2)[\nu_1] .
\end{array}
\]

We now obtain an expression for the Hessian $\partial_\alpha^2 \phase$ in terms of the Hamilton flow on $T^* M$ associated to $A$ (or rather its projection to the manifold $M$).

Note that $\Exp_{A}^{x,t}$, defined in~\eqref{eq11201112}, is a smooth map defined on a (possibly empty) open subset of $T^*_x M$ for all $x \in M$ and $t\in \R$.
Since $T^*_x M$ is a vector space, the tangent space $T_\xi T^*_x M$ is canonically identified with $T^*_x M$ at each point $\xi \in T^*_x M$, so we can think of $\DD \Exp_A^{x,t}|_\xi$ as a linear map $T^*_x M \to T_{\Exp_A^{x,t}(\xi)} M$.
 Note also that
\begin{equation}\label{eq:Exprho}
\Exp_{A}^{x,-t}(\alpha|_x) = \rho^\alpha_t(x)
\end{equation}
for all $\alpha \in \Forms^1(W)$, $x \in W$, $t \in \R$ such that $(-t,\alpha|_x)$ is in the domain of $\Phi_A$.

\begin{proposition}\label{prp:hessian_formula}
	For all $(t,x,y,\alpha)\in \cD_\phase$ and $\nu_1,\nu_2\in\CLF(M)$, 
	\[
	\de_\alpha^2\phase(t,x,y,\alpha)[\nu_1,\nu_2] = - \nu_1|_{\sigma^\alpha_t(x)} [\DD\sigma_t^\alpha|_{x} [ \DD \Exp_A^{\sigma_t^{\alpha}(x),-t}|_{\alpha|_{\sigma_t^{\alpha}(x)}}[\nu_2|_{\sigma_t^{\alpha}(x)}]]] .
	\]
\end{proposition}
\begin{proof}
	Note that, by \eqref{eq09190934},
	\begin{align*}
	\de_\alpha^2\phase(t,x,y,\alpha)[\nu_1,\nu_2]
  &= \left.\frac{\dd}{\dd h}\right|_{h=0} \int_y^{\sigma_t^{\alpha+h\nu_2}(x)} \nu_1 \\
	&= \nu_1|_{\sigma^\alpha_t(x)} \left[\left.\frac{\dd}{\dd h}\right|_{h=0} \sigma_t^{\alpha+h\nu_2}(x)\right] .
	\end{align*}
	
	On the other hand, for all $\nu \in \CLF(M)$, since $\rho_t^{\alpha+h\nu}(\sigma_t^{\alpha+h\nu}(x))=x$ for all  small enough $h\in\R$, by \eqref{eq:Exprho},
	\begin{align*}
	0 
	&= \left.\frac{\dd}{\dd h}\right|_{h=0} \rho_t^{\alpha+h\nu}(\sigma_t^{\alpha+h\nu}(x)) \\
	&= \left.\frac{\dd}{\dd h}\right|_{h=0} \rho_t^{\alpha+h\nu}(\sigma_t^{\alpha}(x))
		+ \left.\frac{\dd}{\dd h}\right|_{h=0} \rho_t^{\alpha}(\sigma_t^{\alpha+h\nu}(x)) \\
	&= \DD \Exp_A^{\sigma_t^{\alpha}(x), -t}|_{\alpha|_{\sigma_t^{\alpha}(x)}}[\nu|_{\sigma_t^{\alpha}(x)}]
		+ \DD\rho^\alpha_t|_{\sigma_t^{\alpha}(x)}\left[ \left.\frac{\dd}{\dd h}\right|_{h=0} \sigma^{\alpha+h\nu}_t(x) \right] ,
	\end{align*}
	so
	\[
	\left.\frac{\dd}{\dd h}\right|_{h=0} \sigma^{\alpha+h\nu}_t(x) = - \DD\sigma_t^\alpha|_{x} [ \DD \Exp_A^{\sigma_t^{\alpha}(x),-t}|_{\alpha|_{\sigma_t^{\alpha}(x)}}[\nu|_{\sigma_t^{\alpha}(x)}]],
	\]
	and we are done.
\end{proof}

An interesting consequence of the above formula is that $\de_\alpha^2\phase(t,x,y,\alpha)[\nu_1,\nu_2]$ only depends on the values of $\nu_1$ and $\nu_2$ at the point $\sigma^\alpha_t(x)$, and therefore $\de_\alpha^2\phase(t,x,y,\alpha)$ is effectively a bilinear form on the finite-dimensional space $T_{\sigma^\alpha_t(x)}^* M$.

We say that a linear subspace $\cV$ of $\CLF(M)$ is \emph{spanning} for $U$ if  $\{\nu|_x \tc \nu\in \cV\} = T^*_x M$ for all $x \in U$. The following result is an immediate consequence of Proposition \ref{prp:hessian_formula} and Corollary \ref{cor:phase_critical}.

\begin{corollary}\label{cor:phase_rank}
	Let $\cV\subset\CLF(M)$ be a linear subspace that is spanning for $U$.
	Then, for all $(t,x,y,\alpha) \in \cD_\phase$,
	\begin{equation}\label{eq:hessian_rank_gen}
		\rank(\de_\alpha^2\phase (t,x,y,\alpha)|_{\cV \times \cV}) = 
		\rank(\DD\Exp_{A}^{{\sigma^\alpha_t(x)},-t}|_{\alpha|_{{\sigma^\alpha_t(x)}}}).
	\end{equation}
	In particular, if $\cV$ is separating for $U$ and $\de_\alpha\phase(t,x,y,\alpha)|_{\cV}=0$, then
	\begin{equation}\label{eq09191629}
		\rank(\de_\alpha^2\phase (t,x,y,\alpha)|_{\cV \times \cV}) = 
		\rank(\DD\Exp_{A}^{y,-t}|_{\alpha|_y}).
	\end{equation}
\end{corollary}

\subsection{Construction of an operator phase function}
Recall that, for all ED $(\Omega,U,\epsilon)$, the set $\{ (\alpha,t,\rho^\alpha(x)) \tc \alpha \in \Omega, \, t \in (-\epsilon,\epsilon), \, x \in U\}$ is an open neighbourhood of $\Omega \times \{0\} \times U$ in $\CLF(M) \times \R \times M$. A simple compactness argument yields the following strengthening of the existence result for ED in Proposition \ref{prp:existenceED}.

\begin{lemma}\label{lem:compactED}
Let $K$ and $\Theta$ be compact subsets of $M$ and $\CLF(M)$ such that $\alpha|_x \in \cD_A$ for all $\alpha \in \Theta$ and $x \in K$. Then there exists an ED $(\Omega,U,\epsilon)$ such that $\Theta \subset \Omega$ and $K \subset \interior(\bigcap_{\alpha\in\Omega,|t|<\epsilon} \rho^\alpha_t(U))$.
\end{lemma}

We can now prove our main result.

\begin{proposition}\label{prp:operatorphaseeikonal}
Assume that $A$ is $1$-homogeneous and $\cD_A = \nT^* M$.
Let $o \in M$ and $(W,x)$ be any system of local coordinates for $M$ at $o$. Then there exists an open neighbourhood $V \subset W$ of $o$, an $\epsilon > 0$ and a smooth function $w : (-\epsilon,\epsilon)  \times V \times \nR^n \to \R$, where $n=\dim M$, with the following properties.
\begin{enumerate}[label=(\roman*)]
\item\label{en:operatorphaseeikonal_ders} 
$w$ is $1$-homogeneous in $\xi$ and, for all $(t,x,\xi) \in (-\epsilon,\epsilon) \times V \times \nR^n$,
\[
w(0,x,\xi) = x \cdot \xi, \qquad \partial_x w(t,x,\xi) \neq 0 .
\]
\item\label{en:operatorphaseeikonal_structure}
The function $\phi : (-\epsilon,\epsilon) \times V \times V \times \nR^n \to \R$, 
\[
\phi(t,x,y,\xi) = w(t,x,\xi) - w(0,y,\xi),
\]
is a phase function that solves the eikonal equation~\eqref{eq09121416}.
Moreover, $(x,y,\xi) \mapsto \phi(t,x,y,\xi)$ is an operator phase function for all $t \in (-\epsilon,\epsilon)$.
\item\label{en:operatorphaseeikonal_critical} 
For all $(t,x,y,\xi) \in (-\epsilon,\epsilon) \times V \times V \times \nR^n$,
\[
\partial_\xi w(t,x,\xi) = y \quad \iff \quad \Exp_{A}^{y,-t}(\xi \cdot \dd x|_y) = x
\]
and in that case
\[
\partial_t w(t,x,\xi) = A(\xi \cdot dx|_y).
\]
\item\label{en:operatorphaseeikonal_hessian} For all $(t,x,y,\xi) \in (-\epsilon,\epsilon) \times V \times V \times \nR^n$ such that $\partial_\xi w(t,x,\xi) = y$,
\[
\rank(\partial_\xi^2 w(t,x,\xi)) = \rank(\DD\Exp_{A}^{y,-t}|_{\xi \cdot \dd x|_y}).
\]
\end{enumerate}
Here $\xi \cdot \dd x \in \Forms^1(W)$ is the form $\sum_j \xi_j \, \dd x_j$ in the coordinates $(W,x)$ for all $\xi \in \R^n$.
\end{proposition}
\begin{proof}
Without loss of generality we may assume that $M$ is an open subset of $\R^n$. Let $\cV \subset \CLF(M)$ to be the $\R$-linear span of $\dd x_1,\dots,\dd x_n$, so clearly $\cV$ is both separating and spanning for $M$. Let $S$ be the unit sphere in $\cV$ (corresponding to the choice of $\dd x_1,\dots,\dd x_n$ as an orthonormal basis). Then $S$ is a compact subset of $\CLF(M)$. Since $\alpha|_o \in \nT^* M$ for all $\alpha \in S$, by Lemma \ref{lem:compactED} and Proposition \ref{prp:ced} we can find a CED $(\Omega,U,\epsilon)$ such that $S \subset \Omega$ and $o \in V \defeq \interior(\bigcap_{\alpha\in\Omega,|t|<\epsilon} \rho^\alpha_t(U))$. We can now define a smooth function $\phi : \R \times V \times V \times \nR^n \to \R$ by
\[
\phi(t,x,y,\xi) = \phase(t,x,y, \xi \cdot \dd x),
\]
where $\phase$ is the ``phase function'' associated to $(\Omega,U,\epsilon)$ defined in \eqref{eq09191349}, while $\xi \cdot \dd x = \sum_j \xi_j \dd x_j$.

Then, by \eqref{eq:eikonal_phi}, $\phi$ solves the eikonal equation, and $\partial_x \phi$ and $\partial_y \phi$ vanish nowhere by \eqref{eq:domainx} and \eqref{eq:domainy}. From Proposition \ref{prp:hom_phase} we deduce that $\phi$ is $1$-homogeneous in $\xi$, and 
\[
\phi(t,x,y,\xi) = \xi \cdot (\sigma_t^{\xi \cdot \dd x}(x) - y) = w(t,x,\xi) - w(0,y,\xi),
\]
where $w(t,x,\xi) = \xi \cdot \sigma_t^{\xi \cdot \dd x}(x)$. 
This shows \ref{en:operatorphaseeikonal_ders} and \ref{en:operatorphaseeikonal_structure}.
Moreover Corollary \ref{cor:phase_critical}
and \eqref{eq:Exprho} show that
\[
\partial_\xi \phi(t,x,y,\xi) = 0 \quad \iff \quad x = \rho^{\xi \cdot \dd x}_t(y) \quad \iff \quad x = \Exp_A^{y,-t}(\xi \cdot \dd x|_y),
\]
and in that case $\partial_t \phi(t,x,y,\xi) = A(\xi \cdot dx|_y)$ by \eqref{eq:derivt}.
This shows \ref{en:operatorphaseeikonal_critical}, because $\de_\xi\phi(t,x,y,\xi)=\de_\xi w(t,x,\xi)-y$ and $\de_t\phi=\de_tw$.
Moreover Corollary \ref{cor:phase_rank} gives that
\[
\partial_\xi \phi(t,x,y,\xi) = 0 \quad \Longrightarrow \quad \rank (\partial_\xi^2 \phi(t,x,y,\xi)) = \rank (\DD \Exp_{A}^{y,-t}|_{\xi \cdot \dd x|_y}),
\]
and we are done, because $\partial_\xi^2 \phi(t,x,y,\xi)=\partial_\xi^2 w(t,x,\xi)$.
\end{proof}

\section{Sub-Laplacians on sub-Riemannian manifolds}\label{sec:subriemannian}

In this section we recall the main definitions and results about sub-Riemannian manifolds and sub-Laplacians that will be of use later, and show how the results from the previous sections yield a Fourier integral representation for the sub-Riemannian wave propagator. For a more extensive introduction to sub-Riemannian geometry, we refer to \cite{MR1867362,abb_introduction_2018}.

\subsection{The sub-Riemannian Hamiltonian}

If $\cH\subset \Sect(TM)$ is a linear subspace of vector fields on a manifold $M$ and if $x\in M$, then we denote by $\cH_x\subset T_xM$ the space $\{v|_x \tc v \in \cH\}$. 
If $U\subset M$, we write $\cH_U=\bigcup_{x\in U} \cH_x$.
We define inductively on $k\in\N$ the spaces $\cH^{(k)} \subset \Sect(TM)$ as $\cH^{(1)}=\cH$ and $\cH^{(k+1)}=\cH^{(k)}+[\cH,\cH^{(k)}]$.
Then $\cH$ is said to be \emph{bracket-generating} at $x\in M$ if there is an $s\in\N$ such that $\cH^{(s)}_x=T_xM$.
We say that $\cH \subset \Sect(TM)$ is bracket-generating on $M$ if it is bracket-generating at each $x\in M$.
More generally, a subset of $\Sect(TM)$ is said to be bracket-generating at $x$ (respectively on $M$), if its linear span is bracket-generating at $x$ (respectively on $M$).

\begin{definition}
	We call $(M,H)$ a \emph{quadratic Hamiltonian pair} if $M$ is a smooth manifold and $H:T^*M\to[0,\infty)$ is a smooth map, called \emph{Hamiltonian}, 
	such that
	the restriction of $H$ to $T_x^*M$ is a homogeneous quadratic form, for all $x\in M$.
		If the space $\cH$ of \emph{horizontal vector fields} for $H$, defined by
	\[
	\cH \defeq \left\{v\in \Sect(TM) \tc \forall\alpha\in T^*M \tc \left( H(\alpha)=0 \,\, \Rightarrow \,\, \alpha(v)=0 \right) \right\},
	\]
	is bracket-generating, then we call $(M,H)$ a \emph{bracket-generating quadratic Hamiltonian pair} or \emph{sub-Riemannian manifold}.
\end{definition}

Equivalently,
a quadratic Hamiltonian pair $(M,H)$ is defined by a smooth positive semidefinite section $b_H$ of the vector bundle of symmetric bilinear forms on $T^*M$, given by
\[
b_H(\alpha,\beta) = \frac12 \left(H(\alpha+\beta)-H(\alpha)-H(\beta) \right) .
\]
for all $x \in M$ and $\alpha,\beta \in T^*_x M$.
The Hamiltonian $H$ also induces a bundle homomorphism $B_H:T^*M\to TM$ defined by the property
\[
\alpha[B_H\beta] = b_H(\alpha,\beta)
\]
for all $x \in M$ and $\alpha,\beta \in T^*_x M$.
Notice that $B_H(T^*M)=\cH_M$.

The push forward of $b_H$ through $B_H$ is a scalar product on $\cH_x$, for each $x\in M$, which is given by
\begin{equation}\label{eq:innerproduct}
\langle B_H\alpha,B_H\beta \rangle_H = b_H(\alpha,\beta) ,
\qquad \forall\alpha,\beta\in T^*_xM .
\end{equation}

The \emph{horizontal gradient} of a smooth real-valued function $f$ on $M$ is the real vector field $\nabla_H f = B_H (\dd f) \in \cH$.
Notice that, for all $\alpha \in T^*M$, 
\[
\alpha[\nabla_H f]
= b_H(\alpha, d f) 
= \langle B_H \alpha, \nabla_H f \rangle_H.
\]

In the sequel we shall mainly work with complex-valued functions on $M$ and, correspondingly, we often make use of the complexified tangent and cotangent bundles $\C TM$ and $\C T^*M$. The map $B_H$ extends to a complex-linear bundle homomorphism $B_H : \C T^* M \to \C T M$, while $b_H$ and $\langle \cdot,\cdot\rangle_H$ extend to sesquilinear forms on the fibres of $\C T^*M$ and $\C\cH$ respectively. The horizontal gradient $\nabla_H$ extends to a complex-linear first-order differential operator $\nabla_H : C^\infty(M) \to \Sect(\C TM)$.

In a coordinate chart $(U,x)$ of $M$ and in the corresponding local trivialization $(T^*U,(x,\xi))$ of $T^*M$, we have $H\left( \sum_j \alpha_j\dd x^j\right) = \sum_{jk} H^{jk} \alpha_j\alpha_k$ and $B_H(\alpha) = \sum_{k} \left(\sum_j H^{jk}\alpha_j\right) \de_k$, where $H^{jk}:U\to\R$ are smooth functions and $H^{jk}=H^{kj}$.
Moreover, the horizontal gradient of a smooth function $f$ is
\[
\nabla_H f = \sum_k \left(\sum_jH^{jk}\de_jf\right) \de_k .
\]

\subsection{The sub-Laplacian and its functional calculus}

A measure $\mu$ on a manifold $M$ is a 
\emph{smooth positive  measure} 
if for every coordinate chart $(U,x)$ the restricted measure is absolutely continuous with respect to the Lebesgue measure and  has a strictly positive smooth density. If $(M,H)$ is a Hamiltonian pair and $\mu$ is a smooth positive measure on $M$, we call $(M,H,\mu)$ a \emph{measured quadratic Hamiltonian pair}.

\begin{definition}\label{dfn:divergence}
	Let $\mu$ be a smooth positive measure on a manifold $M$.
	The $\mu$-\emph{divergence} of a smooth vector field $v\in \Sect(\C TM)$ is the unique smooth function $\div_\mu v\in C^\infty(M)$ such that
	\[
	\int_M \dd\phi [\overline{v} ] \,\dd\mu = - \int_M \phi \, \overline{\div_\mu v} \, \dd\mu ,
	\qquad\forall\phi\in C^\infty_c(M) .
	\]
\end{definition}
In other words, minus the $\mu$-divergence $-\div_\mu : \Sect(\C TM) \to C^\infty(M)$ is the formal adjoint of the exterior derivative $\dd : C^\infty(M) \to \C \Forms^1(M)$ with respect to $\mu$.

If $(U,x)$ is a coordinate chart and
 $\dd\mu(x) = \rho(x)\,\dd x$ 
on $U$, then
\[
\div_\mu v = \sum_{j=1}^n\left( \de_j v^j + v^j \frac{\de_j\rho}{\rho} \right) .
\]

\begin{definition}
	Let $(M,H,\mu)$ be a measured quadratic Hamiltonian pair.
	The \emph{sub-Laplacian} of a function $f\in C^\infty(M)$ is the smooth function
	\[
	\sLap f = -\div_\mu(\nabla_Hf) = -\div_\mu(B_H(\dd f)) .
	\] 
\end{definition}
If $(U,x)$ is a coordinate chart and $\dd\mu(x) = \rho(x) \,\dd x$ on $U$, then
\begin{equation}\label{eq08292346}
\sLap f = -\sum_{jk} \left( H^{kj}\, \de_k\de_jf + \left(\de_k H^{kj} + H^{kj} \frac{\de_k\rho}{\rho} \right)\,\de_jf \right) .
\end{equation}
This shows that $\sLap:C^\infty(M)\to C^\infty(M)$ is a second-order differential operator with principal symbol $H$.

Notice that, for all $f,g\in C^\infty(M)$, 
\[
\int_M \sLap f\,\, \overline{g} \,\dd\mu 
= \int_M b_H(\dd f,\dd g) \,\dd\mu 
= \int_M \langle \nabla_H f, \nabla_H g \rangle_H \,\dd\mu.
\]
This implies that $\sLap$ is a nonnegative symmetric operator.
Therefore, there exists a nonnegative self-adjoint extension of $\sLap$ on $L^2(\mu)$, such as Friedrichs' extension,
see for instance \cite[Section XI.7]{MR1336382}.

Once such a self-adjoint extension of $\sLap$ is chosen, a Borel functional calculus for $\sLap$ is defined via the spectral theorem and, for all bounded Borel functions $F : \R \to \C$, the operator $F(\sLap)$ is bounded on $L^2(M)$. Since $\sLap$ is self-adjoint,
\[
F(\sLap)^* = \overline{F}(\sLap),
\]
and moroever, since additionally $\sLap$ preserves real-valued functions,
\[
\overline{F(\sLap)f} = \overline{F}(\sLap) \overline{f}
\]
for all $f \in L^2(M)$. In particular, for all $p \in [1,\infty]$,
\begin{equation}\label{eq:adjointcalculus}
\| F(\sLap) \|_{p \to p} = \| \overline{F}(\sLap) \|_{p \to p} = \| F(\sLap) \|_{p' \to p'}.
\end{equation}

Functional calculus allows us to define the wave propagator $t \mapsto \cos(t\sqrt{\sLap})$ associated with $\sLap$.
In the sequel we will need a couple of assumptions on the wave propagator. The first is finite propagation speed:
\begin{equation*}\tag{FPS}\label{FPS} 
\begin{array}{c}
\text{for all $U \subset M$ open and $K\subset U$ compact there is an $\epsilon>0$}\\
\text{ such that $\spt(\cos(t\sqrt{\sLap})u)\subset U$}\\
\text{for all $t\in(-\epsilon,\epsilon)$ and $u\in C^\infty_c(M)$ with $\spt(u)\subset K$.}
\end{array}
\end{equation*}
This assumption is satisfied in fairly general context: see for instance
\cite{MR925249,MR2114433,MR2352563,MR3026352,MR3080173}
and references therein.
The second is smoothness preservation:
\begin{equation*}\tag{SP}\label{SP}
\begin{array}{c}
\text{for all $K \subset M$ compact there exists $\epsilon > 0$ such that,}\\
\text{for all $u \in C^\infty_c(M)$ with $\spt(u) \subset K$,}\\
\text{the function $(t,x) \mapsto \cos(t\sqrt{\sLap}) u(x)$ is smooth on $(-\epsilon,\epsilon) \times M$.}
\end{array}
\end{equation*}
Since $\cos(t\sqrt{\sLap})$ is a contraction on $L^2(M)$, it is easily seen that this assumption is satisfied under sub-ellipticity assumptions on $\sLap$ (e.g., when $(M,H)$ is bracket-generating, by H\"ormander's theorem \cite{MR0222474}), or more generally when $\sLap$ commutes with an operator $D$ such that
\[
C^\infty_c(M) \subset \{ f \in L^2(M) \tc D^k f \in L^2(M) \,\, \forall k \in \N\} \subset C^\infty(M).
\]
We remark that hypoellipticity of $\sLap$ is not a necessary condition for \eqref{FPS} and \eqref{SP} to hold: for instance, if $M$ is a Lie group and $\sLap = -v^2$ for some left-invariant vector field $v$, then \eqref{FPS} and \eqref{SP} are satisfied.

Some results in the sequel will require a further assumption on the functional calculus for $\sLap$:
\begin{equation*}\tag{SFC}\label{SFC}
\begin{array}{c}
\text{for all $F \in \Sch(\R)$, the operator $F(\sLap)$ is bounded on $L^1(M)$.}
\end{array}
\end{equation*}
This assumption is verified, e.g., whenever there is a doubling distance on $(M,\mu)$ such that $\sLap$ satisfies gaussian-type heat kernel bounds, cf.\ \cite{MR782662,MR1172944,hebisch_functional_1995,MR1943098} and \cite[Theorem 6.1(iii)]{MR3671588}. We remark that, under \eqref{SFC}, if $F \in \Sch(\R)$, then $F(\sLap)$ is bounded on $L^p(M)$ for all $p \in [1,\infty]$ (by duality and interpolation) and moreover, by the closed graph theorem, the correspondence $F \mapsto F(\sLap)$ is continuous from $\Sch(\R)$ to the space of $L^p$-bounded operators (with the operator norm topology).

\subsection{Sub-Riemannian structures defined by systems of vector fields}

A common way to define a quadratic Hamiltonian pair or a sub-Riemannian manifold is by choosing a family of vector fields $v_1,\dots,v_r\in \Sect(TM)$ and defining 
\begin{equation}\label{eq10071358}
H = \sum_{j=1}^r v_j \otimes v_j .
\end{equation}
We have the following expressions: if $\alpha,\beta\in T^*_x M$, then
\begin{gather*}
b_H(\alpha,\beta) = \sum_{j=1}^r  \alpha(v_j|_x)\beta(v_j|_x),  \quad
H(\alpha) = \sum_{j=1}^r  \alpha(v_j|_x)^2 ,  \quad
B_H(\alpha) = \sum_{j=1}^r \alpha(v_j|_x) v_j|_x  , \\
\cH_x = \Span \{v_j|_x\}_{j=1,\dots,r} , \qquad
\nabla_H f = \sum_{j=1}^r (v_jf)\, v_j.
\end{gather*}
In particular, $\cH$ is bracket generating if and only if the family of vector fields $v_1,\dots,v_r$ is bracket generating.
Moreover, if $\mu$ is a smooth positive measure on $M$, for all
$a=\sum_{j=1}^r a^j v_j \in \cH$ and $f \in C^\infty(M)$,
\begin{equation}\label{eq:divLexpl}
\div_\mu(a) = - \sum_{j=1}^r v_j^\mu a^j, \qquad
\sLap f = \sum_{j=1}^r v_j^\mu v_j f ,
\end{equation}
where $v^\mu$ is the formal adjoint of $v$, that is, the differential operator $v^\mu : C^\infty(M) \to C^\infty(M)$ such that $\int_M f\,\overline{vg} \,\dd\mu=\int_M v^\mu f\,\bar g \,\dd\mu$, for all $f,g\in C^\infty_c(M)$.

\begin{remark}\label{rm:sumofsquares}
If $(M,H)$ is a quadratic Hamiltonian pair such that $x \mapsto \dim(\cH_x)$ is constant, then $H$ can be written as in \eqref{eq10071358}, at least locally: indeed $\cH_M$ is a smooth subbundle of $T M$ and one can take as $v_j$ a local orthonormal frame of $\cH_M$. However, not all quadratic Hamiltonian pairs $(M,H)$ admit the decomposition \eqref{eq10071358} with smooth vectors fields $v_j$, not even locally (cf.\ \cite[p.\ 8]{MR0457908}).
	Indeed, by \cite{MR1510517}, there is a homogeneous nonnegative real polynomial $p(x,y,z)$ of degree 6 in three variables that is not a finite sum of squares of polynomials (see also \cite{MR900345,MR1882548,MR2254649}).
	One can thus see, arguing with Taylor series, that $p$ is not a finite sum of squares of smooth functions in any neighbourhood of the origin.
	Now, fix a frame $(X,Y,Z)$ of $T\R^3$ and the dual coframe $\alpha_X,\alpha_Y,\alpha_Z$ of $T^*M$.
	Define
	\[
	H = X\otimes X + Y\otimes Y + p\cdot Z\otimes Z  .
	\]
	Note that $H(\alpha_Z) = p$.
	If $H$ were of the form $\sum_j v_j\otimes v_j$, then $p = H(\alpha_Z) = \sum_j v_j(\alpha_Z)^2$ would be a sum of squares of smooth functions.
	Therefore, $H$ cannot be written as in~\eqref{eq10071358}.
	Note that, by choosing $X$, $Y$ and $Z$ so that $[X,Y]=Z$, we also obtain a sub-Riemannian structure that is not written as in~\eqref{eq10071358} with smooth vectors fields $v_j$.
	However, $H$ can always be written as in~\eqref{eq10071358} with Lipschitz vector fields $v_j$, see \cite{MR0229669}.
\end{remark}

A particular class of quadratic Hamiltonian pairs where the above-described pathologies do not occur is defined below.

\begin{definition}
	A quadratic Hamiltonian pair $(M,H)$ is called \emph{equiregular} if $\cH^{(k)}_M$ is a subbundle of $TM$ for all $k$.
\end{definition}

In other words, we are requiring $x \mapsto \dim(\cH^{(k)}_x)$ to be constant, for all $k$.

Not all quadratic Hamiltonian pairs are equiregular, as shown by the example in Remark \ref{rm:sumofsquares}. A simpler, classical example arises when $\sLap = -(X^2+Y^2)$, with $X = \partial_x$ and $Y = x \partial_y$, is the Grushin operator on $\R^2$; in this case, despite the non-equiregularity, the Hamiltonian can be globally written in the form \eqref{eq10071358}.

In any case, for an arbitrary quadratic Hamiltonian pair, from the lower semicontinuity of the functions $x\mapsto \dim(\cH^{(k)}_x)$ for $k \in \N$, we immediately deduce the following result.

\begin{lemma}\label{lem10101454}
	Let $(M,H)$ be a quadratic Hamiltonian pair. 
	\begin{enumerate}[label=(\roman*)]
	\item Every nonempty open set of $M$ contains a nonempty open set $M_1$ such that $(M_1,H)$ is an equiregular quadratic Hamiltonian pair.
	\item If $x \in M$ satisfies $\max_{k\in \N} \dim(\cH^{(k)}_x) = \dim M$, then every neighbourhood of $x$ contains a nonempty open set $M_1$ such that $(M_1,H)$ is an equiregular sub-Riemannian manifold.
	\end{enumerate}
\end{lemma}

\subsection{The sub-Riemannian exponential map}

Let $(M,H)$ be a quadratic Hamiltonian pair.
As in Section~\ref{sec09190901}, we denote by $\Phi_H$ the flow on $T^*M$ of the Hamiltonian vector field $\X_H \in \Sect(T(T^*M))$ defined by means of the standard symplectic form on $T^*M$, and write $\Exp_{H}^{o,t}(\xi) = \pi_M \Phi_H^t(\xi)$ for all $o \in M$, $t \in \R$ and $\xi \in T_o^* M$ for which $(t,\xi)$ is in the domain of $\Phi_H$.

Since $H$ is $2$-homogeneous, i.e., $H(\lambda\xi)=\lambda^2 H(\xi)$ for all $\lambda\in\R$ and $\xi\in T^*M$, we deduce the following properties of the flow:
\begin{equation}\label{eq09271803}
\Phi_H^t(\lambda\xi) = \lambda\Phi_H^{\lambda t}(\xi)
\quad\text{and}
\quad
\Exp_H^{o,t}(\lambda\xi) = \Exp_H^{o,\lambda t}(\xi) .
\end{equation}
Because of this scaling property, the exponential map $\Exp_H^o \defeq \Exp_H^{o,1}$ at time $t=1$ already contains all the relevant information.

From the fact that $H$ is a quadratic form, we deduce the following information on the curves defined via the exponential map.

\begin{lemma}\label{lem10011320}
Let $\xi\in T_o^*M$.
Let $\xi(t)=\Phi_H^t(\xi)$ and $x(t)=\Exp_H^{o}(t\xi)$ (these are both defined for $t$ in an open interval containing $0$).
Then 
\[
\dot x(t) = 2B_H(\xi(t)).
\]
In particular, $x(t)$ is a horizontal curve, i.e., $\dot x(t)\in\cH_M$ for all $t$, and $\langle \dot x(t),\dot x(t) \rangle_H = 4 H(\xi)$ is constant. Moreover, if $H(\xi)=0$, then $\Exp_H^{o}(t\xi)=o$ for all $t \in \R$.
\end{lemma}
\begin{proof}
	In canonical coordinates, by \eqref{eq09190853}, 
	\[
	\dot x_j 
	= \frac{\de H}{\de\xi_j}(x,\xi) 
	= \frac{\de }{\de\xi_j} \left(\sum_{a,b} H^{ab}(x)\xi_a\xi_b \right)
	= 2\sum_a H^{aj}(x)\xi_a
	= 2 B_H(\xi),
	\]
and the conclusion follows by \eqref{eq:innerproduct}.
\end{proof}

We will need a regularity property of the exponential map:
\begin{equation*}\tag{RE}\label{RE}
\begin{array}{c}
\text{there are $o\in M$ and $\xi\in T^*_oM$ such that}\\
\text{$s\xi$ is a regular point of $\Exp_H^o$ for all $s\in [-1,1]\setminus\{0\}$.} 
\end{array}
\end{equation*}

We say that a quadratic Hamiltonian pair $(M,H)$ is \emph{analytic}	 if $M$ is an analytic manifold and $H$ is an analytic function.

\begin{lemma}\label{lem10101432}
	Analytic sub-Riemannian manifolds $(M,H)$ satisfy \eqref{RE}. 
\end{lemma}
\begin{proof}
	Let $(M,H)$ be an analytic sub-Riemannian manifold. By Lemma \ref{lem10101454}, up to restricting to an open subset, we may assume that $(M,H)$ is equiregular.
	Let $o\in M$ and $U\subset T^*_o M$ be a neighbourhood of $0$ on which $\Exp_H^o$ is defined.
	From \cite[Theorem 1]{MR2513150}, we deduce that, after choosing analytic coordinates, $U \ni \xi \mapsto \det(\DD\Exp_H^o|_{\xi})\in \R$ is a nonzero analytic map.
	Therefore, there is $\xi \in U$ such that $t\mapsto \det(\DD\Exp_H^o|_{t\xi})$ is a nonzero analytic map on an open interval $I$ containing $0$.
	Since zeros of this map do not have cluster points in $I$, there is $\epsilon>0$ such that $\det(\DD\Exp_H^o|_{t\xi}) \neq 0$ for all $t\in(-\epsilon,\epsilon) \setminus \{0\}$.
\end{proof}

\begin{remark}
	In fact, property~\eqref{RE} holds for all sub-Riemannian manifolds of constant rank, at all points.
	We sketch how one deduces~\eqref{RE} from partial statements that appear in the literature.
	
	\newcommand{\End}{\mathtt{End}}
	
	In \cite[Definition 3.2]{MR3852258} the notion of \emph{ample geodesic} is introduced.
	We need two facts about ample geodesics. 
	First, one obtains from 
	\cite[Proposition 5.23]{MR3852258} that, for every $o\in M$, there is $\xi\in T^*_o M$ such that $\gamma : [-1,1] \to M$, $\gamma(t) = \Exp_H^o(t\xi)$ is an ample geodesic at $o$.
	Second, if $\gamma : [-1,1] \to M$ is an ample geodesic at $\gamma(0)$, then it is \emph{strongly normal}, that is, it is not abnormal on each subinterval of the form $[0,t]$ or $[t,0]$, see \cite[Definition 2.14 and Proposition 3.6(iii)]{MR3852258}.
		
	If $\gamma : s \mapsto \Exp_H^o(s\xi)$, then a point $\gamma(t)$ is said to be \emph{conjugate} to $\gamma(0)$ along $\gamma$ if 
	$\Im(\DD\Exp^o_H|_{t\xi}) \neq T_{\gamma(t)} M$
	\cite[Definition A.1]{MR3852258}.
	By \cite[Proposition A.2]{MR3852258}, if $\gamma:[-1,1]\to M$ is strongly normal, then there is $\epsilon>0$ such that $\gamma(t)$ is not conjugate to $\gamma(0)$ along $\gamma$ for all $t\in[-\epsilon,\epsilon]\setminus\{0\}$.
	See also \cite[\S3.(iii)]{MR2513150} and \cite[Corollary 8.50]{abb_introduction_2018}.
	
	Finally, we conclude that for every $o\in M$ there are $\xi\in T^*_o M$ and $\epsilon>0$ such that $\Im(\DD\Exp^o_H|_{t\xi})=T_{\Exp^o_H(t\xi)}M$ for all $t \in [-\epsilon,\epsilon]\setminus\{0\}$, i.e., \eqref{RE} holds.
\end{remark}

\subsection{Carnot groups}\label{ss:carnot}
A \emph{Carnot group} is a connected simply connected Lie group $G$ whose Lie algebra $\Lie{g}$ is stratified, i.e., $\Lie{g}=\bigoplus_{j=1}^s V_j$ for some linear subspaces $V_1,\dots,V_s$ with $[V_1,V_j]=V_{j+1}$ for all $j=1,\dots,s$ (here $V_{s+1}=0$),
and whose first layer $V_1$ is endowed with a fixed scalar product.
We shall always assume that Carnot groups are endowed with a (bi-invariant) Haar measure.

We can describe Carnot groups as quadratic Hamiltonian pairs as follows.
Let $(v_1,\dots,v_r)$ be an orthonormal basis of $V_1$ (in particular the $v_k \in \Sect(TG)$ are left-invariant vector fields) and set $H=\sum_{j=1}^r v_j\otimes v_j$ (note that $H$ is independent on the choice of the orthonormal basis).

Correspondingly, the sub-Laplacian on a Carnot group is $\sLap=-\sum_{j=1}^r v_j^2$. As a left-invariant sub-Laplacian on a Lie group, $\sLap$ is essentially self-adjoint (cf.\ \cite{MR0110024}), hence it admits a unique self-adjoint extension.

Carnot groups are a special case of equiregular sub-Riemannian manifolds and they appear as infinitesimal models of all sub-Riemannian manifolds (possibly after applying some ``lifting'' procedure), see \cite{MR0436223,MR806700,MR1421822} and references therein.
We will use this fact to extend our main result to all quadratic Hamiltonian pairs.

Carnot groups satisfy all our key assumptions.
\begin{lemma}\label{lem10111454}
	Carnot groups satisfy \eqref{RE}, \eqref{FPS}, \eqref{SP} and \eqref{SFC}.
\end{lemma}
\begin{proof}
	Since Carnot groups are analytic, \eqref{RE} follows from Lemma \ref{lem10101432}.
	Since the corresponding sub-Laplacians are essentially self-adjoint, \eqref{FPS} is well known, see for instance \cite{MR925249,MR3026352} and sub-ellipticity of $\sLap$ also gives \eqref{SP}. Finally, \eqref{SFC} is proved in \cite{MR782662}.
\end{proof}

\subsection{Eikonal equation on sub-Riemannian manifolds}
Let $(M,H)$ be a quadratic Hamiltonian pair. Define $\cD_A = \{ \xi \in T^* M \tc H(\xi) \neq 0\}$ and let $A : \cD_A \to \R$ be defined by $A(\xi) = \sqrt{H(\xi)}$. Note that, since $H$ is $2$-homogeneous, $\cD_A$ is a conic open subset of $T^* M$ and $A$ is a $1$-homogeneous smooth function on $\cD_A$.

\begin{lemma}\label{lem09201541}
	For all $\eta \in \cD_A$,
	\[
	\X_A|_\eta =\frac{1}{2A(\eta)} \X_H|_\eta.
	\]
	In particular, for all $x \in M$, $\eta \in \cD_A \cap T_x^* M$ and $t \in \R$,
	\begin{equation}\label{eq09201534}
	\Phi^t_A(\eta) = \Phi_H^{\frac{t}{2A(\eta)}}(\eta) , 	\qquad \Exp_{A}^{x,t}(\eta) = \Exp_H^{x}\left(\frac{t\eta}{2A(\eta)}\right),
	\end{equation}
	whenever one of the two sides is well defined.
\end{lemma}
\begin{proof}
	The relation between $\X_A$ and $\X_H$ is immediately given by \eqref{eq09190852} and the fact that $H=A^2$.
	Since $H$ and $A$ are constant along the integral curves of $\X_H$ and $\X_A$ (see \eqref{eq09181520}), it is immediately seen that $\cD_A$ is invariant under the flow of $H$, and moreover an integral curve of $\X_H$ in $\cD_A$ is obtained by time-rescaling of an integral curve of $\X_A$, and conversely, which leads to \eqref{eq09201534}.
\end{proof}

Recall the definition of the vertical differential $\DD_2 H$ from \eqref{eq:vert_diff}. Since $H$ is nonnegative and $2$-homogeneous, $\DD_2 H|_\eta = 0$ if and only if $H(\eta) = 0$. So, for all $x \in M$ and $\eta \in \cD_A \cap T_x^* M$, $\ker \DD_2 H|_\eta$ is a $1$-codimensional subspace of $T_x^* M$.

\begin{corollary}\label{cor:rankAH}
For all $x \in M$, $t \in \R \setminus \{0\}$ and $\eta \in T^*_x M$ in the domain of $\Exp_{A}^{x,t}$,
\[
\rank(\DD \Exp_{A}^{x,t}|_\eta) = \rank \left(\left.\DD\Exp_H^{x}|_{\lambda\eta}\right|_{\ker \DD_2 H|_{\lambda\eta}}\right),
\]
where $\lambda = t/(2\sqrt{H(\eta)})$.
In particular, if $\lambda \eta$ is a regular point of $\Exp_{H}^{x}$,  then $\DD \Exp_{A}^{x,t}|_\eta$ has maximal possible rank $n-1$.
\end{corollary}
\begin{proof}
Since $\eta \in \cD_A$, the vertical differential $\DD_2 H|_{\eta}$ does not vanish, so the level set of $H$ in $T_x^* M$ through $\eta$ is locally a $1$-codimensional submanifold $S$ of $T_x^* M$ whose tangent space at $\eta$ is $\ker \DD_2 H|_{\eta}$. Note also that, since $H$ is $2$-homogeneous, $\eta \notin \ker \DD_2 H|_{\eta}$, whence $T_x^* M = \ker \DD_2 H|_{\eta} \oplus \R \eta$. By Proposition \ref{prp:ced}\ref{en:ced1}, $\Exp_{A}^{x,t}$ is $0$-homogeneous, so $\DD \Exp_{A}^{x,t}|_\eta[\eta] = 0$ and therefore
\[
\rank(\DD \Exp_{A}^{x,t}|_\eta) = \rank\left(\left.\DD \Exp_{A}^{x,t}|_\eta\right|_{\ker \DD_2 H|_{\eta}}\right).
\]
On the other hand, $\left.\DD \Exp_{A}^{x,t}|_\eta\right|_{\ker \DD_2 H|_{\eta}}$ is the differential at $\eta$ of the restriction of $\Exp_A^{x,t}$ to $S$. Moreover,  for all $\xi \in S$, since $H(\xi) = H(\eta)$, by Lemma \ref{lem09201541} we deduce
\[
\Exp_{A}^{x,t}(\xi) = \Exp_H^{x}\left(\lambda \xi\right)
\]
and, by homogeneity, $\ker \DD_2 H|_{\lambda \eta}$ is the tangent space at $\lambda \eta$ of $\lambda S$. Hence
\[
\left.\DD \Exp_{A}^{x,t}|_\eta\right|_{\ker \DD_2 H|_{\eta}} = \lambda \left.\DD\Exp_H^{x}|_{\lambda \eta}\right|_{\ker \DD_2 H|_{\lambda \eta}}
\]
and we are done.
\end{proof}

\subsection{Fourier integral representation of the wave propagator}\label{sec10031549}
By combining the results obtained so far, in this section we obtain a Fourier integral representation of a frequency localised portion of the wave propagator $\cos(t\sqrt{\sLap})$ associated to a sub-Laplacian $\sLap$.

\begin{theorem}\label{thm10030928}
Let $(M,H,\mu)$ be a measured quadratic Hamiltonian pair with sub-Laplacian $\sLap$.
Let a self-adjoint extension of $\sLap$ be chosen and assume that \eqref{FPS} and \eqref{SP} are satisfied.
Let $o\in M$ and $(M_o,x)$ be a coordinate chart at $o$.
We identify $M_o$ with an open neighbourhood of $0$ in $\R^n$.
Let $\Gamma \subset \nR^n$ be a closed cone such that
\begin{equation}\label{eq10030000}
\Gamma \subset \{\xi \in \nR^n \tc H(o,\xi) \neq 0\} .
\end{equation}
Then there are an open neighbourhood $X \subset M_o$ of $o$, a $T>0$ and a smooth function $w:(-T,T)\times X\times \nR^n \to \R$ with the following properties.
\begin{enumerate}[label=(\roman*)]
\item\label{en:eitL_phhom} 
$w$ is $1$-homogeneous in $\xi$ and, for all $(t,x,\xi) \in (-T,T) \times X \times \nR^n$,
\[
w(0,x,\xi) = x \cdot \xi, \qquad \partial_x w(t,x,\xi) \neq 0 .
\]
\item\label{en:eitL_phder}
The function $\phi : (-T,T) \times X \times X \times \nR^n \to \R$, 
\begin{equation}\label{eq:phase_decomp}
\phi(t,x,y,\xi) = w(t,x,\xi) - w(0,y,\xi),
\end{equation}
is a phase function.
Moreover, $(x,y,\xi) \mapsto \phi(t,x,y,\xi)$ is an operator phase function for all $t \in (-T,T)$.
\item\label{en:eitL_phcrit} 
For all $t \in (-T,T)$, $x,y \in X$ and $\xi \in \Gamma$,
\[
\partial_\xi w(t,x,\xi) = y \quad\iff\quad x = \Exp_H^y(-t\xi / (2\sqrt{H(y,\xi)})),
\]
and in that case
\[
\partial_t w(t,x,\xi) = \sqrt{H(y,\xi)}.
\]
\item\label{en:eitL_phhess} 
For all $t \in (-T,T) \setminus \{0\}$, $x,y \in X$ and $\xi \in \Gamma$
such that $\partial_\xi w(t,x,\xi) = y$,
\[
\rank (\partial_\xi^2 w(t,x,\xi)) = \rank (\DD \Exp_H^y|_{\lambda \xi}|_{\ker \DD_2 H|_{(y,\lambda\xi)}}),
\]
where $\lambda = -t/(2\sqrt{H(y,\xi)})$.
\end{enumerate}	
	Moreover, for all open subsets $X',X''$ of $X$ with $X'' \Subset X' \Subset X$, there is a $T' \in (0,T]$ such that the following hold true:
	if $P\in\Psi^0_\cl(M)$ is a compactly supported operator with $\spt(P) \subset X'' \times X''$ (cf.\ Section \ref{sec:pre_preliminaries}), whose restriction to $M_o \times M_o$ has a distributional integral kernel given by the oscillatory integral
	\begin{equation}\label{eq:PFIO}
	P(x,y) = \int_{\R^n} e^{i (x-y) \cdot \xi} p(x,y,\xi) \,\dd\xi 
	\end{equation}
	 for some amplitude $p\in S^0_{\cl}(M_o\times M_o; \R^n)$ with $\essspt(p)\subset X''\times X''\times \Gamma$,
	then there is  a $Q \in \Rop(M;(-T',T') \times M)$ with support $\spt(Q) \subset (-T',T') \times M_o \times M_o$, whose restriction to $(-T',T') \times M_o \times M_o$ has  the distributional integral kernel
	\begin{equation}\label{eq:QFIO}
	Q_t(x,y) = Q(t,x,y) = \int_{\R^n} e^{i\phi(t,x,y,\xi)} q(t,x,y,\xi) \, \dd\xi
	\end{equation}
	for some amplitude $q\in S_{\cl}^0((-T',T')\times M\times M; \R^n)$,
	such that:
	\begin{enumerate}[resume,label=(\roman*)]
	\item\label{en:eitL_fiospt} 
	$\spt(q)\subset (-T',T')\times X'\times X'\times\Gamma$;
	\item\label{en:eitL_fiocos}
	there exists $R  \in \Rop^{-\infty}(M; (-T',T') \times M)$ with $\spt(R) \subset (-T',T') \times X' \times X'$ such that, for all $t \in (-T',T')$,
\[
	\cos(t\sqrt{\sLap})P = \frac{1}{2}(Q_t+Q_{-t}) + R_t.
\]
	\end{enumerate}
\end{theorem}

In the above statement the expressions in \eqref{eq:PFIO} and \eqref{eq:QFIO} are intended as the integral kernels of the operators $P$ and $Q$
with respect to the Lebesgue measure on the coordinate chart $(M_o,x)$. However an analogous statement holds for the integral kernels with respect to the measure $\mu$ on the manifold $M$: indeed, changing the reference measure corresponds to multiplying the amplitudes $p$ and $q$ in \eqref{eq:PFIO} and \eqref{eq:QFIO} by a smooth function in the variable $y$ (the density of one measure with respect to the other), which does not change the symbol class or the support.

\begin{proof}
	By \eqref{eq10030000}, there are open subsets $W,W',W'' \subset \Sphere^{n-1}$ such that
	\[
	\Gamma\cap \Sphere^{n-1} \subset W \Subset W' \Subset W'' \Subset \{\xi \tc H(o,\xi) \neq 0\} .
	\]
	Up to shrinking $M_o$, we can assume that, for all $x\in M_o$,
	\begin{equation}
	W'' \Subset\{\xi \tc H(x,\xi) \neq 0\} .
	\end{equation}
	Therefore, if $\psi \in C^\infty(\Sphere^{n-1})$ is such that $\psi|_{W'} = 1$ and $\spt(\psi) \subset W''$, then
	\[
	\tilde H(x,\xi) = \psi(\xi/|\xi|) \, H(x,\xi) + (1-\psi(\xi/|\xi|)) \, |\xi|^2
	\]
	defines a smooth $2$-homogeneous function $\tilde H : \nT^* M_o  \to (0,\infty)$ such that $H = \tilde H$ on $M_o \times \R^+ W'$.

	Let $\tilde A = \sqrt{\tilde H}$. Note that both $A$ and $\tilde A$ are $1$-homogeneous and $\tilde A = A$ on $M_o \times \R^+ W' \subset \nT^* M_o$. By Proposition \ref{prp:ced}\ref{en:ced1}, since $\Gamma \cap \Sphere^{n-1}$ is a compact subset of $W'$, there exist $\epsilon>0$ and an open neighbourhood $U \subset M_o$ of $o$ such that, for all $t \in (-\epsilon,\epsilon)$, $x \in U$ and $\xi \in \Gamma$, the point $(t,(x,\xi))$ is in the domain of both $\Phi_A$ and $\Phi_{\tilde A}$ and 
	\begin{equation}\label{eq:sameflow}
	\Phi_A^t(x,\xi) = \Phi_{\tilde A}^t(x,\xi) \in M_o \times \R^+ W'.
	\end{equation}
	Indeed, as long as the flow associated with $A$ stays in $M_0 \times \R^+ W'$, it must coincide with the flow of $\tilde A$, and conversely.

Let now $w : (-T,T) \times X \times \nR^n \to \R$ and $\phi : (-T,T) \times X\times X \times \nR^n \to \R$ (where $X \subset U$ is an open neighbourhood of $o$ and $T \in (0,\epsilon]$) be the smooth functions given by Proposition \ref{prp:operatorphaseeikonal} applied to $\tilde A$ in place of $A$. In particular, $\phi$ satisfies the eikonal equation
\[
\partial_t \phi(t,x,y,\xi) = \tilde A(\partial_x \phi(t,x,y,\xi))
\]
for all $(t,x,y,\xi) \in  (-T,T) \times X \times X \times \nR^n$, and moreover parts \ref{en:eitL_phhom} and \ref{en:eitL_phder} are satisfied. In addition, from parts \ref{en:operatorphaseeikonal_critical} and \ref{en:operatorphaseeikonal_hessian} of Proposition \ref{prp:operatorphaseeikonal}, combined with \eqref{eq:sameflow}, Lemma \ref{lem09201541} and Corollary \ref{cor:rankAH}, we deduce parts \ref{en:eitL_phcrit} and \ref{en:eitL_phhess}.

Note that $\partial_x w(0,x,\xi) = \xi$. Hence we can find $T_0 \in (0,T]$ such that
\begin{equation}\label{eq:w_cone_propagation}
\partial_x w(t,x,\xi) \in \R^+W
\end{equation}
for all $t \in (-T_0,T_0)$, $x \in X'$ and $\xi \in \Gamma$.

	By \eqref{FPS}, up to choosing a smaller $T_0$, we may assume that $\spt(\cos(t\sqrt{\sLap}) f) \subset X'$ for all $t \in (-T_0,T_0)$ and $f \in C^\infty_c(M)$ with $\spt(f) \subset X''$. Similarly, by \eqref{SP}, up to choosing a smaller $T_0$, we may assume that $(t,x) \mapsto \cos(t\sqrt{\sLap})u(x)$ is smooth on $(-T_0,T_0) \times M$ for all $u \in C^\infty_c(M)$ with $\spt(u) \subset X'$.

	Recall that, by \eqref{eq08292346}, the principal symbol of the sub-Laplacian $\sLap$ is $H$.
	Let $\tilde{\sLap}\in\Psi_{\cl}^2(M_o)$ be a properly supported operator such that the asymptotic expansion 
	of its symbol is the same as that of $\sLap$ on $M_o$, except that the principal symbol $H$ is replaced by $\tilde H$.

Let $b \in C^\infty(\Sphere^{n-1})$ be such that $b|_{\Sphere^{n-1} \setminus W'} = 1$ and $\spt(b) \subset \Sphere^{n-1} \setminus W$. Let $B \in \Psi^0_\cl(M_o)$ be a properly supported operator such that all the terms in the asymptotic expansion of its symbol vanish, except for the principal symbol $(x,\xi) \mapsto b(\xi/|\xi|)$. Then, by \eqref{eq11231818}, 
\begin{equation}\label{eq:wf_B}
\WF(B) \subset \{ (x,x;\xi,-\xi) \tc x \in M_o, \, \xi \in \nR^n \setminus \R^+W \}.
\end{equation}
Moreover, since all the terms in the asymptotic expansion of the symbols of $\sLap$ and $\tilde\sLap$ coincide on $M_o \times \R^+ W'$, from the composition formula for pseudodifferential operators (see, e.g., \cite[Theorem 18.1.8]{MR781536}) we immediately deduce that
\begin{equation}\label{eq:slap_localisation}
(\sLap - \tilde\sLap) (\Id - B) \in \Psi^{-\infty}(M_o).
\end{equation}
	
Since $\tilde H$ is everywhere positive, by Lemma~\ref{lem09030037} there is a properly supported $\opA \in \Psi^1_{\cl}(M_o)$ with principal symbol $\tilde A$ such that
\begin{equation}\label{eq:slap_squareroot}
\opA^2 - \tilde{\sLap} \in\Psi^{-\infty}(M_o).
\end{equation}
We now apply Theorem \ref{thm08141709} to the phase function $\phi$ and the pseudodifferential operators $\opA$ and $P$ on $M_o$, thus obtaining a $T' \in (0,T_0]$ and a Fourier integral operator $Q$ of the form \eqref{eq:QFIO} with
	\[
	\spt(q)\subset (-T',T')\times X'\times X'\times\Gamma
	\]
	(this proves part \ref{en:eitL_fiospt}) and such that $Q_0 - P$ and $(i\partial_t + \opA) Q$ are smoothing. 

We now prove that $(\de_t^2 +\sLap) Q$ is smoothing. Indeed, let us write	
\[\begin{split}
(\de_t^2 +\sLap) Q 
&= (\de_t^2 +\tilde\sLap) Q + (\sLap-\tilde\sLap) B Q + (\sLap-\tilde\sLap) (\Id-B) Q \\
&= (-i\partial_t +\opA)(i\partial_t + \opA) Q + (\sLap-\tilde\sLap) B Q + C Q,
\end{split}\]
where $C =(\tilde\sLap - \opA^2) + (\sLap-\tilde\sLap) (\Id-B) \in \Psi^{-\infty}(M_o)$ by \eqref{eq:slap_localisation} and \eqref{eq:slap_squareroot}. Since $\opA$ and $\sLap - \tilde\sLap$ are pseudodifferential operators and preserve smooth functions, it is enough to show that $(i\partial_t + \opA) Q$, $BQ$ and $CQ$ are smoothing. On the other hand, $(i\partial_t + \opA) Q$ is smoothing by construction. As for the other operators, let us write $BQ = (\Id \otimes B) Q$ and $CQ = (\Id \otimes C) Q$, where $\Id \otimes B, \Id\otimes C : \Rop((-T',T') \times M_o ; (-T',T') \times M_o)$, and where $\Id$ denotes the identity operator with respect to the variable $t$. 
Then, by \eqref{eq:wf_B} and \cite[Theorem 8.2.9]{MR717035}, we deduce that
\begin{multline*}
\WF(\Id \otimes B) \\\subset \{ (t,x,t,x;\tau,\xi,-\tau,-\xi) \tc t \in (-T',T'), \,x \in M_o, \, (\tau,\xi) \in \nR^{1+n}, \,  \xi \notin \R^+W \}.
\end{multline*}
Moreover, $\WF(C)=\emptyset$, so, again by \cite[Theorem 8.2.9]{MR717035},
\[
\WF(\Id \otimes C) \subset \{(t,x,t,y;\tau,0,-\tau,0) \tc t \in (-T',T'), \, x,y \in M_o, \, \tau \in \nR\}.
\]
Finally, by \eqref{eq11231818} and \eqref{eq:w_cone_propagation},
\[\begin{split}
\WF(Q) 
&\subset \{ (t,x,y;\partial_t w(t,x,\xi),\partial_x w(t,x,\xi), -\xi) \tc t \in (-T',T') , x,y \in X', \xi \in \Gamma \} \\
&\subset (-T',T') \times X' \times X' \times \nR \times \R^+ W \times (-\Gamma).
\end{split}\]
By \cite[Theorem~8.2.14, p.~270]{MR717035}, we can combine the above information to conclude that
\[
\WF((\Id \otimes B)Q) = \emptyset = \WF((\Id \otimes C) Q),
\]
i.e., $BQ$ and $C Q$ are smoothing.
So $(\partial_t^2 +\sLap) Q$ is smoothing as well. 

Note now that $\partial_t^2 +\sLap$ is a differential operator on $(-T',T') \times M$ and
\[
\spt((\partial_t^2 +\sLap) Q) \subset \spt(Q) \subset (-T',T') \times X' \times X';
\]
hence $Q$ naturally extends by zeros to an operator in $\Rop(M;(-T',T')\times M)$ and $(\partial_t^2 +\sLap) Q$ remains smoothing after the extension.

Define now $\tilde Q$ by $\tilde Q_t = (Q_t+Q_{-t})/2$ for all $t \in (-T',T')$. Then $S \defeq (\partial_t^2 +\sLap) \tilde Q$ is also smoothing and $\spt(S) \subset \spt(\tilde Q) \subset (-T',T') \times X' \times X'$. In addition $\tilde Q_0 = Q_0$ and $\de_t \tilde Q|_{t=0} = 0$.

	We claim that, for every $u \in C_c^\infty(M)$ and $t \in (-T',T')$, the following Duhamel-type formula holds:
	\begin{equation}\label{eq08171102}
	\cos(t\sqrt{\sLap}) Q_0 u = \tilde Q_t u - \int_0^t \frac{\sin((t-\tau)\sqrt{\sLap})}{\sqrt{\sLap}} S_\tau u\, \dd \tau .
	\end{equation}
	To prove the claim, let $B(t)$ denote the right-hand side of \eqref{eq08171102}.
	By direct computation, one shows that
\[
B(0)=Q_0 u, \quad \de_t B(0)= (\de_t\tilde Q|_{t=0}) u=0, \quad \text{and} \quad \de_t^2B(t) = -\sLap B(t).
\]
	Since there is only one solution $U \in C^2((-T',T'); L^2(M))$ to
	\[
	\begin{cases}
	(\de_t^2  + \sLap)U(t)=0 \\
	\de_tU(0)=0 \\
	U(0) = Q_0 u,
	\end{cases}
	\]
	we conclude that $B(t)=\cos(t\sqrt{\sLap}) Q_0 u$, i.e., \eqref{eq08171102} holds.
	
Define $R : C^\infty_c(M) \to C((-T',T');L^2(M))$ by
\[\begin{split}
R_t u &= \cos(t\sqrt{\sLap}) (P-Q_0) u - \int_0^t \frac{\sin((t-\tau)\sqrt{\sLap})}{\sqrt{\sLap}} S_\tau u\, \dd \tau \\
&= \cos(t\sqrt{\sLap}) (P-Q_0) u - \int_0^t \int_0^{t-\tau} \cos(s\sqrt{\sLap}) S_\tau u\, \dd s \, \dd \tau
\end{split}\]
for all $t \in (-T',T')$. Since $S$ and $P-Q_0$ are smoothing and $\spt(P-Q_0) \cup \spt(S_\tau) \subset X' \times X'$, by the smoothness preservation property of the wave propagator we conclude that $R$ is smoothing. In addition, by \eqref{eq08171102}, for all $t\in(-T',T')$,
\[
\cos(t\sqrt{\sLap}) P = \tilde Q_t + R_t,
\]
and $\spt(\cos(t\sqrt{\sLap})P) \subset X' \times X''$ (here we use finite propagation speed and the fact that $\spt(P) \subset X'' \times X''$), while $\spt(\tilde Q_t) \subset X' \times X'$, so  $\spt(R_t) \subset X' \times X'$. This completes the proof of part \ref{en:eitL_fiocos}.
\end{proof}

\section{Proof of the main result}\label{sec:proof_main}

In this section we combine the results of the previous sections and prove Theorem \ref{thm:main}. As mentioned in the introduction, in order to apply the Fourier integral operator representation for the wave propagator, the additional assumptions introduced in Section \ref{sec:subriemannian} are needed. Therefore we will first present the proof under these additional assumptions, and at the end we will show how to remove them by transplantation.

\subsection{Preliminaries}

The following result, similar to \cite[Theorem 7.7.7]{MR717035}, will be useful to compute the action of Fourier integral operators with phase function of the form \eqref{eq:phase_decomp}.

\begin{lemma}\label{lem06142004}
	Let $\Omega \subset \R^m$ be open, $w:\Omega \times\nR^n\to\R$ and $q:\Omega\times\R^n\times\R^n\to\C$ be smooth functions such that
	\[
	\spt(q) \subset C \times \R^n \times \R^n
	\]
	for some closed subset $C$ of $\Omega$, and moreover,
	for all $\beta\in\N^n$ and all $N\in\N$,
	\begin{equation}\label{eq:q_sch}
	|\de_y^\beta q(x,y,\xi)| \lesssim_{\beta,N} (1+|y|)^{-N} 
	\end{equation}
	for all $x\in \Omega$ and $y,\xi \in \R^n$.
	For all $u\in\Sch(\R^n)$ and $\lambda \geq 1$, if $u_\lambda$ is defined by $u_\lambda(\eta) \defeq \lambda^n u(\lambda \eta)$, then, for all $k\in\N$,
	\begin{multline}\label{eq:integral_devel}
	\int_{\R^n}\int_{\R^n} e^{i(w(x,\xi)-y\cdot\xi)} q(x,y,\xi) u_\lambda(y) \, \dd y \, \dd\xi \\
	= \sum_{|\alpha|\le k} \frac{\lambda^{n-|\alpha|}}{i^{-|\alpha|} \alpha!} \int_{\R^n} e^{iw(x,\lambda\xi)} \de_y^\alpha q(x,0,\lambda\xi) \de_\xi^\alpha \hat u(\xi) \,\dd\xi
		+ \lambda^{n-(k+1)} R_{k,\lambda}^{u,q}(x)
	\end{multline}
	where, for all $\lambda \geq 1$ and $k \in \N$,
	\begin{equation}\label{eq09302049}
	\spt(R_{k,\lambda}^{u,q}) \subset C \qquad\text{and}\qquad \sup_{x \in \Omega} |R^{u,q}_{k,\lambda}(x)| \lesssim_{q,u,k} 1.
	\end{equation}
\end{lemma}
\begin{proof}
	Let $\hat q$ denote the partial Fourier transform of $q$ in the variable $y$, i.e.,
	\[
	\hat q(x,\eta,\xi) = \int_{\R^n} e^{-i\eta\cdot y}q(x,y,\xi) \,\dd y .
	\]
	Since the Fourier transform preserves the Schwartz class, from \eqref{eq:q_sch} we deduce that, for all $\beta \in \N^n$ and $N \in \N$,
	\begin{equation}\label{eq:tr_q_sch}
	|\de_\eta^\beta \hat q(x,\eta,\xi)| \lesssim_{\beta,N} (1+|\eta|)^{-N}
	\end{equation}
	for all $x \in \Omega$ and $\eta,\xi \in \R^n$. 
		Moreover, Since the Fourier transform maps pointwise products into convolutions,
	\[\begin{split}
	&\int_{\R^n} e^{iw(x,\xi)} \int_{\R^n} e^{-i\xi\cdot y} q(x,y,\xi) u_\lambda(y) \,\dd y \,\dd\xi \\
	&= (2\pi)^{-n} \int_{\R^n} e^{iw(x,\xi)}  \int_{\R^n} \hat q(x,\xi-\eta ,\xi) \hat u(\eta/\lambda) \,\dd\eta \,\dd\xi \\
	&= (2\pi)^{-n} \lambda^{2n} \int_{\R^n} e^{iw(x,\lambda\xi)}  \int_{\R^n} \hat q(x,\lambda(\xi-\eta),\lambda\xi) \hat u(\eta) \,\dd\eta \,\dd\xi \\
	&= (2\pi)^{-n} \lambda^{2n} \int_{\R^n} e^{iw(x,\lambda\xi)} \left( \sum_{|\alpha|\le k}\frac{\de^\alpha \hat u(\xi)}{\alpha!} \int_{\R^n} \hat q(x,\lambda(\xi-\eta),\lambda\xi) (\eta-\xi)^\alpha \,\dd\eta \right)  \,\dd\xi \\
	&\quad + (2\pi)^{-n} \lambda^{2n} \int_{\R^n} e^{iw(x,\lambda\xi)}  \int_{\R^n} \hat q(x,\lambda(\xi-\eta),\lambda\xi) R(\eta,\xi) \,\dd\eta \,\dd\xi,
	\end{split}\]
	where
	\begin{equation}\label{eq:taylor_remainder}
	R(\eta,\xi) \defeq \hat u(\eta)-\sum_{|\alpha|\le k}\frac{\de^\alpha \hat u(\xi)(\eta-\xi)^\alpha}{\alpha!}.
	\end{equation}
	
	Since
	\[\begin{split}
	&(2\pi)^{-n} \int_{\R^n} \hat q(x,\lambda(\xi-\eta),\lambda\xi) (\eta-\xi)^\alpha \dd\eta \\
	&=  (2\pi)^{-n} i^{|\alpha|} \lambda^{-n-|\alpha|} \int_{\R^n} \hat q(x,\eta,\lambda\xi) (i\eta)^\alpha \dd\eta \\
	&= i^{	|\alpha|} \lambda^{-n-|\alpha|} \de_y^\alpha q(x,0,\lambda\xi),
	\end{split}\]
	the above computations yield \eqref{eq:integral_devel}, if we define
	\[
	R^{u,q}_{k,\lambda}(x) \defeq (2\pi)^{-n}\lambda^{n+k+1} \int_{\R^n} \int_{\R^n} e^{iw(x,\lambda\xi)}\, \hat q(x,\lambda(\xi-\eta),\lambda\xi)\, R(\eta,\xi) \,\dd\eta \,\dd\xi.
	\]

Next, we need to show \eqref{eq09302049}.
Clearly $\spt(R_{k,\lambda}^{u,q}) \subset C$.
To estimate $R^{u,q}_{k,\lambda}$, notice that, since $u \in \Sch(\R^n)$, from \eqref{eq:taylor_remainder} it follows immediately that, for all $N \in \N$ and $\eta,\xi \in \R^n$,
	\begin{equation}\label{eq:sch_est_1}
	|R(\eta,\xi)| \lesssim_{u,N}
	 \begin{cases}  (1+|\eta|)^{k}, & \text{if $|\xi|\leq |\eta|/2$,} \\
        (1+|\eta|)^{-N} , &  \text{if $|\xi|> |\eta|/2$.}
\end{cases}
	\end{equation}
	Moreover, by Taylor's theorem,
	\[
	R(\eta,\xi) = \sum_{|\beta| = k+1} \frac{k+1}{\beta!} (\xi-\eta)^\beta  \int_0^1 (1-t)^k \de^\beta \hat u(\eta+t(\xi-\eta)) \,\dd t;
	\]
	therefore, since $u \in \Sch(\R^n)$, for all $\xi,\eta \in \R^n$ and $N \in \N$,
	\begin{equation}\label{eq:sch_est_2}
	|R(\eta,\xi)| \lesssim_{u,N} |\xi-\eta|^{k+1} (1+\dist(0,[\eta,\xi]))^{-N},
	\end{equation}
	where $\dist(0,[\eta,\xi])$ is the distance to the origin of the line segment with endpoints $\eta$ and $\xi$.
	
	From the definition of $R^{u,q}_{k,\lambda}$ we have immediately that
	\begin{equation}\label{remainest1}
\frac{|R^{u,q}_{k,\lambda}(x)|}{\lambda^{n+k+1}}
	\lesssim \int_{\R^n} \int_{\R^n} |\hat q(x,\lambda(\xi-\eta),\lambda\xi)| \, |R(\eta,\xi)| \,\dd\eta \, \dd\xi.
\end{equation}
Notice next that, if we define 
	\[
	X = \{(\eta,\xi) \in \R^n \times \R^n \tc \min\{|\eta|,|\xi|\} \geq 2 \text{ and } |\eta-\xi|\geq 1\},
	\]
then, for $(\eta,\xi) \notin X$, we may compare
	\begin{equation}\label{eq:dist_comparison}
	1+\dist(0,[\eta,\xi]) \sim 1+\min\{|\eta|,|\xi|\}.
	\end{equation}
	We therefore split the integral in \eqref{remainest1} by decomposing the domain into $X$ and its complement.
	
	As for $X$, note first that $|\eta-\xi|\geq 1$ and $|\eta|\geq 2$ on $X$.
	Moreover, in view of \eqref{eq:sch_est_1}, we further decompose $X = X_1 \cup X_2$, where $X_1 = \{(\eta,\xi) \in X \tc |\xi|\leq |\eta|/2\}$, 
	and $X_2 = X \setminus X_1$. 
	Then, by \eqref{eq:sch_est_1}, on $X_1$ we have that  $|R(\eta,\xi)| \lesssim_{u,N}|\eta|^{k},$ and on $X_2$ we have that $ |R(\eta,\xi)| \lesssim_{u,N}|\eta|^{-N}$, where  we may assume $N$ to be sufficiently large. Therefore, in combination with \eqref{eq:tr_q_sch},  we see that 
\[\begin{split}
	&\iint_{X} |\hat q(x,\lambda(\xi-\eta),\lambda\xi)| \, |R(\eta,\xi)| \,\dd\eta \, \dd\xi \\
	&\lesssim  \iint_{X_1} (\lambda|\eta-\xi|)^{-N}|\eta|^{k} \,\dd\eta \, \dd\xi
	+ \iint_{X_2} (\lambda|\eta-\xi|)^{-N}|\eta|^{-N} \,\dd\eta \, \dd\xi\\
		&\lesssim  \iint_{|\eta|\ge2, \, |\xi|\le |\eta|/2} (\lambda|\eta|)^{-N}|\eta|^{k} \,\dd\eta \, \dd\xi
	+ \iint_{|\eta|\ge 2, \, |\eta-\xi|\ge 1} (\lambda|\eta-\xi|)^{-N}|\eta|^{-N} \,\dd\eta \, \dd\xi\\
	&\sim \lambda^{-N}.
\end{split}
\]
	
As for the complement of $X$, using
\eqref{eq:dist_comparison}, \eqref{eq:sch_est_2}, and \eqref{eq:tr_q_sch} with $N$ sufficiently large,  we find that 
\[\begin{split}
	&\iint_{\R^{2n} \setminus X} |\hat q(x,\lambda(\xi-\eta),\lambda\xi)| \, |R(\eta,\xi)| \,\dd\eta \, \dd\xi \\
	&\lesssim  \iint_{\R^{2n}} (1+\lambda|\eta-\xi|)^{-N} |\eta-\xi|^{k+1} (1+\min\{|\eta|,|\xi|\})^{-N} \,\dd\eta \, \dd\xi\\
	&\sim  \iint_{\R^{2n}} (1+\lambda|b|)^{-N} |b|^{k+1} (1+|a|)^{-N} \,\dd a \, \dd b \sim \lambda^{-(n+k+1)},
	\end{split}\]
	and we are done.
\end{proof}

Combining the previous result with the Fourier integral representation for the wave propagator of Theorem \ref{thm10030928}, we are now in a position to understand in a very precise way how the wave propagator acts on suitably defined bump functions $\tilde g_\la$ at scale $1/\la$, whose Fourier supports are essentially living in a frequency domain on which $|\xi|\sim \la$, $\la \gg 1$, and which are supported microlocally in narrow ``elliptic'' conic neighbourhoods of points at which the exponential mapping is non-degenerate.

These expressions will become particularly convenient for the subsequent applications of the method of stationary phase.

\begin{proposition}\label{prp:nonprincipal_estimates}
Let $(M,H,\mu)$ be a measured quadratic Hamiltonian pair and $\sLap$ be the corresponding sub-Laplacian. Assume that a self-adjoint extension of $\sLap$ has been chosen so that \eqref{RE}, \eqref{FPS} and \eqref{SP} are satisfied. Then there exist $\xi_* \in \nR^n$, $T \in \R^+$, a nonempty open $X \subset M$, a smooth function $w : (-T,T) \times X \times \nR^n \to \R$ $1$-homogeneous in the last variable, and functions $q_{j,\alpha} \in C^\infty( (-T,T) \times X \times \nR^n )$ for all $j \in \N$ and $\alpha \in \N^n$, such that the following hold true.
\begin{enumerate}[label=(\roman*)]
\item\label{en:nonprincipal_nonvanishingq} For all $t \in (-T,T)$, $x \in X$, $\xi_0 \in \R^+ \xi_*$,
\[
q_{0,0}(t,x,\xi_0) \neq 0.
\]
\item For all $t_0 \in (0,T)$ and $\tau_0 \in (0,\infty)$, there exist $x_0 \in X$ and $\xi_0 \in \R^+ \xi_*$ such that
\begin{gather*}
\partial_t w(t_0,x_0,\xi_0) = \tau_0, \qquad \partial_x w(t_0,x_0,\xi_0) \neq 0,  \\
\partial_\xi w(t_0,x_0,\xi_0) = 0, \qquad \rank \partial_\xi^2 w(t_0,x_0,\xi_0) = n-1.
\end{gather*}

\item If $t_0,\tau_0,x_0,\xi_0$ are as above, then there exist open neighbourhoods $U_0 \subset \nR^n$ of $\xi_0$, $J_0 \subset (0,T)$ of $t_0$ and $B_0 \subset X$ of $x_0$ such that, for all $g \in C_c^\infty(\R^n)$ with $\spt(g) \subset U_0$, there exist functions $\tilde g_\la \in C^\infty_c(M)$ for all $\la \geq 1$ such that
\[
\| \tilde g_\la \|_{L^p(M)} \lesssim_{g,p} \la^{n/p'}
\]
for all $\la \geq 1$ and $p \in [1,\infty]$, and moreover, for all $N \in \N$,
\begin{multline*}
\cos(t\sqrt{\sLap}) \tilde g_\la(x) \\
= \sum_{|\alpha| \leq N, j \leq N} \la^{n-|\alpha|-j} \int_{\nR^n} e^{i \la w(t,x,\xi)} q_{j,\alpha}(t,x,\xi) \,\partial^\alpha g(\xi) \,\dd\xi + O(\la^{n-N-1})
\end{multline*}
as $\la \to \infty$, uniformly in $t \in J_0$ and $x \in B_0$.
\end{enumerate}
\end{proposition}
\begin{proof}
By the assumption \eqref{RE}, we can find $o\in M$ and $\xi_* \in T^*_o M\setminus \{0\}$ such that $r\xi_*$ is a regular point for $\Exp^o_H$ for all $r\in[-1,1]\setminus\{0\}$.
Notice that $H(\xi_*) \neq 0$ by Lemma~\ref{lem10011320}.

Let $(M_o,x)$ be a coordinate chart centred at $o$. Let us identify $M_o$ with an open neighbourhood of the origin in $\R^n$. Let $\Gamma \subset \{ \xi \in \nR^n \tc H(0,\xi) \neq 0\}$ be any closed  cone in $\nR^n$ whose interior contains $\xi_*$. Let $w : (-T,T) \times X \times \nR^n \to \R$ be the function given by Theorem~\ref{thm10030928}, where $T>0$ and $X\subset M_o$ is an open neighbourhood of the origin.

Let $X''',X'',X'$ be open neighbourhoods of $0$ in $X$ such that $X''' \Subset X'' \Subset X' \Subset X$. Let $p_{\mathrm{sp}}\in C^\infty_c(M_o)$ be such that $p_{\mathrm{sp}}|_{X'''} \equiv 1$ and $\spt(p_{\mathrm{sp}})\subset X''$.
Let $p_{\mathrm{fr}} \in C^\infty(\Sphere^{n-1})$ be such that $p_{\mathrm{fr}}(\xi_*/|\xi_*|)=1$ and $\spt(p_{\mathrm{fr}}) \subset \Gamma \cap \Sphere^{n-1}$.
Then we can find $p \in S^0_\cl(M_o\times M_o;\R^n)$ such that $\spt(p) \subset X'' \times X'' \times \Gamma$, and all terms in the asymptotic expansion of $p$ vanish except for the $0$-homogeneous term
\[
p_0 : (x,y,\xi) \mapsto p_{\mathrm{sp}}(x) \, p_{\mathrm{sp}}(y) \, p_{\mathrm{fr}}(\xi/|\xi|).
\]
Let $P \in \Psi^0_\cl(M)$ be the pseudodifferential operator supported in $X'' \times X''$ and defined by \eqref{eq:PFIO}. Theorem \ref{thm10030928} then gives us an operator
$Q \in \Rop(M;(-T',T') \times M)$ supported in $(-T',T') \times X' \times X'$ for some $T' \in (0,T]$ and given by \eqref{eq:QFIO}, with amplitude $q \in S^0_\cl((-T',T')\times M_o \times M_o; \R^n)$ supported in $(-T',T') \times X' \times X' \times \Gamma$, such that,
for all $t \in (-T',T')$,
\begin{equation}\label{eq:cos_repn}
\cos(t\sqrt{\sLap}) P = (Q_t+Q_{-t})/2 + R_t
\end{equation}
for some smoothing operator $R : \Rop^{-\infty}(M;(-T',T')\times M)$ supported in $(-T',T') \times X' \times X'$.

Let $q \sim \sum_{j \geq 0} q_j$ be the asymptotic expansion of $q$. Note that, by construction, the $0$-homogeneous term $q_0$ equals the corresponding term $p_0$ in the expansion of $p$ for $t=0$,
and in particular $q_0(0,x,y,\xi_*) = 1$ for all $x,y \in X'''$. By continuity and homogeneity, up to taking a smaller $T'$, we may assume that
\begin{equation}\label{eq:q0nonvanishing}
q_0(t,x,y,r\xi_*) \neq 0
\end{equation}
for all $x,y \in X'''$, $r > 0$ and $t \in (-T',T')$.

Up to taking a smaller $T'$, we may also assume that $T' \leq 2\sqrt{H(0,\xi_*)}$, and the curve
\[
(-T',T') \ni t \mapsto \Exp_H^o(-t\xi_*/(2\sqrt{H(0,\xi_*)})) \in M
\]
takes values in $X'''$ and is injective; note that, by Lemma \ref{lem10011320}, this curve has nonvanishing tangent vector.
Hence, for all $t_0 \in (0,T')$, the point $-t_0 \xi_*/(2\sqrt{H(0,\xi_*)})$ is a regular point of $\Exp_H^o$ and, if we set $x_0 = \Exp^o_H(0,-t_0 \xi_*/(2\sqrt{H(0,\xi_*)}))$, then $X'''\ni x_0 \neq \Exp^o_H(0,t_0 \xi_*/(2\sqrt{H(0,\xi_*)}))$. So, from Theorem \ref{thm10030928}\ref{en:eitL_phder}-\ref{en:eitL_phcrit}-\ref{en:eitL_phhess} we deduce that, for all $\xi_0 \in \R^* \xi_*$,
\begin{gather}
\label{eq:critical_point}
\partial_\xi w(t_0,x_0,\xi_0) = 0, \quad  \rank(\partial_\xi^2 w(t_0,x_0,\xi_0)) = n-1, \\
\partial_t w(t_0,x_0,\xi_0) > 0, \quad \partial_x w(t_0,x_0,\xi_0) \neq 0, \\
\label{eq:noncritical_point}
\partial_\xi w(-t_0,x_0,\xi_0) \neq 0.
\end{gather}
By homogeneity of $w$, for all $\tau_0 \in (0,\infty)$, we can then choose $\xi_0 \in \R^+ \xi_*$ such that
\[
\partial_t w(t_0,x_0,\xi_0) = \tau_0.
\]

Let now $t_0,\gamma_0,x_0,\xi_0$ be as above.
If $\psi \in C^\infty_c(M_o)$ is such that $\psi|_{X'} \equiv 1$, we can define a linear map $\Pi : C^\infty(\R^n) \to C^\infty_c(M)$ by $\Pi f = \psi f$.
Let $g \in C^\infty_c(\R^n)$ with $\spt(g) \subset \nR^n$. For all $\la > 0$, let $\check g_\la \in \Sch(\R^n)$ denote the inverse Fourier transform of $g(\cdot/\la)$, and define $\tilde g_\la = P \Pi \check g_\la$. Note that
\[
\tilde g_\la(x) = p_{\mathrm{sp}}(x) \int_{\nR^n} \int_{\R^n} e^{i \xi \cdot (x-y)} p_{\mathrm{fr}}(\xi/|\xi|) \, p_{\mathrm{sp}}(y) \, \check g_\la(y) \,\dd y \,\dd\xi+ P^\infty \Pi \check g_\la (x),
\]
where $P^\infty \in \Psi^{-\infty}(M)$ is supported in
$X''\times X''$,
and, by Lemma \ref{lem06142004}, for all $N \in \N$,
\[
\int_{\nR^n} \int_{\R^n} e^{i \xi \cdot (x-y)} p_{\mathrm{fr}}(\xi/|\xi|) \, p_{\mathrm{sp}}(y) \, \check g_\la(y) \,\dd y \,\dd\xi = \la^n v_g(\la x) + O(\la^{-N})
\]
as $\la \to \infty$, where $v_g \in \Sch(\R^n)$ is given by $\hat v_g(\xi) = (2 \pi)^n p_{\mathrm{fr}}(\xi/|\xi|) g(\xi)$. Since
\[
\| \la^n v_g(\la \cdot) \|_{L^p(\R^n)} \lesssim_{g,p} \la^{n/p'}
\]
for all $p \in [1,\infty]$, we conclude that, for $\la$ sufficiently large,
\begin{equation}\label{eq:upperbound_norm}
\|\tilde g_\la\|_{L^p(M)} \lesssim_{g,p} \la^{n/p'} .
\end{equation}

By \eqref{eq:cos_repn}, for all $t \in (-T',T')$, we can write
\begin{equation}\label{eq:cos_split}
\cos(t\sqrt{\sLap}) \tilde g_\la = (\cc_{t,\la}^0 + \cc_{-t,\la}^0)/2 + \cc_{t,\la}^\infty,
\end{equation}
where
\[
\cc_{t,\la}^0 = Q_{t} \Pi \check g_\la, \qquad \cc_{t,\la}^\infty = R_t \Pi \check g_\la.
\]

Note now that 
\[
\cc_{t,\la}^{\infty}(x)
	= \int_{\R^n} R(t,x,y) \, \check g_\la(y) \,\dd y 
\]
for a smooth function $R \in C^\infty((-T',T') \times M \times \R^n)$ supported in $(-T',T') \times X' \times X'$. By taking Fourier transforms in $y$, we can rewrite this as
\[
\cc_{t,\la}^{\infty}(x)
	= \frac{\la^n}{(2\pi)^{n}} \int_{\R^n} \hat R(t,x,-\la\xi) \, g(\xi) \,\dd\xi ,
\]
where $\hat R$ denotes the partial Fourier transform of $R$ in the last variable. Since $X' \Subset \R^n$, the function $\hat R(t,x,\xi)$ has fast decay in $\xi$ uniformly in $x \in \R^n$ and $t \in [-T'',T'']$ for any $T'' < T'$, while $|\xi| \sim 1$ on $\spt(g)$. So, for all $T'' \in (0,T')$ and $N \in \N$,
\begin{equation}\label{eq:smoothing_decay}
\sup_{t \in [-T'',T'']} \|\cc_{t,\la}^\infty\|_{L^\infty(M)} \lesssim_{T'',N} (1+\la)^{-N}.
\end{equation}

As for the other terms in \eqref{eq:cos_split}, by \eqref{eq:QFIO} we can write
\[
\cc_{t,\la}^{0}(x)
	= \int_{\nR^n} \int_{\R^n} e^{i(w(t,x,\xi)-y\cdot\xi)} q(t,x,y,\xi) \, \check g_\la(y) \,\dd y \,\dd\xi .
\]

For every $\ell\in\N$, we have $q = \sum_{j=0}^\ell q_j + q^\ell$, where $q_j$ is homogeneous in $\xi$ of degree $-j$ and $q^\ell$ is an amplitude of order $-\ell-1$. Correspondingly
\[\begin{split}
 \cc_{t,\la}^0(x) 
	&= \sum_{j=0}^\ell \int_{\nR^n} \int_{\R^n} e^{i(w(t,x,\xi)-y\cdot\xi)} 
	q_j(t,x,y,\xi) \, \check g_\la(y) \,\dd y  \,\dd\xi   \\
	&+  \int_{\nR^n} \int_{\R^n} e^{i(w(t,x,\xi)-y\cdot\xi)} 
	q^\ell(t,x,y,\xi) \, \check g_\la(y) \,\dd y \,\dd\xi .
\end{split}\]

We apply Lemma \ref{lem06142004} to each term of the above sum and obtain
\begin{equation}\label{eq:dec_qm}
 \cc_{t,\la}^{0}(x) = \cc_{t,\la}^{1}(x) + \cc_{t,\la}^{2}(x) + \cc_{t,\la}^{3}(x),
\end{equation}
where
\begin{align*}
\cc_{t,\la}^{1}(x) &= \sum_{\substack{|\alpha|\le k\\ j\le\ell}} \frac{\la^{n-|\alpha|-j}}{i^{-|\alpha|} \alpha!}
	\int_{\nR^n} e^{i \la w(t,x,\xi)} \de^\alpha_y q_j(t,x,0,\xi) \, \de^\alpha g(\xi) \,\dd\xi ,\\
\cc_{t,\la}^{2}(x) &= \sum_{|\alpha|\le k} \frac{\la^{n-|\alpha|}}{i^{-|\alpha|} \alpha!}
	\int_{\nR^n} e^{i \la w(t,x,\xi)} \de^\alpha_y q^{\ell}(t,x,0,\la\xi) \, \de^\alpha g(\xi) \,\dd\xi  ,\\
\cc_{t,\la}^{3}(x) &= \la^{n-(k+1)} 
	\sum_{j\le\ell+1} R_{\la,j,k}(t,x)   ,
\end{align*}
for some functions $R_{\la,j,k}$ with
 $\sup_{x \in \R^n, t\in[-T'',T'']} |R_{\la,j,k}(t,x)| \lesssim_{u,T''} 1$ for all $T'' \in (0,T')$. In particular, for all $T'' \in (0,T')$ and $\la \geq 1$,
\begin{equation}\label{eq:decay_m3}
\sup_{t \in [-T'',T'']} \|\cc_{t,\la}^{3}\|_{L^\infty(M)} \lesssim_{g,T''} \la^{n-(k+1)}.
\end{equation}
Moreover, since $\de^\alpha_y q^{\ell}$ is an amplitude of order $-\ell-1$ and $|\xi| \sim 1$ on $\spt(\de^\alpha g)$, we easily obtain that, for all $\la \geq 1$ and $T'' \in (0,T')$,
\begin{equation}\label{eq:decay_m2}
\sup_{t \in [-T'',T'']} \|\cc_{t,\la}^{2}\|_{L^\infty(M)} \lesssim_{g,T''} \la^{n-(\ell+1)}.
\end{equation}

By \eqref{eq:noncritical_point}, there exist open neighbourhoods $J_0 \Subset (0,T')$ of $t_0$, $U_0 \Subset \nR^n$ of $\xi_0$ and $B_0 \Subset X'''$ of $x_0$ such that
\[
|\partial_\xi w(-t,x,\xi)| \geq |\partial_\xi w(-t_0,x_0,\xi_0)|/2 > 0
\]
for all $t \in J_0$, $x \in B_0$ and $\xi \in U_0$. Hence, if we choose $g$ so that $\spt(g) \subset U_0$, then integration by parts in $\xi$ (see, e.g., \cite[Theorem 7.7.1]{MR717035}) immediately gives that, for all $\la>0$ and $N \in \N$,
\begin{equation}\label{eq:decay_m1minus}
\sup_{x \in B_0, t \in J_0} |\cc_{-t,\la}^{1}(x)| \lesssim_{g,N} (1+\la)^{-N}.
\end{equation}

By combining the above estimates, we obtain that
\begin{multline*}
\cos(t\sqrt{\sLap}) \tilde g_\la(x) \\
= \sum_{\substack{|\alpha|\le k\\ j\le\ell}} \frac{\la^{n-|\alpha|-j}}{2 i^{-|\alpha|} \alpha!} \int_{\nR^n} e^{i \la w(t,x,\xi)} \partial_y^\alpha q_{j}(t,x,0,\xi) \,\partial^\alpha g(\xi) \,\dd\xi + O(\la^{n-\min\{k,\ell\}-1})
\end{multline*}
as $\la \to \infty$, uniformly in $x \in B_0$ and $t \in J_0$. The conclusion follows by setting $q_{j,\alpha}(t,x,\xi) = (2i^{-|\alpha|} \alpha!)^{-1} \de^\alpha_y q_j(t,x,0,\xi)$, taking $k=\ell=N$, and relabeling $T'$ as $T$ and $X'''$ as $X$. 
\end{proof}

Finally, we state a simple application of the method of stationary phase that will be of use in the sequel.

\begin{lemma}\label{lem:stationaryphase}
Let $I \subset \R$ and $X \subset \R^n$ be open, and let $w : I \times X \times \nR^n \to \R$ be smooth and $1$-homogeneous in the last variable. Assume that there exist $t_0 \in I \setminus \{0\}$, $x_0 \in X$ and $\xi_0 \in \nR^n$ such that
\begin{equation}\label{eq:statphase_ass}
\partial_t w(t_0,x_0,\xi_0) = t_0, \quad \partial_\xi w(t_0,x_0,\xi_0) = 0, \quad \rank \partial_\xi^2 w(t_0,x_0,\xi_0) = n-1.
\end{equation}
Then there exist $\sigma \in \Z$, open neighbourhoods $B \Subset X$ of $x_0$, $U \Subset \nR^n$ of $\xi_0$ and $J \Subset I \setminus \{0\}$ of $t_0$, and smooth functions $t^c : B \to J$, $\xi^c : B \to U$and $d : B \to \R^+$ such that
\[
t^c(x_0) = t_0, \qquad \xi^c(x_0) = \xi_0
\]
and, for all smooth functions $b : \R \times \R^n \times \R^n \to \C$ with $\spt(b) \subset J \times B\times U$,
	\begin{multline*}
	\int_{\R} \int_{\nR^{n}} e^{i\la[w(t,x,\xi)-t^2/2]} \, b(t,x,\xi) \,\dd\xi \,\dd t\\
	= \la^{-(n+1)/2} \, d(x) \, e^{i(\pi\sigma/4 - \la(t^c(x))^2/2)} \, b(t^c(x),x,\xi^c(x))   
	+ O(\la^{-(n+1)/2-1}) ,
	\end{multline*}
	as $\la\to\infty$, uniformly in $x \in B$.
\end{lemma}
\begin{proof}
We want to apply the method of stationary phase to the above integral, with phase $f(x,t,\xi) = w(t,x,\xi) - t^2/2$, where $x$ plays the role of a parameter. Observe that
\[
\partial_{(t,\xi)} f = \begin{pmatrix} \partial_t w - t \\ 
\partial_\xi w \end{pmatrix}, \quad
\partial_{(t,\xi)}^2 f = \begin{pmatrix}
\partial_t^2 w - 1 & \partial_t\partial_\xi w^T \\
\partial_t\partial_\xi w & \partial_\xi^2 w
\end{pmatrix},
\]
so, by \eqref{eq:statphase_ass}, $\partial_{(t,\xi)} f(x_0,t_0,\xi_0) = 0$. In addition, from the $1$-homogeneity of $w$ we deduce that $\xi \cdot \partial_\xi w(t,x,\xi) = w(t,x,\xi)$, so
\begin{equation}\label{eq:w_mixed_der}
\xi \cdot \partial_t \partial_\xi w(t,x,\xi) = \partial_t w(t,x,\xi)
\end{equation}
and
\begin{equation}\label{eq:w_hess_ker}
\xi \in \ker \partial_\xi^2 w(t,x,\xi).
\end{equation}

Therefore, if we write the matrix of $\partial_\xi^2 w(t,x,\xi)$ with respect to the decomposition $\R^n = \R (\xi/|\xi|) \oplus \xi^\perp$, then, by \eqref{eq:w_hess_ker},
\[
\partial_\xi^2 w(t,x,\xi) = \begin{pmatrix}
0 & 0 \\
0 & \partial_\xi^2 w(t,x,\xi)|_{\xi^\perp \times \xi^\perp}
\end{pmatrix},
\]
which, together with \eqref{eq:w_mixed_der}, implies that
\[
\partial_{(t,\xi)}^2 f(x,t,\xi) = \begin{pmatrix}
\partial_t^2 w(t,x,\xi) - 1 & \partial_t w(t,x,\xi)/|\xi| & * \\
\partial_t w(t,x,\xi)/|\xi| &  0 & 0 \\
* & 0 & \partial_\xi^2 w(t,x,\xi)|_{\xi^\perp \times \xi^\perp}
\end{pmatrix}.
\]
In particular, from \eqref{eq:statphase_ass} we deduce that
\begin{gather*}
\det \partial_\xi^2 w(t_0,x_0,\xi_0)|_{\xi_0^\perp \times \xi_0^\perp} \neq 0, \\
\det \partial_{(t,\xi)}^2 f(x_0,t_0,\xi_0) = -(t_0/|\xi_0|)^2 \det \partial_\xi^2 w(t_0,x_0,\xi_0)|_{\xi_0^\perp \times \xi_0^\perp} \neq 0
\end{gather*}
(recall that $t_0 \neq 0$).
Consequently, by the implicit function theorem, there are open neighbourhoods $B \Subset X$ of $x_0$, $J \Subset I \setminus \{0\}$ of $t_0$ and $U \Subset \Gamma$ of $\xi_0$, and smooth functions $t^c : B \to J$, $\xi^c : B \to U$ such that $\det \partial_{(t,\xi)}^2 f(x,t,\xi) \neq 0$ for $(x,t,\xi) \in B \times J \times U$ and
	\[
	\{(x,t,\xi) \in B \times J \times U \tc \partial_{(t,\xi)} f(x,t,\xi)=0\}
	= \{(x,t^c(x),\xi^c(x)) \tc x \in B\} .
	\]
If $\sigma \in \Z$ is the signature of $\partial_{(t,\xi)}^2 f(x_0,t_0,\xi_0)$ and we define $d : B \to \R^+$ by
\[
d(x) = (2\pi)^{(n+1)/2} \left|\det\partial_{(t,\xi)}^2 f(x,t^c(x),\xi^c(x))\right|^{-1/2},
\]
then, up to shrinking the neighbourhoods $B,J,U$, the conclusion follows by \cite[Thm 7.7.6]{MR717035}.
\end{proof}

\subsection{Mihlin--H\"ormander estimates}

For all $\epsilon>0$, let $\CO$ be the set of the real-valued functions $\chi \in C^\infty_c(\R)$ with $\spt(\chi) \subset (0,\infty)$.

Let $\chi\in \CO$.
For $\la\in \nR$, define
\begin{equation}\label{eq:ms_cos_sub}
m_\la^\chi(s) 
= |\la|^{1/2} \int_\R \chi(|t|) \, e^{i(s t-\la t^2/2)} \,\dd t 
= 2|\la|^{1/2} \int_\R \chi(t) \, e^{-i\la t^2/2} \cos(s t) \,\dd t .
\end{equation}
Note that $m_\la^\chi \in \Sch(\R)$ is even.  
Moreover, by the method of stationary phase, $m_\la^\chi(s)$  is essentially of the form $\tilde \chi(|s/\la|) \, e^{is^2/(2\la)}$, with $\tilde\chi\in \CO$, i.e., $m_\la^\chi(s)$ is essentially a ``Schr\"odinger multiplier'' at time $\sim 1/\lambda$, spectrally localised where $|s| \sim |\la|$.

A simple stationary phase argument (exploiting, e.g., \cite[Section VIII.1.2]{MR1232192}) yields, for all $k \in \N$ and $\la \in \nR$,
\[
\sup_{s \in \R} |s^k \partial_s^k m^\chi_\la(s)| \lesssim_{\chi,k} (1+|\la|)^k,
\]
whence, by interpolation, we also deduce that, for all $\alpha \in [0,\infty)$ and $\la \in \nR$,
\begin{equation}\label{eq:mhnorm}
\| m^\chi_\la \|_{L^2_{\alpha,\sloc}} \lesssim_{\chi,\alpha} (1+|\la|)^\alpha.
\end{equation}

In view of this estimate, it is clear that the next result proves Theorem \ref{thm:main} under certain regularity assumptions, introduced in Section \ref{sec:subriemannian}.

\begin{theorem}\label{thm:MHlowerbd}
Let $(M,H,\mu)$ be a measured quadratic Hamiltonian pair of dimension $n$, and $\sLap$ the corresponding sub-Laplacian. Assume that a self-adjoint extension of $\sLap$ has been chosen so that \eqref{RE}, \eqref{FPS} and \eqref{SP} are satisfied. Then there exist $\chi \in \CO$ and $\la_0 >0$ such that, for all $p \in [1,\infty]$ and $\la \in \nR$ with $|\la| \geq \la_0$,
\[
\| m^\chi_\la(\sqrt{\sLap}) \|_{p \to p} \gtrsim_p |\la|^{n|1/p-1/2|}.
\]
\end{theorem}
\begin{proof}
By \eqref{eq:adjointcalculus}, since $\overline{m_\la^\chi} = m_{-\la}^\chi$, it is enough to prove the theorem for $p \in [1,2]$ and $\la > 0$, which we from now on assume.

Let $\xi_* \in \nR^n$, $w : (-T,T) \times X \times \nR^n \to \R$, $q_{j,\alpha} \in C^\infty((-T,T) \times X'\times \nR^n)$ be given by Proposition \ref{prp:nonprincipal_estimates}. Let us take any $t_0 \in (0,T)$ and let $x_0 \in X$ and $\xi_0 \in \R^+ \xi_*$ be given by Proposition \ref{prp:nonprincipal_estimates} so that
\begin{gather*}
\partial_t w(t_0,x_0,\xi_0) = t_0, \qquad \partial_x w(t_0,x_0,\xi_0) \neq 0,  \\
\partial_\xi w(t_0,x_0,\xi_0) = 0, \qquad \rank \partial_\xi^2 w(t_0,x_0,\xi_0) = n-1.
\end{gather*}
Let then $B_0 \subset X$, $U_0 \subset \nR^n$ and $J_0 \subset (0,T)$ be given by Proposition \ref{prp:nonprincipal_estimates}. For all $g \in C^\infty_c(\R^n)$ with $\spt(g) \subset U_0$, and all $N \in \N$, we then have
\begin{multline}\label{eq:cos_FIO}
\cos(t\sqrt{\sLap}) \tilde g_\la(x) \\
= \sum_{|\alpha| \leq N, j \leq N} \la^{n-|\alpha|-j} \int_{\nR^n} e^{i \la w(t,x,\xi)} q_{j,\alpha}(t,x,\xi) \,\partial^\alpha g(\xi) \,\dd\xi + O(\la^{n-N-1})
\end{multline}
as $\la \to \infty$, uniformly in $t \in J_0$ and $x \in B_0$, where $\{\tilde g_\la\}_{\la \geq 1} \subset C^\infty_c(M)$ satisfies
\begin{equation}\label{eq:MH_upperbound_norm}
\| \tilde g_\la \|_{L^p(M)} \lesssim_{g,p} \la^{n/p'}.
\end{equation}

Assume that $\chi \in \CO$ and $\spt(\chi) \subset J_0$.
By \eqref{eq:ms_cos_sub} and \eqref{eq:cos_FIO}, 
\begin{equation}\label{eq:MH_nonprincipal_dec}
m_\la^\chi(\sqrt{\sLap}) \tilde g_\la(x) 
= 2 \sum_{|\alpha| \leq N, j \leq N} \la^{n+1/2-|\alpha|-j} \mm_{\la,\alpha,j}(x) + O(\la^{n-N-1/2})
\end{equation}
as $\la \to \infty$, uniformly in $x \in B_0$,
where
\[
\mm_{\la,\alpha,j}(x) = 
\int_\R \int_{\nR^n} e^{i \la (w(t,x,\xi)-t^2/2)} \chi(t) \, q_{j,\alpha}(t,x,\xi) \,\partial^\alpha g(\xi) \,\dd\xi \,\dd t .
\]

Let $J \Subset J_0$, $B \Subset B_0$, $U \Subset U_0$ be the open neighbourhoods of $t_0$, $x_0$, $\xi_0$ given by Lemma \ref{lem:stationaryphase} applied to the function $w$. If $\chi$ and $g$ are chosen so that $\spt(\chi) \subset J$ and $\spt(g) \subset U$, then Lemma \ref{lem:stationaryphase} implies that,
for all $\la \geq 1$,
\begin{equation}\label{eq:decay_m1plus}
\sup_{x \in B} |\mm_{\la,\alpha,j}(x)| \lesssim_{\chi,g} \la^{-(n+1)/2},
\end{equation}
and moreover
\[
\la^{(n+1)/2} |\mm_{\la,0,0}(x)| = 	d(x) \, |\chi(t^c(x))| \, |q_{0,0}(t^c(x),x,\xi^c(x))| \, |g(\xi^c(x))| + O(\la^{-1})
\]
as $\la \to \infty$, uniformly in $x \in B$, where $t^c : B \to J$, $\xi^c : B \to U$, $d : B \to \R^+$ are smooth functions with $t^c(x_0) = t_0$, $\xi^c(x_0) = \xi_0$.
If we choose $\chi$ and $g$ so that $\chi(t_0) \neq 0$ and $g(\xi_0) \neq 0$, then, by Proposition \ref{prp:nonprincipal_estimates}\ref{en:nonprincipal_nonvanishingq},
\[
|\chi(t^c(x_0))| \, |q_{0,0}(t^c(x_0),x_0,\xi^c(x_0))| \, |g(\xi^c(x_0))| \neq 0.
\]
Hence, if we choose a sufficiently small neighbourhood $B'\subset B$ of $x_0$, then there exists $\la_0 \geq 1$ such that, for all $\la \geq \la_0$,
\begin{equation}\label{eq:main_lowerbd}
\inf_{x \in B'} |\mm_{\la,0,0}(x)| \gtrsim_{\chi,g} \la^{-(n+1)/2}.
\end{equation}
By combining the above estimates \eqref{eq:MH_nonprincipal_dec}, \eqref{eq:decay_m1plus} and \eqref{eq:main_lowerbd} and choosing $N$ large enough, we obtain that, up to taking a larger $\la_0$, for all $\la \geq \la_0$ and $p \in [1,\infty]$,
\[
\|m_\la^\chi(\sqrt{\sLap}) \tilde g_\la\|_{L^p(M)} \gtrsim_{p,B'}\inf_{x \in B'} |m_\la^\chi(\sqrt{\sLap}) \tilde g_\la(x)| \gtrsim_{\chi,g} \la^{n/2}.
\]
Combining this with \eqref{eq:MH_upperbound_norm}, 
we conclude that
\[
\|m_\la^\chi(\sqrt{\sLap})\|_{L^p(M) \to L^p(M)} \gtrsim_{\chi,g,p} \la^{n(1/2-1/p')} = \la^{n(1/p-1/2)},
\]
as desired.
\end{proof}

\subsection{Miyachi--Peral estimates}
Let $\Sch_e$ be the set of all even, real-valued Schwartz functions on $\R$ that are not identically zero.
For $\chi\in\Sch_e$ and $\la,t > 0$, define
\begin{equation}\label{eq:wave_mult}
\begin{split}
m_{\la,t}^\chi(s) 
&\defeq \chi(s/\la)\, \cos\left(ts \right) \\
&= \frac{1}{2\pi} \int_\R \hat\chi(\tau) \cos((t+\tau/\la)s) \,\dd \tau ,
\end{split}
\end{equation}
where the second equality follows from the Fourier inversion and pro\-stha\-phae\-resis formulas. The following result proves Theorem \ref{thm:main}\ref{en:main_MP} under the assumptions introduced in Section \ref{sec:subriemannian}.
 
\begin{theorem}\label{thm11011139}
	Let $(M,H,\mu)$ be a measured quadratic Hamiltonian pair of dimension $n$ and $\sLap$ the corresponding sub-Laplacian.  Assume that a self-adjoint extension of $\sLap$ has been chosen so that \eqref{RE}, \eqref{FPS}, \eqref{SP} and \eqref{SFC} are satisfied. 
	Then there exists $t_*>0$ such that, for all $t_0 \in (0,t_*]$ and all $\chi \in \Sch_e$, there exists $\la_0 >0$ such that, for all $p \in [1,\infty]$ and $\la \geq \la_0$,
	\begin{equation}\label{eq10201104}
	\|m_{\la,t_0}^\chi (\sqrt{\sLap})\| _{p\to p} \gtrsim_{\chi,t_0,p}  \la^{(n-1)\left|1/p-1/2\right|} .
	\end{equation}
\end{theorem}
\begin{proof}
By \eqref{eq:adjointcalculus}, since $m^\chi_{\la,t}$ is real-valued, it is enough to prove the theorem for $p \in [1,2]$.

Fix $\eta\in(0,1/2)$ and a smooth even function $\rho : \R \to \R$ such that $\spt(\rho) \subset (-1,1)$ and $\rho(0) = 1$.
For $\chi\in \Sch_e$ and $t > 0$, define
\begin{align}
m_{\la,t}^{\chi,0} (s) 	&\defeq \frac{1}{2\pi} \int_\R \rho(\tau \la^{-\eta}) \hat\chi(\tau) \cos((t+\tau/\la)s) \,\dd \tau, \label{eq:wavemult}\\
m_{\la,t}^{\chi,\infty}(s) 	&\defeq m_{s,t}^{\chi}(s) - m_{\la,t}^{\chi,0}(s)	= \frac{1}{2\pi} \int_\R \left(1-\rho(\tau \la^{-\eta})\right) \hat\chi(\tau) \cos((t+\tau/\la)s) \,\dd \tau .\notag
\end{align}
Since $1-\rho(\tau \la^{-\eta})=0$ for $\tau\in[-\la^\eta,\la^\eta],$ and since $\hat\chi \in \Sch(\R)$ is rapidly decreasing, by means of integrations by parts one can easily show that, for all $\alpha,\beta,N\in\N$,
\[
	\sup_{s \in \R} |s^\beta\de_s^\alpha m_{\la,t}^{\chi,\infty}(s)| \lesssim_{\alpha,\beta,N,t} (1+\la)^{-N};
\]
consequently, since $m_{\la,t}^{\chi,\infty}$ is even, $m_{\la,t}^{\chi,\infty}(\sqrt{\cdot})$ extends \cite{MR0007783} to a Schwartz function $\tilde m_{\la,t}^{\chi,\infty}$ on $\R$ satisfying
\[
	\sup_{s \in \R} |s^\beta\de_s^\alpha \tilde m_{\la,t}^{\chi,\infty}(s)| \lesssim_{\alpha,\beta,N,t} (1+\la)^{-N}.
\]
Therefore, by \eqref{SFC},
for all $p\in[1,\infty]$ and $N \in \N$,
\begin{equation}\label{eq10201110}
\| m_{\la,t}^{\chi,\infty}(\sqrt{\sLap}) \|_{p\to p} \lesssim_{N,t} (1+\la)^{-N},
\end{equation}
so it will be enough to prove the desired lower bound for $m_{\la,t}^{\chi,0}(\sqrt{\sLap})$ instead of $m_{\la,t}^{\chi}(\sqrt{\sLap})$.

Let $\xi_* \in \nR^n$, $w : (-T,T) \times X \times \nR^n \to \R$, $q_{j,\alpha} \in C^\infty((-T,T) \times X'\times \nR^n)$ be given by Proposition \ref{prp:nonprincipal_estimates}. 
Set $t_* = T/2$.

Let $t_0 \in (0,t_*]$ and $\chi \in \Sch_e$. Then there exists $\tau_0 > 0$ such that
\begin{equation}\label{eq:w_nonvanishing_chi}
\chi(\tau_0) \neq 0.
\end{equation}
Let $x_0 \in X$ and $\xi_0 \in \R^+ \xi_*$ be given by Proposition \ref{prp:nonprincipal_estimates} so that
\begin{gather}
\label{eq:w_phase_derivatives}
\partial_t w(t_0,x_0,\xi_0) = \tau_0, \qquad \partial_x w(t_0,x_0,\xi_0) \neq 0,  \\
\label{eq:w_critical_point} 
\partial_\xi w(t_0,x_0,\xi_0) = 0, \qquad \rank \partial_\xi^2 w(t_0,x_0,\xi_0) = n-1.
\end{gather}
Let then $B_0 \subset X$, $U_0 \subset \nR^n$ and $J_0 \subset (0,T)$ be given by Proposition \ref{prp:nonprincipal_estimates}. For all $g \in C^\infty_c(\R^n)$ with $\spt(g) \subset U_0$, and all $N \in \N$, we then have
\begin{multline}\label{eq:w_cos_FIO}
\cos(t\sqrt{\sLap}) \tilde g_\la(x) \\
= \sum_{|\alpha| \leq N, j \leq N} \la^{n-|\alpha|-j} \int_{\nR^n} e^{i \la w(t,x,\xi)} q_{j,\alpha}(t,x,\xi) \,\partial^\alpha g(\xi) \,\dd\xi + O(\la^{n-N-1})
\end{multline}
as $\la \to \infty$, uniformly in $t \in J_0$ and $x \in B_0$, where $\{\tilde g_\la\}_{\la \geq 1} \subset C^\infty_c(M)$ satisfies
\begin{equation}\label{eq:w_upperbound_norm}
\| \tilde g_\la \|_{L^p(M)} \lesssim_{g,p} \la^{n/p'}.
\end{equation}

Let $\la_0 > 0$ be sufficiently large so that $[t_0-2\la_0^{\eta-1},t_0+2\la_0^{\eta-1}] \subset J_0$.
In view of \eqref{eq:wavemult} and \eqref{eq:w_cos_FIO},
since $|\tau|/\la \leq \la^{\eta-1}$ where $\rho(\tau \la^{-\eta}) \neq 0$, for all $\la \geq \la_0$ we can write
\begin{equation}\label{eq:wave_zerosplitting}
m_{\la,t_0}^{\chi,0}(\sqrt{\sLap}) \tilde g_\la(x) = \frac{1}{2\pi} \sum_{|\alpha| \leq N, j \leq N} \la^{n-|\alpha|-j} \mm_{\la,\alpha,j}(x) + O(\la^{n-N-1})
\end{equation}
as $\la \to \infty$, uniformly in $x \in B_0$, 
where
\begin{equation}\label{eq:w_main_splitting_int}
\mm_{\la,\alpha,j}(x) 
= \int_\R  \rho(\tau \la^{-\eta}) \, \hat\chi(\tau) \, F_{\la,\alpha,j}(\tau/\la,x) \,\dd\tau
\end{equation}
and
\[
F_{\la,\alpha,j}(t,x) = \int_{\R^n} e^{i \la w(t_0+t,x,\xi)} q_{j,\alpha}(t_0+t,x,\xi) \, \de^\alpha g(\xi) \,\dd\xi .
\]
Let us write $\xi = r(\xi_0 + \zeta)$, where $\zeta \in \xi_0^\perp$ and $r \in \R$; then
\[
F_{\la,\alpha,j}(t,x) = |\xi_0| \int_{\R} \int_{\xi_0^\perp} e^{i \la w_0(r,t,x,\eta)} b_{\alpha,j}(r,t,x,\zeta) \,\dd\zeta \, |r|^{n-1} \,\dd r ,
\]
where
\begin{align*}
w_0(r,t,x,\zeta) &=  r w(t_0+t,x,\xi_0+\zeta), \\
 b_{\alpha,j}(r,t,x,\zeta) &= q_{j,\alpha}(t_0+t,x,r(\xi_0+\zeta)) \, \de^\alpha g(r(\xi_0+\zeta)).
\end{align*}
Note that $r$ and $\zeta$ can be made arbitrarily close to $1$ and $0$ respectively in the domain of integration, by taking the support of $g$ sufficiently close to $\xi_0$.
Moreover,
\[
\partial_\zeta w_0(r,t,x,0) = r \partial_\xi w(t_0+t,x,\xi_0)|_{\xi_0^\perp}, \quad \partial_\zeta^2 w_0(r,t,x,0) = r \partial_\xi^2 w(t_0+t,x,\xi_0)|_{\xi_0^\perp \times \xi_0^\perp}.
\]

Thanks to the assumptions \eqref{eq:w_critical_point} we can apply the method of stationary phase \cite[Thm 7.7.6]{MR717035} to the integral in $\zeta$. Indeed, since $\xi_0 \in \ker \partial_\xi^2 w(t_0,x_0,\xi_0)$ by homogeneity (cf.\ \eqref{eq:w_hess_ker}), from \eqref{eq:w_critical_point} it follows that $\partial_\xi w(t_0,x_0,\xi_0) = 0$ and $\partial_\xi^2 w(t_0,x_0,\xi_0)|_{\xi_0^\perp \times \xi_0^\perp}$ is nondegenerate.
So the implicit function theorem yields open neighbourhoods $B \Subset B_0$ of $x_0$ and $I \Subset J_0 - t_0$ of $0$, and a smooth function $\zeta^c : I \times B \to \xi_0^\perp$ such that $\zeta^c(0,x_0) = 0$ and $\partial_\xi w(t_0+t,x,\xi_0+\zeta^c(t,x))|_{\xi_0^\perp} = 0$, and moreover (up to shrinking $I$ and $B$ and choosing $\spt(g)$ sufficiently close to $\xi_0$),
\begin{equation}\label{eq:wave_stationaryphase}
\la^{(n-1)/2} \, F_{\la,\alpha,j}(t,x) =  d(t,x) \int_{\R} e^{i \la r w^c(t,x)} b_{\alpha,j}^c(r,t,x)  \, |r|^{(n-1)/2} \,\dd r + O(\la^{-1})
\end{equation}
as $\la \to \infty$, uniformly in $x \in B$ and $t \in I$,
where
\begin{gather*}
w^c(t,x) = w(t_0+t,x,\xi_0 + \zeta^c(t,x)), \qquad b_{\alpha,j}^c(r,t,x) = b_{\alpha,j}(r,t,x,\zeta^c(t,x)), \\
 d(t,x) = (2\pi)^{(n-1)/2} \, e^{i\pi \sigma/4} \, |\xi_0|  \, |\det(\partial_\xi^2 w(t_0+t,x,\xi_0+\zeta^c(t,x))|_{\xi_0^\perp \times \xi_0^\perp})|^{-1/2}
\end{gather*}
and $\sigma$ is the signature of $\partial_\xi^2 w(t_0,x,\xi_0+\zeta^c(x))|_{\xi_0^\perp \times \xi_0^\perp}$. Note that
\begin{equation}\label{eq:w_criticalphase_der}
\partial_t w^c(0,x_0) = \partial_t w(t_0,x_0,\xi_0), \qquad  \partial_x w^c(0,x_0) = \partial_x w(t_0,x_0,\xi_0)
\end{equation}
(since $\partial_\xi w(t_0,x,\xi_0+\zeta^c(t_0,x))|_{\xi_0^\perp} = 0$ by construction).

By plugging the above estimate into \eqref{eq:w_main_splitting_int} and using the fact that $\hat\chi \in \Sch(\R)$, $\eta < 1$ and $|\tau|/\la \leq \la^{\eta-1}$ in the domain of integration, it is immediately deduced that, for $x \in B$ and $\la$ sufficiently large,
\begin{equation}\label{eq:wave_stationaryphase_estimates}
|\mm_{\la,\alpha,j}(x)| \lesssim \la^{-(n-1)/2}.
\end{equation}

We want now to obtain the reverse inequality in the case $\alpha=0$ and $j=0$. Note that a Taylor expansion of $w^c$ around $t=0$ yields
\[
w^c(t,x) = w^c(0,x) + t \partial_t w^c(0,x) + t^2 W(t,x) 
\]
for some smooth function $W : I \times B \to \R$,
and similarly
\[
e^{ax} = 1+ aE(a,x)
\]
for some smooth function $E : \R^2 \to \R$, so
\begin{equation}\label{eq:wave_phase_expansion}
e^{ir\la w^c(\tau/\la,x)} = e^{ir (\la w^c(0,x) + \tau \partial_t w^c(0,x))} (1 + (\tau^2/\la) \tilde W(\tau^2/\la,\tau/\la,x)),
\end{equation}
where $\tilde W(a,t,x) = E(a,W(t,x))$. Since $|\tau|/\la \leq \la^{\eta-1}$ and $\tau^2/\la \leq \la^{2\eta-1}$ whenever $\rho(\tau \la^{-\eta}) \neq 0$, and $\eta <1/2$, from \eqref{eq:w_main_splitting_int}, \eqref{eq:wave_stationaryphase} and \eqref{eq:wave_phase_expansion} we deduce that, as $\la \to \infty$,
\begin{multline}\label{eq:wave_mainmain}
\la^{(n-1)/2} \, \mm_{\la,0,0}(x) \\
= d(0,x) \int_\R e^{ir \la w^c(0,x)} A(\la^{-\eta},r,x) \,  g(r(\xi_0+\zeta^c(0,x)))  \,\dd r + O(\la^{\eta-1}),
\end{multline}
where
\[
A(\nu,r,x) = G(\nu,r \partial_t w^c(0,x)) \, q_{0,0}(t_0,x,r(\xi_0+\zeta^c(0,x))) \,|r|^{(n-1)/2} 
\]
and
\[
G(\nu,t) = \int_\R \rho(\nu \tau) \, \hat\chi(\tau) \, e^{i \tau t} \,d\tau.
\]

In order to obtain the desired lower bound for $\mm_{\la,0,0}(x)$, we need to ensure that there is no cancellation in the integral in \eqref{eq:wave_mainmain}.
Note that $G : \R^2 \to \C$ is continuous and $G(0,t) = 2\pi \chi(t)$, because $\rho(0) = 1$. Consequently, by Proposition \ref{prp:nonprincipal_estimates}\ref{en:nonprincipal_nonvanishingq}, \eqref{eq:w_nonvanishing_chi}, \eqref{eq:w_phase_derivatives} and \eqref{eq:w_criticalphase_der},
\[
A(0,1,x_0) = 2\pi \chi(\partial_t w(t_0,x_0,\xi_0)) \, q_{0,0}(t_0,x_0,\xi_0) \neq 0.
\]
Hence, if we choose $g$ supported sufficiently close to $\xi_0$ and let $B'\subset B$ be a sufficiently small neighbourhood of $x_0$, then, for all $x \in B'$  and $\la$ sufficiently large,
\[
r \leq 2, \quad |A(\la^{-\eta},r,x) - A(0,1,x_0)| \leq 10^{-10} |A(0,1,x_0)|, \quad d(0,x) \geq d(0,x_0)/2
\]
in the domain of integration.
In addition, if we assume that $g \geq 0$ and $g(\xi_0) > 0$, then, up to shrinking $B'$,
\[
\inf_{x \in B'} \int_\R g(r(\xi_0+\zeta^c(0,x)))  \,\dd r > 0.
\]

In conclusion, in order to avoid cancellation in the integral in \eqref{eq:wave_mainmain}, it is enough to ensure that $|w^c(0,x)-w^c(0,x_0)| \leq 10^{-10} s^{-1}$. On the other hand,
by homogeneity of $w$ and \eqref{eq:w_critical_point},
$w^c(0,x_0) = \xi_0 \cdot \partial_\xi w(t_0,x_0,\xi_0) = 0$,
and moreover, by \eqref{eq:w_phase_derivatives} and \eqref{eq:w_criticalphase_der}, $\partial_x w^c(0,x_0) \neq 0$.
This shows that $w^c(0,\cdot)$ vanishes on a smooth hypersurface $S \Subset B'$ passing through $x_0$, and consequently, for all sufficiently small $\epsilon > 0$, $|w^c(0,x)| \lesssim \epsilon$ for all $x$ in an $\epsilon$-neighbourhood $S_\epsilon$ of $S$. Hence, if we take $\epsilon = c\la^{-1}$ with $c >0$ sufficiently small, we can ensure that there is no cancellation in the integral in \eqref{eq:wave_mainmain} when $x \in S_{c\la^{-1}}$, and therefore
\[
|\mm_{\la,0,0}(x)| \gtrsim \la^{-(n-1)/2}
\]
for $x \in S_{c\la^{-1}}$ and $\la$ sufficiently large.
If we combine this with \eqref{eq:wave_zerosplitting} and \eqref{eq:wave_stationaryphase_estimates} (and choose $\ell,k,N$ sufficiently large), we obtain that
\[
|m_{\la,t_0}^{\chi,0}(\sqrt{\sLap}) \tilde g_\la(x)| \gtrsim \la^{n-(n-1)/2}
\]
for $x \in S_{c\la^{-1}}$ and $\la$ sufficiently large. On the other hand, $|S_{c\la^{-1}}| \sim \la^{-1}$, whence, for all $p \in [1,2]$,
\[
\|m_{\la,t_0}^{\chi,0}(\sqrt{\sLap}) \tilde g_\la(x) \|_p \gtrsim_p \la^{n-1/p-(n-1)/2},
\]
and combining this with \eqref{eq:w_upperbound_norm} we obtain that
\[
\|m_{\la,t_0}^{\chi,0}(\sqrt{\sLap})\|_{p\to p} \gtrsim_p \la^{n-1/p-(n-1)/2-n(1-1/p)} = \la^{(n-1)(1/p-1/2)},
\]
and we are done.
\end{proof}

\subsection{Transplantation}\label{sec10301346}

Finally we prove our main result in full generality.

\begin{proof}[Proof of Theorem \ref{thm:main}]
Let $(M,H,\mu)$ and $\sLap$ be as in Theorem \ref{thm:main}.
By the bracket-generating condition assumed on $H$ and Lemma~\ref{lem10101454}, there is a nonempty open set $M_o\subset M$ where $(M_o,H)$ is an equiregular sub-Riemannian manifold.
Up to shrinking $M_o$, there are $v_1,\dots,v_r \in \Sect(TM_o)$ such that $H=\sum_j v_j\otimes v_j$ on $M_o$, as in \eqref{eq10071358}, and $\langle v_j,v_k \rangle_H = \delta_{jk}$.

Fix $o\in M_o$. It is well known (see, e.g., \cite{MR806700,MR1334873,MR1421822,MR3353698} and references therein) that there is a coordinate system $(U,\phi)$ centred at $o$ and a system of dilations $\delta_\epsilon:\R^n\to\R^n$, $\epsilon>0$, of the form
\[
\delta_\epsilon(x_1,\dots,x_n) = \left(\epsilon^{w_1}x_1,\epsilon^{w_2}x_2,\dots,\epsilon^{w_n}x_n \right),
\]
with $1\le w_1\le w_2\le\dots\le w_n$ integers, such that, if $V \defeq \phi(U)\subset\R^n$, then the vector fields
\[
v_j^{(\epsilon)}|_x = \epsilon\, \DD\delta_\epsilon^{-1}[\DD\phi[v_j|_{\phi^{-1}(\delta_\epsilon x)}]] 
\]
defined on $\delta_\epsilon^{-1}V\subset\R^n$ converge as $\epsilon \to 0$ to some bracket generating vector fields $v_j^{(0)}$ on $\R^n$.
The convergence is uniform on compact sets in the $C^k$ norm, for all $k$.
Moreover, there is a Lie group structure on $\R^n$ which makes it into a Carnot group $G$, so that the vector fields $v_1^{(0)},\dots,v_r^{(0)}$  are left-invariant and form an orthogonal basis of the first layer.

From the above convergence result, it readily follows that the sub-Laplacian $\sLap_o=-\sum_{j=1}^r(v_j^{(0)})^2$ on the Carnot group $G$ is a local model of $\sLap$ at $o$, in the sense of \cite[Definition 5.1]{MR3671588}. Moreover, by Lemma~\ref{lem10111454}, the Carnot group $G$ and the sub-Laplacian $\sLap_o$ satisfy the assumptions \eqref{RE}, \eqref{FPS}, \eqref{SP} and \eqref{SFC}.

Suppose now that \eqref{eq:main_MH} holds for some $p \in [1,\infty]$ and $\alpha \geq 0$. Then, by \cite[Theorem 5.2]{MR3671588}, for all bounded Borel functions $m : [0,+\infty) \to \C$,
\[
\|m(\sLap_o)\|_{p \to p} 
\leq \liminf_{r\to 0^+} \|m(r^2\sLap)\|_{p \to p} 
\lesssim \liminf_{r \to 0^+} \|m(r^2\cdot)\|_{L^\infty_{\alpha,\sloc}} 
=  \|m\|_{L^\infty_{\alpha,\sloc}} .
\]
In other words, the estimate \eqref{eq:main_MH} also holds for the sub-Laplacian $\sLap_o$.
In view of \eqref{eq:mhnorm} and Theorem \ref{thm:MHlowerbd}, we conclude that $\alpha\geq n|1/2-1/p|$, and part \ref{en:main_MH} is proved.

As for part \ref{en:main_MP}, suppose that $p \in [1,\infty]$, $\alpha \geq 0$, $\chi \in C_c^\infty((0,\infty))$, and $\epsilon, R> 0$ are such that the estimate \eqref{eq:main_MP} holds. In view of \eqref{eq:adjointcalculus}, we may assume that $\chi$ is real-valued. If we set $\chi_e = \chi(|\cdot|)$, then 
$\chi_e \in \Sch_e$ and, in view of \eqref{eq:wave_mult}, the estimate \eqref{eq:main_MP} can be restated as
\[
\| m_{\la,t}^{\chi_e}(\sqrt{\sLap}) \|_{p \to p} \lesssim (\la t)^\alpha
\]
for all $\la,t>0$ with $t \leq \epsilon$ and $\la t \geq R$.
Hence, by \cite[Theorem 5.2]{MR3671588}, for all $\la,t>0$ with $\la t \geq R$,
\begin{multline*}
\|m_{\la,t}^{\chi_e}(\sqrt{\sLap_o})\|_{p\to p} 
\leq \liminf_{h\to 0^+} \|m_{\la,t}^{\chi_e}(h \sqrt{\sLap})\|_{p\to p}  \\
= \liminf_{h\to 0^+} \|m_{(\la t)/(th),th}^{\chi_e}(\sqrt{\sLap})\|_{p\to p} 
\lesssim (\la t)^\alpha.
\end{multline*}
By Theorem \ref{thm11011139} we deduce that $\alpha \geq (n-1) | 1/p - 1/2 | $.
\end{proof}

\providecommand{\bysame}{\leavevmode\hbox to3em{\hrulefill}\thinspace}

\end{document}